\documentclass[10.9pt,leqno]{amsart}
\usepackage{amsmath,amsthm,amssymb,mathtools,tocvsec2}
\usepackage[breaklinks=true]{hyperref}

\allowdisplaybreaks

\theoremstyle{plain}
\newtheorem{theorem}{Theorem}[section]

\newtheorem{corollary}[theorem]{Corollary}
\newtheorem{proposition}[theorem]{Proposition}
\newtheorem{lemma}[theorem]{Lemma}
\newtheorem{question}[theorem]{Question}

\theoremstyle{definition}
\newtheorem{definition}[theorem]{Definition}
\newtheorem{remark}[theorem]{Remark}
\newtheorem{example}[theorem]{Example}

\numberwithin{equation}{section}

\makeatletter
\let\c@theorem\c@table
\makeatother

\newcommand{\ba}{\mathbf{a}}
\newcommand{\bc}{\mathbf{c}}
\newcommand{\bm}{\mathbf{m}}
\newcommand{\bn}{\mathbf{n}}
\newcommand{\bu}{\mathbf{u}}
\newcommand{\bv}{\mathbf{v}}
\newcommand{\bx}{\mathbf{x}}
\newcommand{\bal}{\boldsymbol{\alpha}}

\newcommand\R{\mathbb{R}}
\newcommand\N{\mathbb{N}}
\newcommand\C{\mathbb{C}}
\newcommand\Z{\mathbb{Z}}
\newcommand\bp{\mathbb{P}}
\newcommand\nkneq{[n]^k_{\neq}}
\newcommand{\HTN}{HTN}

\begin{document}

\vspace*{-21mm}
\title[Sign patterns of positivity preservers in fixed dimension]{On the
sign patterns of entrywise\\ positivity preservers in fixed dimension}

\author{Apoorva Khare}
\address[A.~Khare]{Indian Institute of Science, Bangalore 560012, India;
and Analysis and Probability Research Group; Bangalore 560012, India}
\email{\tt khare@iisc.ac.in}

\author{Terence Tao}
\address[T.~Tao]{University of California at Los Angeles, CA 90095, USA}
\email{\tt tao@math.ucla.edu}




\thanks{Manuscript received December 16, 2018.}

\thanks{Research of the first author supported in part by 
Ramanujan Fellowship grant SB/S2/RJN-121/2017,
MATRICS grant MTR/2017/000295,
and SwarnaJayanti Fellowship grants SB/SJF/2019-20/14 and
DST/SJF/MS/2019/3 from SERB and DST (Govt.~of India),
by grant F.510/25/CAS-II/2018(SAP-I) from UGC (Govt.~of India),
and by a Young Investigator Award from the Infosys Foundation;
research of the second author supported by
NSF grant DMS-1266164
and by a Simons Investigator Award.}

\thanks{\hspace*{-8mm}\textit{American Journal of Mathematics} 143
(2021), 1863-1929.\ \copyright 2021 by Johns Hopkins University Press.}

\begin{abstract}
Given a domain $I \subset \mathbb{C}$ and an integer $N>0$, a function
$f: I \to \mathbb{C}$ is said to be \emph{entrywise positivity
preserving} on positive semidefinite $N \times N$ matrices $A = (a_{jk})
\in I^{N \times N}$, if the entrywise application $f[A] = (f(a_{jk}))$ of
$f$ to $A$ is positive semidefinite for all such $A$. Such preservers in
all dimensions have been classified by Schoenberg and Rudin as being
absolutely monotonic [\textit{Duke Math.\ J.}\ 1942, 1959]. In fixed
dimension $N$, results akin to work of Horn and Loewner [\textit{Trans.\
Amer.\ Math.\ Soc.}\ 1969] show that the first $N$ non-zero Maclaurin
coefficients of any positivity preserver $f$ are positive; and the last
$N$ coefficients are also positive if $I$ is unbounded. However, very
little was known about the higher-order coefficients: the only examples
to date for unbounded domains $I$ were absolutely monotonic, hence work
in all dimensions; and for bounded $I$ examples of non-absolutely
monotonic preservers were very few (and recent).

In this paper, we provide a complete characterization of the sign
patterns of the higher-order Maclaurin coefficients of positivity
preservers in fixed dimension $N$, over bounded and unbounded domains $I
= (0,\rho)$. In particular, this shows that the above Horn--Loewner-type
conditions cannot be improved upon. As a further special case, this
provides the first examples of polynomials which preserve positivity on
positive semidefinite matrices in $I^{N \times N}$ but not in $I^{(N+1)
\times (N+1)}$. Our main tools in this regard are the Cauchy--Binet
formula and lower and upper bounds on Schur polynomials. We also obtain
analogous results for real exponents, using the
Harish-Chandra--Itzykson--Zuber formula in place of bounds on Schur
polynomials.

We then go from qualitative existence bounds -- which suffice to
understand all possible sign patterns -- to exact quantitative bounds.
This is achieved using a Schur positivity result due to Lam, Postnikov,
and Pylyavskyy [\textit{Amer.\ J.\ Math.}\ 2007], and in particular
provides a second proof of the existence of threshold bounds for tuples
of integer and real powers.
As an application, we extend our previous qualitative and quantitative
results to understand preservers of total non-negativity in fixed
dimension -- including their sign patterns.
We deduce several further applications, including extending a Schur
polynomial conjecture of Cuttler, Greene, and Skandera [\textit{Eur.\ J.\
Comb.}\ 2011] to obtain a novel characterization of weak majorization for
real tuples.
\end{abstract}

\maketitle

\settocdepth{section}
\tableofcontents

\section{Introduction and main results}

\subsection{Notation and prior results}

For any natural number $N$, let $\bp_N(\C)$ denote the convex cone of
positive semidefinite Hermitian $N \times N$ matrices; this defines a
partial ordering $\preceq$ on $N \times N$ matrices, with $A \preceq B$
if and only if $B-A$ is positive semidefinite.

Given a domain $I \subset \C$, let $\bp_N(I) \subset \bp_N(\C)$ denote
the set of matrices $A = (a_{jk})_{j,k=1,\dots,N} \in \bp_N(\C)$ with all
entries $a_{jk}$ in $I$, thus for instance $\bp_N(\R)$ is the cone of
positive semidefinite real symmetric matrices.  A key role in this
paper will also be played by the subset $\bp^1_N(I) \subset \bp_N(I)$ of
\emph{rank one} matrices $\bu \bu^*$ in $\bp_N(I)$.  We
will focus our attention in this paper almost entirely on the cases $I =
(0,\rho)$ for $0 < \rho \leqslant +\infty$, although we will also briefly
consider the case $I = (-\rho,\rho)$, as well as the complex disk $I =
D(0,\rho)$.

\begin{remark}
In this paper, all our vectors $\bu = (u_1,\dots,u_N)^T$ will be column
vectors, with the space of such vectors denoted as $(\C^N)^T$; row
vectors $(u_1,\dots,u_N)$ will be referred to instead as \emph{tuples},
and the space of such tuples denoted as $\C^N$.
A more complete list of notations used in this paper is provided in the
final section.
\end{remark}

Given a matrix $A = (a_{jk})_{j,k=1,\dots,N}$ in $\bp_N(I)$, a function
$f : I \to \C$ acts \textit{entrywise} on $A$ via the formula
\[
f[A] \coloneqq (f(a_{jk}))_{j,k=1,\dots,N}.
\]

\begin{remark}
Note that the entrywise application $f[A]$ of $f$ to $A$ should not be
confused with the more common functional calculus $f(A)$ of $f$ applied
to $A$; we will not use the latter in this paper.
\end{remark}

For instance, if $f$ is a monomial $f(x) = x^m$, then
$f[A] = A^{\circ m}$ is the Hadamard product of $m$ copies of $A$. We say
that the function $f: I \to \C$ is \emph{entrywise positivity preserving}
on $\bp_N(I)$ if $f[A] \in \bp_N(\C)$ for all $A \in \bp_N(I)$; similarly
if $\bp_N(I)$ is replaced with any subset of $\bp_N(I)$, such as
$\bp^1_N(I)$.

The Schur product theorem~\cite{Schur1911} asserts that if two matrices
$A,B$ lie in~$\bp_N(\C)$, then so does their Hadamard product $A \circ B$.
As observed in~1925 by P\'olya and Szeg\"o~\cite[Problem
37]{polya-szego}, this immediately implies that any function $f: I \to
\C$ which is \emph{absolutely monotonic}, in the sense that one has a
convergent power series representation $$ f(x) = \sum_{k \geqslant 0} c_k
x^k$$ on $I$ for some non-negative coefficients $c_k \geqslant 0$, will
be entrywise positivity preserving on $\bp_N(I)$ for any $N$.

It is then natural to ask which of the positivity conditions $c_k
\geqslant 0$ are in fact necessary.  More precisely, in this paper we
address the following question:

\begin{question}\label{Q1}
Fix a positive integer $N$ and a set $I \subset \C$, and consider a
convergent power series $f : I \to \C$ which is entrywise positivity
preserving on $\bp_N(I)$. Which coefficients of $f$ can be negative?
\end{question}  

In fact we completely resolve this question in the case $I = (0,\rho)$
for any $0 < \rho \leqslant +\infty$. Additionally, we
completely answer a variant of Question~\ref{Q1} for real powers; and
give some partial results in the cases $I = (-\rho,\rho)$ and $I =
D(0,\rho)$, where (as we explain) there is no ``uniform'' answer to the
question.

Question~\ref{Q1} arises out of a longstanding program in analysis over
the past century.  In \textit{loc.~cit.},
P\'olya and Szeg\"o asked if there are functions $f: \R \to \R$ besides
the convergent absolutely monotonic functions which were entrywise
positivity preserving on $\bp_N(I)$ for all $N$ and $I \subset \R$.  In
his celebrated work~\cite{Schoenberg42}, Schoenberg in~1942 proved this
was not possible for continuous~$f$ (even if one restricted to the case
$I = (-1,1)$), using positive definite functions on spheres (Gegenbauer
polynomials).  Schoenberg was interested in embedding positive definite
metrics into Hilbert space; see also~\cite{Bochner-pd,
vonNeumann-Schoenberg}. The continuity hypothesis in Schoenberg's theorem
was later removed by Rudin~(1959) in~\cite{Rudin59}, using analysis of
measures on the torus, and working in the broader context of studying
functions acting on Fourier--Stieltjes transforms, as explored with
Kahane and others in~\cite{HKKR,Kahane-Rudin}.  In fact, Rudin's result
only required positivity preservation of Toeplitz matrices in
$\bp_N((-1,1))$ of rank at most three, which correspond to measures on
the torus by Herglotz's theorem; a parallel result for Hankel matrices
(which correspond to measures on the real line) was shown
in~\cite{BGKP-hankel}.  In a sense, Schoenberg's result is the (far
harder) converse to that of his advisor, Schur.  For variants of
Schoenberg's theorem for other choices of $I$, see \cite{horn},
\cite{vasudeva79}, \cite{GKR-lowrank}.

Since these results of Schoenberg and Rudin, the question of classifying
the entrywise positivity preservers for a \emph{fixed} dimension $N$ has
been actively studied.  Necessary conditions for entrywise positivity
preservation were first established in the 1969 thesis of
Horn~\cite{horn}, who attributes the result to Loewner.  We summarize
these conditions, as well as some further necessary conditions of
Horn--Loewner type established by subsequent authors, as follows:

\begin{lemma}[Horn--Loewner-type necessary
conditions]\label{horn-type}
Let $N \geqslant 2$ and $0 < \rho \leqslant +\infty$.
\begin{itemize}
\item[(i)] (Horn and Loewner \cite{horn};
Guillot--Khare--Rajaratnam \cite{GKR-lowrank})
Suppose that $f: (0,\rho) \to \R$ is entrywise positivity preserving on
all matrices in $\bp_N((0,\rho))$ of the form $A = a {\bf 1}_{N \times N}
+ \bu \bu^T$, with $a \in [0,\rho), \bu \in [0,\sqrt{\rho-a})^N$. Then $f
\in C^{N-3}((0,\rho))$,
\[
f^{(k)}(x) \geqslant 0, \qquad \forall x \in (0,\rho),\ 0 \leqslant k
\leqslant N-3,
\]

\noindent and $f^{(N-3)}$ is a convex non-decreasing function on
$(0,\rho)$. In particular, if $f \in C^{N-1}((0,\rho))$, then $f^{(k)}(x)
\geqslant 0$ for all $x \in (0,\rho), 0 \leqslant k \leqslant N-1$.

\item[(ii)] (See \cite[Lemma 2.4]{BGKP-fixeddim})  If $f(x) =
\sum_{n=0}^\infty c_n x^n$ is a convergent power series on $(0,\rho)$
that is entrywise positivity preserving on $\bp^1_N((0,\rho))$, and
$c_{n_0} < 0$ for some $n_0$, then we have $c_n > 0$ for at least $N$
values of $n < n_0$.  (In particular, the first $N$ non-zero Taylor
coefficients of~$f$, if they exist, must be positive.)

\item[(iii)]  (See Section~\ref{unbounded-proof}) If $f(x) =
\sum_{n=0}^\infty c_n x^n$ is a convergent power series on $(0,+\infty)$
which is entrywise positivity preserving on $\bp^1_N((0,+\infty))$, and
$c_{n_0} < 0$ for some $n_0$, then we have $c_n > 0$ for at least $N$
values of $n < n_0$, and at least $N$ values of $n > n_0$. (In
particular, if $f$ is a polynomial, then the first $N$ non-zero
coefficients and the last $N$ non-zero coefficients of $f$, if they
exist, are all positive.)
\end{itemize}
\end{lemma}

We make two remarks here. First, the original result of Horn and Loewner
required $f$ to be continuous; this assumption was removed
in~\cite{GKR-lowrank}, in the spirit of Rudin's
strengthening~\cite{Rudin59} of Schoenberg's theorem~\cite{Schoenberg42}
alluded to above. Second, the proof of Lemma~\ref{horn-type}(ii) uses the
positivity property
\begin{equation}\label{Egantmacher}
 \det( u_i^{\alpha_j} )_{i,j=1,\dots,N} > 0
\end{equation}
of generalized Vandermonde determinants for any $0 < u_1 < \dots < u_N$
and $\alpha_1 < \dots < \alpha_N$; see e.g.~\cite[Chapter~XIII, \S8,
Example~1]{Gantmacher_Vol2}, or the bounds in~\eqref{vbal} below.
Variations of this positivity property will recur throughout this paper.

In a slightly different direction, it was shown by FitzGerald--Horn
in~\cite{FitzHorn} (solving a conjecture of Horn~\cite{horn}) that the
fractional monomials $x \mapsto x^\alpha$ are entrywise positivity
preservers on $\bp_N((0,+\infty))$ if and only if $\alpha$ is a
non-negative integer, or a real number greater than $N-2$.
(Note this shows that Lemma~\ref{horn-type}(i) is sharp.)  See the recent
survey~\cite{BGKP-survey1,BGKP-survey2} for further results and
references on entrywise positivity preservers, as well as for more on
Schoenberg and Rudin's motivations in proving their results mentioned
above.  However, in spite of significant subsequent interest and
activity, a complete characterization of the functions -- even for
polynomials -- that entrywise preserve positivity on $\bp_N((0,+\infty))$
remains unknown even for~$N=3$.
(For $N=2$ the problem was resolved by Vasudeva~\cite{vasudeva79}.)

In light of the above discussion, it is natural to ask if for real
analytic preservers $f$, the positive coefficient requirements in
Lemma~\ref{horn-type}(ii) and Lemma~\ref{horn-type}(iii) are sharp.
In~\cite{BGKP-fixeddim}, Schur polynomials were used to establish a
necessary and sufficient condition for entrywise positivity preservation
on $\bp_N((0,\rho))$, $0 < \rho < +\infty$ for polynomials of the form
\[
x \mapsto c_0 + c_1 x + \dots + c_{N-1} x^{N-1} + c_M x^M
\]
with $M \geqslant N$; in particular, it was shown that for any choice of
$M$, one could construct entrywise positivity preserving polynomials with
$c_M$ negative (of course, Lemma~\ref{horn-type}(ii) forces the remaining
coefficients $c_0,\dots,c_{N-1}$ to then be positive).
Via the Schur product theorem, this implies a similar result for
polynomials of the form
\[
x \mapsto c_0 x^h + c_1 x^{h+1} + \cdots + c_{N-1}
x^{h+N-1} + c_M x^{h+M}, \qquad h \in \Z^{\geqslant 0}.
\]
In the $N=2$ case, a similar analysis was also carried out in \cite[\S
3.4]{BGKP-fixeddim} for polynomials of the form 
\[
x \mapsto c_m x^m + c_n x^n + c_p x^p
\]
with $m < n < p$, where again it was shown that for any choice of
$m,n,p$, one could construct such a polynomial with $c_p$ negative but
which was still entrywise positivity preserving on 
$\bp_2((0,\rho))$.

However, aside from these few results (and linear combinations of them),
there were no examples previously known of entrywise positivity
preserving convergent power series with at least one negative
coefficient. In particular, with the exceptions discussed above, all
previously known entrywise positivity preservers on $\bp_N((0,\rho))$
were absolutely monotonic, hence in fact work for all dimensions. For the
unbounded domain $\rho = +\infty$, there was even less progress, with no
examples of preservers with negative coefficients known to date (nor if
such functions could even exist).

\subsection{New results 1: Qualitative bounds}

We begin with the simple observation that Question~\ref{Q1} can have a
``structured'' solution (in the flavor of Lemma~\ref{horn-type}) only for
$I \subset [0,+\infty)$, but not other domains $I = (-\rho,\rho)$ or
$D(0,\rho)$ in the complex plane. For example, the family of polynomials
\[
p_{k,t}(x) \coloneqq t(1 + x^2 + \cdots + x^{2k}) - x^{2k+1}, \qquad k
\geqslant 0, \ t > 0,
\]
can never preserve positivity on $\bp_2((-\rho,\rho))$, since setting
e.g.~$\bu \coloneqq (1,-1)^T$ and $A \coloneqq (\rho/2) \bu \bu^T \in
\bp_2((-\rho,\rho))$, one computes:
\begin{equation}
\bu^T p_{k,t}[A] \bu = -4 (\rho/2)^{2k+1} < 0,
\end{equation}
whence $p_{k,t}[A]$ is not positive semidefinite for any $k \geqslant 0$.
Similar examples with higher-order roots of unity (fail to) work in the
case of complex domains.

Thus the present work is primarily concerned with bounded and unbounded
domains $I \subset (0,+\infty)$. In the case of bounded intervals $I =
(0,\rho)$, we completely resolve Question~\ref{Q1} by showing that the
non-zero coefficients beyond the first $N$ of an entrywise positivity
preserver on $\bp_N((0,\rho))$ are allowed to be of arbitrary sign:

\begin{theorem}\label{main-1}
Let $N > 0$ and $0 \leqslant n_0 < n_1 < \cdots < n_{N-1}$ be integers,
and for each $M > n_{N-1}$, let $\epsilon_M \in \{-1,0,+1\}$ be a sign.
Let $0 < \rho < +\infty$, and let $c_{n_0},\dots,c_{n_{N-1}}$ be positive
reals.  Then there exists a convergent power series
\[
f(x) = c_{n_0} x^{n_0} + c_{n_1} x^{n_1} + \dots + c_{n_{N-1}}
x^{n_{N-1}} + \sum_{M > n_{N-1}} c_M x^M
\]
on $(0,\rho)$ that is an entrywise positivity preserver on
$\bp_N((0,\rho))$, such that for each $M > n_{N-1}$, $c_M$ has the sign
of $\epsilon_M$.
\end{theorem}

In particular, Theorem~\ref{main-1} shows that the Horn--Loewner-type
necessary criterion in Lemma~\ref{horn-type}(ii) cannot be improved upon.
Note from a limiting argument that we may replace $(0,\rho)$ here by
$[0,\rho]$, and hence by any subset of $[0,\rho]$, if desired.

Theorem~\ref{main-1} follows readily from the following special case:

\begin{theorem}\label{Tmain}
Let $N > 0$ and $0 \leqslant n_0 < \dots < n_{N-1} < M$ be integers.  Let
$0 < \rho < +\infty$, and let $c_{n_0},\dots,c_{n_{N-1}}$ be positive
reals.  Then there exists a negative number $c_M$ such that
\begin{equation}\label{cpoly}
 x \mapsto c_{n_0} x^{n_0} + c_{n_1} x^{n_1} + \dots + c_{n_{N-1}}
 x^{n_{N-1}} + c_M x^M
\end{equation}
entrywise preserves positivity on $\bp_N((0,\rho))$.
\end{theorem}

Indeed, to derive Theorem~\ref{main-1} from Theorem~\ref{Tmain}, we see
(since the space of entrywise positivity preserving functions forms a
cone, and because any monomial is entrywise positivity preserving thanks
to the Schur product theorem) that for any $M > n_{N-1}$, there exists
$\delta_M > 0$ such that the polynomial~\eqref{cpoly} is entrywise
positivity preserving whenever $|c_M| \leqslant \delta_M$; by shrinking
$\delta_M$ if necessary, we may assume that $\delta_M \leqslant
\frac{1}{M!}$ (say) for all $M$.  Multiplying~\eqref{cpoly} (with $c_M$
replaced by $\epsilon_M \delta_M$) by $2^{n_{N-1}-M}$ and summing over
all $M > n_{N-1}$, we obtain Theorem~\ref{main-1} with $c_M \coloneqq
2^{n_{N-1}-M} \epsilon_M \delta_M$.

Theorem~\ref{Tmain} can be reformulated as a matrix inequality: for any
$0 \leqslant n_0 < \dots < n_{N-1} < M$, $0 < \rho < +\infty$, and
$c_{n_0},\dots,c_{n_{N-1}} > 0$, there exists a finite threshold
${\mathcal K}$ (depending on $n_0,\dots,n_{N-1}, \rho, c_{n_0}, \dots,
c_{n_{N-1}}, M$) such that
\begin{equation}\label{am}
 A^{\circ M} \preceq {\mathcal K} \sum_{j=0}^{N-1} c_{n_j} A^{\circ n_j} 
\end{equation}
for any $A \in \bp_N((0,\rho))$.  The quantity ${\mathcal K}$ provided by
the argument will be explicit (see~\eqref{k1def}) but not completely
optimal; the optimal threshold is given in Theorems~\ref{Treal-rank1},
\ref{Treal-rank2} below, and established in Section~\ref{S8}.

The bounds in Theorem~\ref{Tmain} will be sufficiently strong that we can
replace the monomials $x^M$ in~\eqref{am} with arbitrary convergent power
series:

\begin{corollary}[Analytic functions]\label{Tanalytic}
Fix integers $N>0$ and $0 \leqslant n_0 < \cdots <
n_{N-1}$, and a polynomial $c_{n_0} x^{n_0} + \cdots +
c_{n_{N-1}} x^{n_{N-1}}$, with $c_{n_j} > 0\ \forall j$.
Let $0 < \rho < +\infty$.
Given a power series $g(x) = \sum_{M > n_{N-1}} g_M x^M$ which is
convergent at $\rho$, there exists a finite threshold $\mathcal{K} =
\mathcal{K}(n_0,\dots,n_{N-1}, \rho, c_{n_0}, \dots, c_{n_{N-1}}, g)$
such that the function
\[
x \mapsto \mathcal{K} \sum_{j=0}^{N-1} c_{n_j} x^{n_j} - g(x)
\]
is entrywise positivity preserving on $\bp_N((0,\rho))$.  Equivalently,
one has
\begin{equation}\label{amp}
 g[A] \preceq {\mathcal K} \sum_{j=0}^{N-1} c_{n_j}
 A^{\circ n_j}
\end{equation}
for all $A \in \bp_N((0,\rho))$.
\end{corollary}

We establish this result in Section~\ref{Tan-sec}.  It should be possible
to relax the requirement that $g$ be a convergent power series to the
hypothesis that $g$ is in the regularity class $C^M([0,\rho])$ for some
sufficiently large $M$, but we will not attempt to do so here.

\begin{remark}  
If one specializes~\eqref{am} to the rank one matrix $A = \bu \bu^T$ with
$\bu = (u_1,\dots,u_N)^T$ and $0 < u_1 < \cdots < u_N$, we conclude in
particular that the vectors $(u_1^{n_j}, \dots, u_N^{n_j})^T$ for
$j=1,\dots,N$ are linearly independent, which is
essentially~\eqref{Egantmacher} (in the case of non-negative integer
exponents).  One may thus view Theorem~\ref{Tmain} as a ``robust''
variant of~\eqref{Egantmacher}.
\end{remark}

Coming to the unbounded domain case $I = (0,+\infty)$, we
once again completely resolve Question~\ref{Q1}. Just as
Theorem~\ref{main-1} demonstrates the sharpness of
Lemma~\ref{horn-type}(ii), our second main result demonstrates the
sharpness of Lemma~\ref{horn-type}(iii):

\begin{theorem}\label{main-2}
Let $N > 0$ and $0 \leqslant n_0 < \dots < n_{N-1}$ be integers, and for
each $M > n_{N-1}$, let $\epsilon_M \in \{-1,0,+1\}$ be a sign. Suppose
that whenever $\epsilon_{M_0} = -1$ for some $M_0 > n_{N-1}$, one has
$\epsilon_M = +1$ for at least $N$ choices of $M > M_0$. Let
$c_{n_0},\dots,c_{n_{N-1}}$ be positive reals.  Then there exists a
convergent power series
\[
f(x) = c_{n_0} x^{n_0} + c_{n_1} x^{n_1} + \dots + c_{n_{N-1}}
x^{n_{N-1}} + \sum_{M > n_{N-1}} c_M x^M
\]
on $(0,+\infty)$ that is an entrywise positivity preserver on
$\bp_N((0,+\infty))$, such that for each $M > n_{N-1}$, $c_M$ has the
sign of $\epsilon_M$.
\end{theorem}

Unlike the setting of bounded $I$, this is also the first existence
result for power series preservers of $\bp_N(I)$ with negative
coefficients.

Like the setting of bounded $I$, Theorem~\ref{main-2} is a consequence of
the following special case:

\begin{theorem}\label{Tunbdd}
Let $N > 0$ and $0 \leqslant n_0 < \dots < n_{N-1} < M < n_N < \dots <
n_{2N-1}$ be integers, and let $c_{n_0},\dots,c_{n_{2N-1}}$ be positive
reals.  Then there exists a negative number $c_M$ such that
\begin{equation}\label{cpoly2}
 x \mapsto c_{n_0} x^{n_0} + c_{n_1} x^{n_1} + \dots + c_{n_{N-1}}
 x^{n_{N-1}} + c_M x^M + c_{n_N} x^{n_N} + \dots + c_{n_{2N-1}}
 x^{n_{2N-1}}
\end{equation}
entrywise preserves positivity on $\bp_N((0,+\infty))$.
\end{theorem}

Indeed, if $N, n_0,\dots,n_{N-1}, (\epsilon_M)_{M>n_{N-1}}$ are as in
Theorem~\ref{main-2}, then from Theorem~\ref{Tunbdd}, one may find for
each $M > n_{n-1}$ with $\epsilon_M = -1$, a real number $0 < \delta_M
\leqslant \frac{1}{M!}$ such that
\[
f_M : x \mapsto c_{n_0} x^{n_0} + c_{n_1} x^{n_1} + \dots
+ c_{n_{N-1}} x^{n_{N-1}} - \delta_M x^M + \sum_{n > M: \epsilon_n = +1}
\frac{1}{n!} x^n
\]
entrywise preserves positivity on $\bp_N((0,+\infty))$.  
For all other powers $M > n_{N-1}$ with $\epsilon_M \neq -1$, define
$f_M(x) \coloneqq \sum_{j=0}^{N-1} c_{n_j} x^{n_j} +
\frac{\epsilon_M}{M!} x^M$.
Now Theorem~\ref{main-2} follows by considering $f(x) \coloneqq \sum_{M >
n_{N-1}} 2^{n_{N-1} - M} f_M(x)$; note here that $|f(x)| \leqslant
\sum_{j=0}^{N-1} c_{n_j} x^{n_j} + (x+1) e^x$ for $x>0$.

We prove Theorem~\ref{Tunbdd} in Section~\ref{unbounded-proof}.  As one
corollary of this theorem (and Lemma~\ref{horn-type}(iii)), we see that
for any $N$, there exist analytic functions that entrywise preserve
positivity on $\bp_N((0,+\infty))$ but not on $\bp_{N+1}((0,+\infty))$.

We are also able to establish analogues of the above theorems in which
the exponents $n_j, M$ are real numbers rather than natural numbers; see
Section~\ref{real-sec}. This allows us to answer Question~\ref{Q1} for
real powers, thus  replacing power series by countable sums of powers,
including but not restricted to Hahn and Puiseux series.  Similarly, we
obtain an analogue of Corollary~\ref{Tanalytic} in which the analytic
function $g$ is replaced by a Laplace transform of more general real
measures with support in $(n_{N-1},\infty)$.

On the other hand, if one replaces the domain $(0,\rho)$ with a two-sided
domain $(-\rho,\rho)$ or with a complex disk $D(0,\rho)$, then the
results largely break down for all tuples $\bn \coloneqq (n_0, \dots,
n_{N-1})$ that do not equal shifts by $h \in \Z^{\geqslant 0}$ of the
``minimal'' tuple $(0,\dots, N-1)$; see Sections~\ref{twoside},
\ref{complex}. As the results for tuples of the form $\bn =
(h,h+1,\dots,h+N-1)$ were uniformly valid over $I = D(0,\rho)$
(see~\cite{BGKP-fixeddim}), it follows that the problem for every other
$\bn$ is more challenging, and new techniques are required to resolve
Question~\ref{Q1}.

Our proof strategy is as follows.  We first focus on establishing
entrywise positivity preservation for rank one matrices $\bu \bu^T$.  In
this case, one can use the Cauchy--Binet formula to obtain an explicit
criterion for positive definiteness, in terms of generalized Vandermonde
determinants.  In the case of natural number exponents, these
determinants can be factored as the product of the ordinary Vandermonde
determinant and a Schur polynomial.  One can then use the totally
positive nature of Schur polynomials to obtain satisfactory upper and
lower bounds on these polynomials (relying crucially on the fact that we
are restricting the entries of the rank one matrix to be non-negative).
The main novelty in our arguments, compared to previous work, is the use
of \emph{lower} bounds on Schur polynomials, which are needed due to the
presence of such polynomials in the denominators of the formulae for
various thresholds whenever $\bn \neq (h,h+1,\dots,h+N-1)$ for $h \in
\Z^{\geqslant 0}$.

Once entrywise positivity preservation is shown for rank one matrices, we
induct on $N$ using an argument of FitzGerald and Horn~\cite{FitzHorn},
relying on the observation that any positive definite matrix can be
viewed as the sum of a rank one matrix and a matrix with vanishing final
row and column, allowing one to derive entrywise positivity preservation
for general positive definite matrices from the rank one case and the
induction hypothesis using the fundamental theorem of calculus.

In the case of real exponents, the same argument as above is used to
extend the threshold from rank-one matrices to all matrices. To produce a
threshold in the rank-one case, Schur polynomials are no longer available
to control generalized Vandermonde determinants, but we can use the
famous Harish-Chandra--Itzykson--Zuber formula~\cite{hc,iz} as a
substitute for obtaining the corresponding upper bound.
For the lower bound, we refine this analysis using Gelfand--Tsetlin
polytopes.
These effective lower and upper bounds allow us to
answer Question~\ref{Q1} for real powers, and also to
extend Corollary~\ref{Tanalytic} to Laplace transforms.
The bounds are also applied later in the paper, to prove a new
characterization of weak majorization (see Theorem~\ref{Tcgs}).
It is remarkable that not only Schur polynomials, but also the
Harish-Chandra--Itzykson--Zuber unitary integral,
Gelfand--Tsetlin patterns, and Schur positivity (below) --
all of which play a central role in our proofs -- arise naturally in type
$A$ representation theory.

\subsection{New results 2: Exact quantitative bounds and
applications}\label{Squant}

We now produce sharper bounds. As our chief purpose in the previously
stated results was to solve Question~\ref{Q1}, it sufficed to use lower
and upper bounds on Schur polynomials to obtain threshold bounds. We
will show the following exact result for rank-one matrices.

\begin{theorem}\label{Treal-rank1}
Fix an integer $N>0$ and real powers $n_0 < \cdots < n_{N-1} < M$. Also
fix real scalars $\rho > 0$ and $c_{n_0}, \dots, c_{n_{N-1}}, c'$, and
define
\begin{equation}
f(x) \coloneqq \sum_{j=0}^{N-1} c_{n_j} x^{n_j} + c' x^M.
\end{equation}

\noindent Then the following are equivalent:
\begin{enumerate}
\item The entrywise map $f[-]$ preserves \thinspace
positivity \thinspace on \thinspace rank-one \thinspace matrices
\thinspace in $\bp_N((0,\rho))$.

\item Either all $c_{n_j}, c' \geqslant 0$; or $c_{n_j} > 0\ \forall j$
and $c' \geqslant -\mathcal{C}^{-1}$, where
\begin{equation}\label{Esharp}
\mathcal{C} = \sum_{j=0}^{N-1} \frac{V(\bn_j)^2}{V(\bn)^2}
\frac{\rho^{M - n_j}}{c_{n_j}}.
\end{equation}

\noindent Here $\bn \coloneqq (n_0, \dots, n_{N-1})$, $\bn_j \coloneqq
(n_0, \dots, n_{j-1}, n_{j+1}, \dots, n_{N-1}, M)$, and given a tuple
$(t_0, \dots, t_{k-1})$ or a vector ${\bf t} = (t_0, \dots, t_{k-1})^T$,
we define its ``Vandermonde determinant'' to be
\[
V(t_0, \dots, t_{k-1}) = V({\bf t}) \coloneqq \prod_{0 \leqslant i < j
\leqslant k-1} (t_j - t_i).
\]
\end{enumerate}
\end{theorem}

Notice in this case that the powers $n_j, M$ are allowed to be negative
as well.

Our next result proves that the sharp threshold~\eqref{Esharp} works for
matrices of all ranks. There is a small subtlety about which powers $n_j$
are allowed if the rank of the matrices is greater than one; see the
remarks after Theorem~\ref{Treal-1} below.

\begin{theorem}\label{Treal-rank2}
With notation as in Theorem~\ref{Treal-rank1}, if we further assume that
$n_j \in \Z^{\geqslant 0} \cup [N-2,\infty)$ for all $j$, then the two
conditions (1), (2) are further equivalent to:
\begin{enumerate}
\setcounter{enumi}{2}
\item The entrywise map $f[-]$ preserves positivity on $\bp_N([0,\rho])$.
\end{enumerate}
\end{theorem}

The proof of these theorems involves refining the approach to prove the
aforementioned results. The key additional tool is a Schur positivity
result by Lam--Postnikov--Pylyavskyy~\cite{LPP}, which implies the
following monotonicity property for ratios of Schur polynomials
$s_\bm(\bu) / s_\bn(\bu)$:

\begin{proposition}
Fix tuples of non-negative integers $0 \leqslant n_0 < \cdots < n_{N-1}$
and $0 \leqslant m_0 < \cdots < m_{N-1}$, such that $n_j \leqslant m_j\
\forall j$. Then the function
\[
f : ((0,+\infty)^N)^T \to \R, \qquad
f(\bu) \coloneqq \frac{s_\bm(\bu)}{s_\bn(\bu)}
\]
is non-decreasing in each coordinate.
\end{proposition}

While $\bm, \bn$ are integer tuples in this result, it
helps prove the above theorems for all real powers. We provide
below in the paper several application of this analysis;
we mention here two of them. First, we extend all of the
previous results on positivity preservers to preservers of \textit{total
non-negativity} on Hankel matrices of a fixed dimension.
(Recall that a possibly non-square real matrix is totally non-negative --
sometimes termed totally positive -- if it has all non-negative real
minors~\cite{karlin}.)
Note that the constraint of having non-negative entries is natural also
in total non-negativity, as in the above results.

Second, we show that a conjecture by
Cuttler--Greene--Skandera~\cite{CGS}, recently proved by Sra~\cite{Sra}
and Ait-Haddou and Mazure~\cite{Blossom2}, can be extended to obtain a
characterization of weak majorization for all non-negative real tuples,
which involves Schur polynomials/generalized Vandermonde determinants,
and which we believe is new:

\begin{theorem}\label{Tcgs}
Fix $N$-tuples $\bm, \bn$ of pairwise distinct real powers.
Then the following are equivalent.
\begin{enumerate}
\item For all vectors $\bu \in ([1,+\infty)^N)^T$, we have:
\begin{equation}\label{Ecgssra}
\frac{|\det (\bu^{\circ m_0} | \dots | \bu^{\circ m_{N-1}})|}{|V(\bm)|}
\geqslant
\frac{|\det (\bu^{\circ n_0} | \dots | \bu^{\circ n_{N-1}})|}{|V(\bn)|}.
\end{equation}

\item The tuple $\bm$ weakly majorizes $\bn$.
\end{enumerate}
\end{theorem}

As we will show, the assertion~(2) is implied by~(1) holding for vectors
$\bu$ in the smaller set $((I_\infty)^N)^T$ -- in fact a certain
countable subset of this suffices -- where $I_\infty$ is any non-empty
deleted neighborhood in $[1,+\infty)$ of $+\infty$.

In this vein, we further strengthen in two ways the aforementioned recent
works by Cuttler--Greene--Skandera, Sra, and Ait-Haddou--Mazure, which
characterizes majorization. Namely, these authors showed -- for integer
powers $\bm, \bn$ -- that~\eqref{Ecgssra} holds now for all $\bu \in
([0,+\infty)^N)^T$, if and only if the integer tuple $\bm$ majorizes
$\bn$. (In fact, these authors related this inequality at all $\bu \in
([0,+\infty))^N$ to $\bm - \bn_{\min}$ majorizing $\bn - \bn_{\min}$; but
after arranging both $\bm, \bn$ in increasing order, note that such a
(weak) majorization is indeed equivalent to $\bm$ (weakly)
majorizing~$\bn$.)

In the Cuttler--Greene--Skandera result, by continuity it suffices to
assume~\eqref{Ecgssra} holds on the positive open orthant
$((0,+\infty)^N)^T$ instead of its closure. Our first strengthening is
that the aforementioned characterization of majorization also holds for
real powers, not just integer ones.

\begin{theorem}\label{Tcgs-maj1}
Fix $N$-tuples $\bm, \bn$ of pairwise distinct real powers.
Then the following are equivalent.
\begin{enumerate}
\item The inequality~\eqref{Ecgssra} holds, now for all vectors $\bu \in
((0,+\infty)^N)^T$.

\item The tuple $\bm$ majorizes $\bn$.
\end{enumerate}
\end{theorem}

Second, even in the original case of non-negative integer tuples $\bm,
\bn$ (but also in full generality -- for arbitrary tuples $\bm,\bn$ of
pairwise distinct real powers), we show that one does not require the
above inequality at all points in the orthant, but only on the open unit
cube and on the subset $((1,+\infty)^N)^T$ in our ``weak majorization''
result above:

\begin{theorem}\label{Tcgs-maj2}
Notation as in Theorem~\ref{Tcgs-maj1}. Then the two assertions are
further equivalent to:
\begin{enumerate}
\setcounter{enumi}{2}
\item The inequality~\eqref{Ecgssra} holds for the ``restricted'' set of
vectors $\bu \in ((0,1)^N)^T \cup ((1,+\infty)^N)^T$.
\end{enumerate}
In fact, here one only needs to work with a certain countable set of
vectors $\bu$ in $((I_0)^N)^T \cup ((I_\infty)^N)^T$, where
$I_\infty$ was defined immediately after Theorem~\ref{Tcgs},
and $I_0 \subset (0,1]$ similarly denotes an arbitrary non-empty deleted
neighborhood in $(0,1]$ of $0$.
\end{theorem}

In the final section, we explain how to further extend (a part of)
Theorem~\ref{Tcgs}, as well as the ``positivity'' part of the result of
Lam--Postnikov--Pylyavskyy, to ``continuous'' versions of Schur
polynomials -- i.e., generalized Vandermonde determinants. This follows
from a more general log-supermodularity phenomenon for strictly totally
positive matrices, which follows from the work of
Skandera~\cite{skandera}.

\medskip

\textit{Acknowledgment.}
We thank the referee(s) for a careful reading of the manuscript.


\section{Preliminaries on Schur polynomials}\label{Sschur}

As the proofs of the main results crucially involve Schur polynomials, in
this section we present some preliminaries on them.   

Fix an integer $N>0$, and define $\bn_{\min}$ to be the tuple
$(0,1,\dots,N-1)$. Given a tuple of strictly increasing non-negative
integers $\bn = (n_0, \dots, n_{N-1})$, we will define the corresponding
Schur polynomial $s_\bn: \R^N \to \R$ in variables $(u_1, \dots, u_N)$ or
in the vector $\bu = (u_1, \dots, u_N)^T$ by the formula
\begin{equation}\label{sdef}
s_\bn(u_1, \dots, u_N) = s_\bn(\bu) \coloneqq \sum_T \bu^{|T|}.
\end{equation}

\noindent Here $T$ ranges over the column-strict Young
tableaux of shape given by the reversal $\overline{\bn-\bn_{\min}} =
(n_{N-1}-N+1, \dots, n_0)$ of $\bn - \bn_{\min} = (n_0, \dots,
n_{N-1}-N+1)$ and cell entries $1, \dots, N$, $|T|$ is the tuple
$|T| \coloneqq (a_1,\dots,a_N)$ where $a_i$ is the number of
occurrences of $i$ in the tableau $T$, and we use the multinomial
notation
\[
\bu^{|T|} = (u_1,\dots,u_N)^{(a_1,\dots,a_N)}
\coloneqq \prod_{j=1}^N u_j^{a_j}.
\]
In particular, $s_\bn$ is a homogeneous polynomial, with total degree
$\sum_{j=0}^{N-1} (n_j-j)$ and positive integer coefficients. Each Schur
polynomial $s_\bn$ may be interpreted as the character of an irreducible
polynomial representation of the Lie group $GL_N(\C)$, although we will
not need this interpretation here.

\begin{example}
Suppose $N=3$ and $\bn = (0,2,4)$, then we consider Young tableaux of
shape $(2,1,0)$ where the entries in each row (resp.~column) weakly
decrease (resp.~strictly decrease); and the entries can only be $1,2,3$.
Thus, all possible tableaux are:
\[
\begin{tabular}{|c|c|}
\hline
3 & 3 \\
\cline{1-2}
2 \\
\cline{1-1}
\end{tabular}\ , \hspace*{3mm}
\begin{tabular}{|c|c|}
\hline
3 & 3 \\
\cline{1-2}
1 \\
\cline{1-1}
\end{tabular}\ , \hspace*{3mm}
\begin{tabular}{|c|c|}
\hline
3 & 2 \\
\cline{1-2}
2 \\
\cline{1-1}
\end{tabular}\ , \hspace*{3mm}
\begin{tabular}{|c|c|}
\hline
3 & 2 \\
\cline{1-2}
1 \\
\cline{1-1}
\end{tabular}\ , \hspace*{3mm}
\begin{tabular}{|c|c|}
\hline
3 & 1 \\
\cline{1-2}
2 \\
\cline{1-1}
\end{tabular}\ , \hspace*{3mm}
\begin{tabular}{|c|c|}
\hline
3 & 1 \\
\cline{1-2}
1 \\
\cline{1-1}
\end{tabular}\ , \hspace*{3mm}
\begin{tabular}{|c|c|}
\hline
2 & 2 \\
\cline{1-2}
1 \\
\cline{1-1}
\end{tabular}\ , \hspace*{3mm}
\begin{tabular}{|c|c|}
\hline
2 & 1 \\
\cline{1-2}
1 \\
\cline{1-1}
\end{tabular}
\]

\noindent which correspond to the individual monomials in the polynomial
\begin{align*}
&\ s_\bn(u_1, u_2, u_3)\\
= &\ u_3^2 u_2 + u_3^2 u_1 + u_3 u_2^2 + u_3 u_2 u_1 + u_3 u_1 u_2 + u_3
u_1^2 + u_2^2 u_1 + u_2 u_1^2 \\
= &\ (u_1 + u_2)(u_2 + u_3)(u_3 + u_1).
\end{align*}

\noindent One may interpret $s_\bn$ as the character of the adjoint
representation of $GL_3(\C)$ on $\mathfrak{sl}_3$.
\end{example}

We will need two basic facts about Schur polynomials: see for instance
\cite[Chapter I]{Macdonald} for proofs and more details.

\begin{proposition}\label{Pschur}
Fix $N \in \mathbb{N}$ and an integer tuple $\bn = (n_0, \dots, n_{N-1})$
with $0 \leqslant n_0 < \cdots < n_{N-1}$.  Then we have the formula
\begin{equation}\label{detfactor}
 \det( \bu^{\circ n_0} |  \dots | \bu^{\circ n_{N-1}} ) = \det
 (u_j^{n_{k-1}})_{j,k = 1,\dots,N} = V(\bu) s_\bn(\bu)
\end{equation}
relating generalized Vandermonde determinants to Schur polynomials for
all $\bu \in \C^n$, where $V(\bu)$ is the Vandermonde determinant
\begin{align*}
V(\bu) \coloneqq &\ \det( \bu^{\circ 0} |  \dots |\bu^{\circ (N-1)}) \\
=&\ \det (u_j^{k-1})_{j,k = 1,\dots,N} \\
=&\ \prod_{1 \leqslant j < k \leqslant N} (u_k - u_j).
\end{align*}
In particular, the polynomial $s_\bn$ is symmetric.
Furthermore, we have the \emph{Weyl dimension formula}
\begin{equation}\label{Ewdf}
s_\bn(1,\dots,1) = \frac{V( \bn )}{V( \bn_{\min} )}.
\end{equation}
\end{proposition}

One can interpret the quantity in~\eqref{Ewdf} as the dimension of the
representation associated to the Schur polynomial $s_\bn$, although we
will not use this interpretation here.  (For instance, the adjoint
representation of $GL_3(\C)$ on $\mathfrak{sl}_3$ has dimension $8$.)
Note that~\eqref{detfactor} immediately establishes~\eqref{Egantmacher}
in the case when the exponents $\alpha_j$ are natural numbers.

The relevance of Schur polynomials to our problem comes from the
following application of Proposition~\ref{Pschur} and the Cauchy--Binet
formula.

\begin{lemma}[Determinant formula]\label{detform}
Let $S$ be a set of non-negative integers of cardinality at least $N$,
and let $h$ be a polynomial of the form
\[
h(x) = \sum_{n \in S} c_n x^n
\]
for some real coefficients $c_n$.  Then for any vector $\bu \in
(\C^n)^T$, we have
\[
\det h[ \bu \bu^* ] = \sum_{\bn \in S^N_{<}} |s_\bn( \bu )|^2
|V(\bu)|^2 \prod_{n \in \bn} c_n
\]
where $S^N_<$ denotes the set of all $N$-tuples $(n_0, \dots,
n_{N-1})$ of elements of $S$, sorted in increasing order $n_0 < \dots <
n_{N-1}$.
\end{lemma}

\begin{proof}
Write $S = \{n_1,\dots,n_M\}$ with $n_1 < \dots < n_M$.
We may factor
\begin{align*}
h[ \bu \bu^* ] &= \sum_{j=1}^M c_{n_j} \bu^{\circ n_j} (\bu^{\circ
n_j})^* \\
&= ( \bu^{\circ n_1} |  \dots | \bu^{\circ n_M} ) \mathrm{diag}(c_{n_1},
\dots, c_{n_M} ) ( \bu^{\circ n_1} | \dots | \bu^{\circ n_M} )^*.
\end{align*}
Applying the Cauchy--Binet formula, we may thus expand $\det h[\bu
\bu^*]$ as
\[
\sum_{1 \leqslant j_1 < \dots < j_N \leqslant M} |\det( \bu^{\circ
n_{j_1}} | \dots |\bu^{\circ n_{j_N}} )|^2 c_{n_{j_1}} \dots
c_{n_{j_M}}
\]
and the claim then follows from Proposition~\ref{Pschur}.
\end{proof}

\begin{remark}\label{Ralgebra}
The argument used to prove Lemma~\ref{detform} in fact gives the more
general algebraic identity
\[
\det h[ \bu \bv^T ] = \sum_{\bn \in S^N_<} s_\bn( \bu ) s_\bn(
\bv) V(\bu) V(\bv) \prod_{n \in \bn} c_n
\]
for arbitrary fields $\mathbb{F}$ and vectors $\bu, \bv \in
(\mathbb{F}^N)^T$, where $h(x) = \sum_{n \in S} c_n x^n \in
\mathbb{F}[x]$, $s_\bn(\cdot)$ is the specialization to
$\mathbb{F}[\cdot]$ of the polynomial $s_\bn(\bu) \in \Z[\bu]$, and
$V(\bu)$, $s_\bn(\bu) V(\bu)$ are (generalized) Vandermonde determinants.
\end{remark}

\section{Bounded domains: the leading term of a Schur
polynomial}\label{bounded-proof}

In this section we prove Theorem~\ref{Tmain}.  Our strategy is to first
establish the result for rank one matrices $A = \bu \bu^T$, in which one
can exploit Lemma~\ref{detform}, and then apply the fundamental theorem
of calculus to extend the entrywise positivity preservation property to
more general positive semidefinite matrices.

\subsection{The case of rank-one matrices}

We begin with the simple but crucial observation that a Schur polynomial
$s_\bn$ is comparable in size to its leading monomial, when applied to
non-negative arguments.

\begin{proposition}[Comparing Schur polynomials with a
monomial]\label{Pleading}
Fix integers $N>0$ and $0 \leqslant n_0 < \cdots < n_{N-1}$, and scalars
$0 \leqslant u_1 \leqslant \cdots \leqslant u_N$. Set $\bn \coloneqq
(n_0, \dots, n_{N-1})$ and $\bu \coloneqq (u_1, \dots, u_N)$. Then we
have
\begin{equation}\label{Eleading}
1 \times \bu^{\bn - \bn_{\min}} \leqslant s_\bn(\bu) \leqslant
\frac{V(\bn)}{V(\bn_{\min})}
\times \bu^{\bn - \bn_{\min}},
\end{equation}
where $\bn_{\min} = (0, \dots, N-1)$.
Furthermore, the constants $1$ and $\frac{V(\bn)}{V(\bn_{\min})}$ on both
sides of~\eqref{Eleading} cannot be improved.
\end{proposition}

\begin{proof}
By Proposition~\ref{Pschur}, $s_\bn(\bu)$ is the sum of exactly
$\frac{V(\bn)}{V(\bn_{\min})}$ monomials (not necessarily distinct), one
of which is equal to $\bu^{\bn - \bn_{\min}}$ (arising from the Young
tableau in which the $i^{th}$ row is entirely occupied by the number
$N+1-i$). All other monomials are of the form $\bu^{\ba}$
for some tuple $\ba = (a_1,\dots,a_N) \neq \bn - \bn_{\min}$ of 
non-negative integers summing to $\sum_{j=0}^{N-1} n_j - j$, and obeying
the majorization condition
\[
\sum_{j=0}^J a_{j+1} \geqslant \sum_{j=0}^J n_j - j
\]
for all $J=0,\dots,N-1$.  From this and the hypothesis $0\leqslant u_1
\leqslant \cdots \leqslant u_N$ we have
\begin{equation}\label{ajt}
0 \leqslant \bu^{\ba} \leqslant \bu^{\bn - \bn_{\min}}
\end{equation}
and the claim~\eqref{Eleading} follows.

Setting $u_j = 1$ for all $j$ and using~\eqref{Ewdf}, we see that the
second inequality in~\eqref{Eleading} is sharp.  For the first
inequality, we set $u_i = A^i$ for some large $A>1$ and observe that we
can now improve~\eqref{ajt} to
\[
0 \leqslant \bu^{\ba} \leqslant \frac{1}{A} \bu^{\bn - \bn_{\min}}
\]
for any monomial \thinspace appearing in $s_\bn$ other than \thinspace
the single dominant monomial $\bu^{\bn - \bn_{\min}}$.  Sending $A \to
\infty$, we obtain the claim.
\end{proof}

Next, we give the precise threshold for positive semidefiniteness of a
polynomial with $N+1$ terms applied to a  generic rank one
matrix.

\begin{proposition}\label{ab}
Let $0 \leqslant n_0 < \dots < n_{N-1} < M$ be non-negative integers, let
$c_{n_0}, \dots, c_{n_{N-1}}$ be positive reals, let $\bu =
(u_1,\dots,u_N)^T$ have distinct positive coordinates, and let $t > 0$ be
real.  Let $p_t$ denote the polynomial
\[
p_t(x) \coloneqq t \sum_{j=0}^{N-1} c_{n_j} x^{n_j} - x^M.
\]
Then $p_t[\bu \bu^T]$ is positive semidefinite if and only if
\begin{equation}\label{tal}
 t \geqslant \sum_{j=0}^{N-1} \frac{s_{\bn_j}(\bu)^2}{c_{n_j}
 s_{\bn}(\bu)^2}, 
\end{equation}
where $\bn = (n_0, \dots, n_{N-1})$ as above, and
the tuples $\bn_j$ are defined as
\[
\bn_j \coloneqq  (n_0, \dots, n_{j-1}, n_{j+1}, \dots, n_{N-1}, M),
\qquad  \forall j = 0, \dots, N-1.
\]
\end{proposition}

\begin{proof}
From Lemma~\ref{detform} we have
\begin{align}\label{dhb}
 &\ \det p_t[\bu \bu^T]\\
 = &\ t^N \left( \prod_{j=0}^{N-1} c_{n_j} \right) V(\bu)^2 s_\bn(\bu)^2
 - t^{N-1} \sum_{j=0}^{N-1} \left(\prod_{0 \leqslant k \leqslant N-1: k
 \neq j} c_{n_k} \right) V(\bu)^2 s_{\bn_j}(\bu)^2 \nonumber
\end{align}

\noindent from which we conclude that $\det p_t[\bu \bu^T]$ is
non-negative precisely when~\eqref{tal} holds.
We also see that the matrix $\sum_{j=0}^{N-1} c_{n_j}
\bu^{\circ n_j} (\bu^{\circ n_j})^T$ has determinant
\[
\left( \prod_{j=0}^{N-1} c_{n_j} \right) V(\bu)^2 s_\bn(\bu)^2,
\]
which is positive, and hence this matrix is not
just positive semidefinite but is in fact positive definite.  In
particular, $p_t[\bu \bu^T]$ is positive definite for sufficiently large
$t$.  Since the determinant function is non-negative on $\bp_N$ and
vanishes on the boundary of $\bp_N$, the claim now follows from
the continuity of the eigenvalues of $p_t[\bu \bu^T]$.
\end{proof}

\begin{remark}\label{Rrayleigh}
In the special case $\bn = \bn_{\min}$ studied in~\cite{BGKP-fixeddim},
one implication in Proposition~\ref{ab} was shown using a Rayleigh
quotient argument. That argument can be extended to work for general
$\bn$; see Section~\ref{Srayleigh}. It is also possible to obtain this
implication using the matrix determinant lemma (see e.g.~\cite{DZ}), but
we will not do so here.
\end{remark}

Now we can obtain Theorem~\ref{Tmain} (with an explicit threshold
${\mathcal K}$) in the special case of rank one matrices.

\begin{proposition}\label{Crank1}
Fix integers $N>0$ and $0 \leqslant n_0 < \cdots < n_{N-1} < M$, and
scalars $c_{n_0}, \dots, c_{n_{N-1}}$ $> 0$. Let $I \subset [0,+\infty)$
be a bounded domain, and write $\rho \coloneqq  \sup I$.  If we define
\begin{equation}\label{k1def}
 \mathcal{K} \coloneqq \sum_{j=0}^{N-1}
 \frac{V(\bn_j)^2}{V(\bn_{\min})^2} \frac{\rho^{M - n_j}}{c_{n_j}}
\end{equation}
then the polynomial
\[
x \mapsto \mathcal{K} (c_{n_0} x^{n_0} + \cdots + c_{n_{N-1}}
x^{n_{N-1}}) - x^M
\]
is entrywise positivity preserving on $\bp^1_N(I)$.
\end{proposition}

\begin{proof}
A matrix in $\bp^1_N(I)$ can be written in the form $\bu \bu^T$, where
the coordinates $u_1,\dots,u_N$ of the vector $\bu$ lie in $[0,
\sqrt{\rho}]$.  By a limiting argument, and permutation symmetry, we may
assume without loss of generality that the $u_i$ are distinct, positive,
and strictly increasing.
By Proposition~\ref{ab}, it will suffice to show that
\[
\sum_{j=0}^{N-1} \frac{s_{\bn_j}(\bu)^2}{c_{n_j} s_{\bn}(\bu)^2}
\leqslant \mathcal{K}.
\]
But by the upper and lower bounds in~\eqref{Eleading}, we have
\begin{align*}
\sum_{j=0}^{N-1} \frac{s_{\bn_j}(\bu)^2}{c_{n_j} s_{\bn}(\bu)^2}
\leqslant & \sum_{j=0}^{N-1} \frac{\left(\frac{V(\bn_j)}{V(\bn_{\min})}
\bu^{\bn_j - \bn_{\min}}\right)^2}{c_{n_j} (\bu^{\bn - \bn_{\min}})^2}\\
= & \sum_{j=0}^{N-1} \frac{V(\bn_j)^2}{V(\bn_{\min})^2 c_{n_j}}
(u_N^{M-n_{N-1}} u_{N-1}^{n_{N-1}-n_{N-2}} \cdots
u_{j+1}^{n_{j+1}-n_j})^2 \\
\leqslant &\sum_{j=0}^{N-1} \frac{V(\bn_j)^2}{V(\bn_{\min})^2}
\frac{\rho^{M - n_j}}{c_{n_j}}.
\end{align*}
The claim now follows from~\eqref{k1def}.
\end{proof}

\subsection{From rank-one matrices to all matrices}

Given the threshold $\mathcal{K}$ from Proposition~\ref{Crank1}, we now
prove the existence of a threshold for all matrices in $\bp_N([0,\rho])$,
i.e.~Theorem~\ref{Tmain}. This will follow from the following more
general ``extension principle'':

\begin{theorem}[Extension principle]\label{Tfitzhorn}
Let $0 < \rho \leqslant +\infty$. Fix an integer $N>1$, and let $h:
(0,\rho) \to \R$ be continuously differentiable.  If $h$ is entrywise
positivity preserving on $\bp^1_N((0,\rho))$, and the derivative $h'$ is
entrywise positivity preserving on $\bp_{N-1}((0,\rho))$, then $h$ is
entrywise positivity preserving on $\bp_N((0,\rho))$.  Similarly with
$(0,\rho)$ replaced by $(-\rho,\rho)$ throughout.
\end{theorem}

\begin{proof}
We use the approach in \cite[Section 3]{BGKP-fixeddim}. Suppose $A =
(a_{jk})_{j,k=1,\dots,N}$ is a matrix in $\bp_N((0,\rho))$.  Define
$\zeta$ to be the last column of $A$ divided by $\sqrt{a_{NN}}$; then $A
- \zeta \zeta^T$ has last row and column zero and is positive
semidefinite, and $\zeta \zeta^T \in \bp_N^1((0,\rho))$. We now use an
integration trick of FitzGerald and Horn \cite[Equation (2.1)]{FitzHorn}.
For any $x,y \in I$, we see from the fundamental theorem of calculus (and
a change of variables $t = \lambda x + (1-\lambda)y$) that
\[
h(x) - h(y) = \int_x^y h'(t)\ dt = \int_0^1 (x-y) h'(\lambda x +
(1-\lambda)y)\ d \lambda.
\]

\noindent Applying this entrywise with $x,y$ replaced by the entries of
$A$ and $\zeta \zeta^T$ respectively, we obtain the identity
\begin{equation}\label{Etrick}
h[A] = h[\zeta \zeta^T] + \int_0^1 ( A - \zeta \zeta^T ) \circ h'[\lambda
A + ( 1 - \lambda ) \zeta \zeta^T]\ d \lambda.
\end{equation}

As $h$ is entrywise positivity preserving on $\bp^1_N((0,\rho))$,
$h[\zeta \zeta^T]$ is positive semidefinite.  Now since $A - \zeta
\zeta^T$ is positive semidefinite and has last row and column zero, we
see from the Schur product theorem that the integrand is positive
semidefinite if the leading principal $(N-1) \times (N-1)$ submatrix of
$h' [ \lambda A + (1 - \lambda) \zeta \zeta^T]$ is.  Since these
principal submatrices of $A$ and $\zeta \zeta^T$ both lie in the convex
set $\bp_{N-1}((0,\rho))$, by assumption on $h'$ we conclude that the
integrand is everywhere positive semidefinite, whence so is $h[A]$
by~\eqref{Etrick}.
This gives the claim for $(0,\rho)$.  A similar argument works if one
replaces $(0,\rho)$ with $(-\rho,\rho)$, noting that one can easily
reduce by a limiting argument to the case where $a_{NN}$ is strictly
positive.
\end{proof}

Using Theorem~\ref{Tfitzhorn}, we now show our first main result.

\begin{proof}[Proof of Theorem~\ref{Tmain}]
Let $\mathcal{K}$ be the quantity defined in~\eqref{k1def}.  It will
suffice to show that for every $N \geqslant 1$, the polynomial
\[
h(x) \coloneqq \mathcal{K} (c_{n_0} x^{n_0} + \cdots + c_{n_{N-1}}
x^{n_{N-1}}) - x^{M}
\]
is entrywise positivity preserving on $\bp_N((0,\rho))$.

We induct on $N$.  For $N=1$ the claim follows from
Proposition~\ref{Crank1}, so suppose that $N > 1$ and that the claim has
already been proven for $N-1$.  First observe that $h$ is equal to
$x^{n_0}$ times another polynomial $\tilde h$, formed by reducing all the
exponents $n_0,\dots,n_{N-1},M$ by $n_0$; note from~\eqref{k1def} that
such a shift would not affect the quantity $\mathcal{K}$.  Also, from the
Schur product theorem we know that if $\tilde h$ is entrywise positivity
preserving on $\bp_N((0,\rho))$, then $h$ will be also.  As a
consequence, we may assume without loss of generality that $n_0=0$.

By Proposition~\ref{Crank1} and Theorem~\ref{Tfitzhorn}, it suffices to
show that $h'$ is entrywise positivity preserving on
$\bp_{N-1}((0,\rho))$.  Since
\[
h'(x) = \mathcal{K} (c_{n_1} n_1 x^{n_1-1} + \dots + c_{n_{N-1}} n_{N-1}
x^{n_{N-1}-1}) - M x^{M-1},
\]
we will be done using the induction hypothesis provided
that we can establish the inequality
\[
\mathcal{K} \geqslant M \, \tilde{\mathcal{K}},
\]
where $\tilde{\mathcal{K}}$ is defined like $\mathcal{K}$ but with
$N$ replaced by $N-1$,
$M$ replaced by $M-1$,
$n_0,\dots,n_{N-1}$ replaced by $n_1-1,\dots,n_{N-1}-1$,
and
$c_{n_0},\dots,c_{n_{N-1}}$ replaced by $n_1 c_{n_1}, \dots, n_{N-1}
c_{n_{N-1}}$ respectively.

Writing $\bn_j = (m_{j,0}, \dots, m_{j,N-1})$ for $j = 0,
\dots, N-1$, and recalling that $n_0 = 0$, we have $m_{j,0}=0$ for
$j=1,\dots,N-1$.  We may therefore verify using \eqref{k1def},
\eqref{Ewdf} that
\begin{align*}
\mathcal{K} = &\ \sum_{j=0}^{N-1} \frac{\rho^{M - n_j}}{c_{n_j}
V(\bn_{\min})^2 } \prod_{0 \leqslant a < b \leqslant N-1} (m_{j,b} -
m_{j,a})^2\\
\geqslant &\ \sum_{j=1}^{N-1} \frac{\rho^{M - n_j}}{c_{n_j}
V(\bn_{\min})^2 } \prod_{0 \leqslant a < b \leqslant N-1} (m_{j,b} -
m_{j,a})^2\\
= &\ \sum_{j=1}^{N-1} \frac{\rho^{M - n_j}}{c_{n_j} V(\bn'_{\min})^2 }
\prod_{0 < a < b \leqslant N-1}(m_{j,b} - m_{j,a})^2 \cdot
\frac{M^2}{(N-1)!^2} \prod_{b \neq 0, j} n_b^2\\
= &\ M \sum_{j=1}^{N-1} \frac{\rho^{M - n_j}}{n_j c_{n_j}
V(\bn'_{\min})^2} \prod_{0 < a < b \leqslant N-1} (m_{j,b} - m_{j,a})^2
\cdot \frac{n_j M}{(N-1)!^2} \prod_{b \neq 0,j} n_b^2\\
\geqslant &\ M \sum_{j=1}^{N-1} \frac{\rho^{M - n_j}}{n_j c_{n_j}
V(\bn'_{\min})^2} \prod_{0 < a < b \leqslant N-1} (m_{j,b} - m_{j,a})^2
\cdot \left(\frac{\prod_{b=1}^{N-1} n_b}{(N-1)!}\right)^2\\
\geqslant &\ M \sum_{j=1}^{N-1} \frac{\rho^{M - n_j}}{n_j c_{n_j}
V(\bn'_{\min})^2} \prod_{0 < a < b \leqslant N-1} (m_{j,b} - m_{j,a})^2\\
= &\ M \, \tilde{\mathcal{K}}
\end{align*}
as required, where $\bn'_{\min} \coloneqq (0, 1, \dots, N-2)$.
\end{proof}

\begin{remark}
In the case $\bn = \bn_{\min}$, this result (with the same value of
the threshold $\mathcal{K}$) was established in~\cite{BGKP-fixeddim}.
This special case is simpler due to the fact that the denominator
$s_\bn(\bu)$ is now equal to $1$.
\end{remark}

\subsection{Threshold bounds for arbitrary analytic
functions}\label{Tan-sec}

We next \thinspace prove Corollary \ref{Tanalytic}.  By replacing $\rho$
with $\rho-\epsilon$ and taking limits, we may assume without loss of
generality that the power series of interest $g(x) = \sum_{M > n_{N-1}}
g_M x^M$ in fact converges in some neighborhood of $\rho$, and hence we
have
\[
|g_M| \leqslant C \rho^{-M} (1+\epsilon)^{-M}
\]
for some $C, \epsilon > 0$ and all $M > n_{N-1}$.

By Theorem~\ref{Tmain} (with the explicit bound~\eqref{k1def}) and the
triangle inequality, it will now suffice to show that
\[
\sum_{M > n_{N-1}} \rho^{-M} (1+\epsilon)^{-M} {\mathcal K}_M < \infty,
\]
where ${\mathcal K}_M$ is the quantity~\eqref{k1def} for the specified
value of $M$.  If we write
\[
\bn'_j \coloneqq (n_0, \dots, n_{j-1}, n_{j+1}, \dots n_{N-1})
\]
for $j=1,\dots,N-1$ (so that $\bn_j = (\bn'_j, M)$), then we can use
Tonelli's theorem to compute
\begin{align*}
&\ \sum_{M > n_{N-1}} \rho^{-M} (1+\epsilon)^{-M} \mathcal{K}_{M}\\
= &\ \sum_{j=0}^{N-1} \sum_{M > n_{N-1}}  \rho^{-M}
(1+\epsilon)^{-M} \frac{V(\bn'_j)^2 \prod_{k \neq j} (M -
n_k)^2}{V(\bn_{\min})^2} \frac{\rho^{M - n_j}}{c_{n_j}}\\
= &\ \sum_{j=0}^{N-1} \frac{V(\bn'_j)^2}{V(\bn_{\min})^2 
c_{n_j} \rho^{n_j}} \sum_{M > n_{N-1}}  (1+\epsilon)^{-M} \prod_{k \neq j}
(M - n_k)^2.
\end{align*}
But the inner summand is exponentially decaying in $M$, and so this sum
is finite as required.

\section{The case of unbounded domain}\label{unbounded-proof}

We now explore the unbounded case: namely, when $\rho = +\infty$.  We
begin by proving Lemma~\ref{horn-type}(iii).  Suppose for contradiction
that this claim failed.  Applying Lemma~\ref{horn-type}(ii),
it follows that there are fewer than $N$ values of $n>n_0$
with $c_n > 0$.  By adding the  absolutely monotone
function $-\sum_{n>n_0: c_n < 0} c_n x^n$ to $f$, we may assume without
loss of generality that there are no values of $n > n_0$ with $c_n < 0$.
In particular, $f$ is now a polynomial of some degree $d \geqslant n_0$,
with fewer than $N$ terms of higher degree than $n_0$.  Now introduce the
polynomial
\[
\tilde f(x) \coloneqq x^d f(1/x).
\]
Observe that if $\bu = (u_1,\dots,u_N)^T$ is a vector with entries in
$(0,+\infty)$, then one has the identity
\[
\tilde f[\bu \bu^T] = \bu^{\circ d} \circ f[ \bu^{\circ -1} (\bu^{\circ
-1})^T ]
\]
where $\bu^{\circ \alpha} \coloneqq (u_1^\alpha,\dots,u_N^\alpha)^T$
denotes the entrywise power of $\bu$ by $\alpha$, and $\circ$ denotes the
Hadamard product.  Since $f$ is entrywise positivity preserving on
$\bp^1_N((0,+\infty))$, it follows from the Schur product
theorem that $\tilde f$ does also.  On the other hand, from construction,
$\tilde f$ has a negative $x^{d-n_0}$ coefficient, but has fewer than $N$
positive coefficients of lower degree.  This contradicts
Lemma~\ref{horn-type}(ii), as required.

Now we prove our main result in this setting.

\begin{proof}[Proof of Theorem~\ref{Tunbdd}]
By replacing $c_{n_0},\dots,c_{n_{2N-1}}$ by their minimal value and then
rescaling, we may assume without loss of generality that $c_{n_j} =1$ for
all $0 \leqslant j \leqslant 2N-1$. Setting $h$ to be the polynomial
\[
h(x) \coloneqq \sum_{j=0}^{2N-1} x^{n_j},
\]
it suffices to show that the polynomial $x \mapsto t h(x)
- x^M$ entrywise preserves positivity on $\bp_N((0,+\infty))$ for $t$
large enough.

We first establish the rank one case, that is to say that for
sufficiently large $t$ and for all $\bu = (u_1,\dots,u_N)^T$ with entries
$u_1,\dots,u_N$ in $(0,+\infty)$, we show that the matrix $t \, h[\bu
\bu^T] - \bu^{\circ M} (\bu^{\circ M})^T$ is positive semidefinite.  By a
limiting argument and symmetry we may assume that $0 < u_1 < \dots <
u_N$.  Lemma~\ref{detform} assures us that $h[\bu \bu^T]$ has positive
determinant, and is thus positive definite as opposed to merely positive
semidefinite. Thus, for each fixed $\bu$, $t \, h[\bu \bu^T] - \bu^{\circ
M} (\bu^{\circ M})^T$ is positive definite for sufficiently large $t$
(where the threshold for $t$ may possibly vary with $\bu$).  Using a
continuity argument as in the proof of Proposition~\ref{ab}, it now
suffices to show that for all sufficiently large $t$, one has
\[
\det( t \, h[\bu \bu^T] - \bu^{\circ M} (\bu^{\circ M})^T ) > 0
\]
uniformly in $\bu$.
Applying Lemma~\ref{detform}, we may write this determinant as
\[
t^N \sum_{B \in S^N_<} s_B(\bu)^2 - t^{N-1} \sum_{C \in
S^{N-1}_<} s_{C \sqcup \{M\}}(\bu)^2
\]
where $S \coloneqq \{ n_0, \dots, n_{2N-1} \}$, and $C \sqcup \{M\}$
denotes the union of the $N-1$-tuple $C$ and $\{M\}$, sorted to be in
increasing order.  It thus suffices to show that for each $C \in
S^{N-1}_<$, the ratio
\[
\frac{ s_{C \sqcup \{M\}}(\bu)^2 }{\sum_{B \in S^N_<} s_B(\bu)^2 }
\]
is uniformly bounded in $\bu$.

Fix $C$, and permute $\bu$ to have non-decreasing coordinates.
Applying Proposition~\ref{Pleading}, it suffices to show that
\[
\frac{ (\bu^{C \sqcup \{M\}})^2 }{ \sum_{B \in S^N_<} (\bu^{B})^2 }
\]
is uniformly bounded in all such $\bu$.  But as $C$ only has cardinality
$N-1$, and there are $N$ elements of $S$ that are less than $M$ and $N$
elements that are greater than $M$, there exist exponents $n_- < M < n_+$
such that $n_-, n_+ \in S \backslash C$.
(Note, this argument also shows the need for the necessary
condition in Lemma~\ref{horn-type}(iii).) This implies that
\[
(\bu^{C \sqcup \{M\}})^2  \leqslant (\bu^{C \sqcup \{n_-\}})^2 +
(\bu^{C \sqcup \{n_+\}})^2
\]

\noindent (breaking into cases depending on whether the component of
$\bu$ that will be paired with $M$ is less than $1$ or not), and hence
the above ratio is uniformly bounded by one, giving the claim.

To remove the restriction to rank one matrices, we induct on $N$ as in
the previous section.  For $N=1$ the claim is already proven, so suppose
that $N > 1$ and that the claim has already been proven for $N-1$.  By
the induction hypothesis (and discarding some manifestly entrywise
positivity preserving terms), the derivative of $t h(x) - x^M$ will
entrywise preserve positivity on $\bp_{N-1}((0,+\infty))$ for $t$ large
enough.  We have already shown that $th(x) - x^M$ also entrywise
preserves positivity $\bp^1_{N}((0,+\infty))$ for $t$ large enough.
Applying Theorem~\ref{Tfitzhorn}, we conclude that $t h(x) - x^M$
entrywise preserves positivity on all matrices in $\bp_N((0,+\infty))$
for $t$ large enough, as required.
\end{proof}

\section{Real exponents: the Harish-Chandra--Itzykson--Zuber
formula}\label{real-sec}

We now explore extensions of the above arguments to answer
Question~\ref{Q1} in the case when the exponents $n_0,\dots,n_{N-1},M$
are only assumed to be real rather than natural numbers.\ We begin by
observing that the parts (ii), (iii) of Lemma~\ref{horn-type} hold for
real powers
as well:

\begin{lemma}[Horn--Loewner-type necessary conditions for real
powers]\label{Lhorn-real}
Fix an integer $N \geqslant 2$ and a scalar $0 < \rho \leqslant +\infty$.
Further fix scalars $c_{n_i} \in \R$ and distinct real powers $n_i$ for
$i \geqslant 0$, and suppose $f(x) \coloneqq \sum_{i=0}^\infty c_{n_i}
x^{n_i}$ is a convergent sum of powers on $(0,\rho)$.
\begin{enumerate}
\item[(ii)] If $f$ is entrywise positivity preserving on
$\bp^1_N((0,\rho))$, and $c_{n_{i_0}} < 0$ for some $i_0
\geqslant 0$, then we have $c_{n_i} > 0$ for at least $N$
values of $i$ for which $n_i < n_{i_0}$.

\item[(iii)] If $f$ is entrywise positivity preserving on
$\bp^1_N((0,+\infty))$, and $c_{n_{i_0}}  < 0$ for some $i_0 \geqslant
0$, then in addition to (ii) we also have $c_{n_i} > 0$ for at least $N$
values of $i$ such that $n_i > n_{i_0}$.
\end{enumerate}
\end{lemma}

The proofs are minor modifications of those of Lemma~\ref{horn-type}(ii),
(iii) respectively.

The objectives of the remainder of this section are to show that (in
analogy to the integer exponent case) the necessary conditions in
Lemma~\ref{Lhorn-real} are once again completely sharp, and that one can
obtain threshold bounds for all Puiseux or Hahn series, in analogy to
Corollary~\ref{Tanalytic}, with quantitative bounds.
For non-negative rational power exponents, one can achieve the
first two objectives by the simple change of variables $y_j \coloneqq
u_j^{1/L}$, where $L>0$ is a common denominator for the rationals
$n_0,\dots,n_{N-1},M$.  However, the quantitative bounds obtained by
doing so depend on $L$ in an unfavorable manner, and so this approach
does not seem to easily extend to the general real exponent case.  Hence
we shall adopt a different approach in the arguments below.

\subsection{Sign patterns of sums of powers}

In this subsection we resolve Question~\ref{Q1} for the more involved
case of real powers (which we take to be non-negative because negative
powers cannot entrywise preserve positivity on matrices of rank $2$ or
higher).
As the theory of Schur polynomials crucially requires integer powers
(or rational powers via the above workaround),
in place of it we now rely on the Harish-Chandra--Itzykson--Zuber
identity
\begin{align}\label{hciz}
&\ \det( e^{\alpha_i x_j} )_{1 \leqslant i,j \leqslant N}\\
= &\ \frac{V(\bal) V(\bx)}{V(\bn_{\min})} \int_{U(N)} \exp \mathrm{tr}
\left( \mathrm{diag}(\alpha_1,\dots,\alpha_N) U
\mathrm{diag}(x_1,\dots,x_N) U^* \right) \ dU, \nonumber
\end{align}

\noindent which is valid for any tuples $\bal =
(\alpha_1,\dots,\alpha_N)$, $\bx = (x_1,\dots,x_N)$ of real numbers, and
where $dU$ denotes Haar probability measure on the unitary group~$U(N)$;
see e.g.,~\cite[(3.4)]{iz}. 
(We thank Ryan O'Donnell for drawing our attention to this identity.)
If $\alpha_1 \leqslant \dots \leqslant \alpha_N$ and
$x_1 \leqslant \dots \leqslant x_N$, then by the Schur--Horn theorem
\cite{sh,ahorn}, the diagonal entries of $U\mathrm{diag}(x_1,\dots,x_N)
U^*$ are majorized by $(x_1,\dots,x_N)$, and hence the trace in the above
expression ranges between $\sum_{j=1}^N \alpha_j x_{N+1-j}$ and
$\sum_{j=1}^N \alpha_j x_j$.  As all Vandermonde determinants appearing
here are non-negative, we conclude the (somewhat crude) inequalities
\[
\frac{V(\bal) V(\bx)}{V(\bn_{\min})}
\exp \sum_{j=1}^N \alpha_j x_{N+1-j}
\leqslant \det( e^{\alpha_i x_j} )_{1 \leqslant i,j \leqslant N}
\leqslant \frac{V(\bal) V(\bx)}{V(\bn_{\min})}
\exp \sum_{j=1}^N \alpha_j x_j.
\]

\noindent Writing $u_j = \exp(x_j)$ and $\overline{\bal} =
(\alpha_N, \dots, \alpha_1)$, we thus have
\begin{equation}\label{vbal}
 \frac{V(\bal) V(\log[\bu])}{ V(\bn_{\min})}
 \bu^{\overline{\bal}} \leqslant \det( \bu^{\circ \alpha_1} | \dots
 |\bu^{\circ \alpha_N} ) \leqslant \frac{V(\bal)
 V(\log[\bu])}{ V(\bn_{\min})} \bu^{\bal}.
\end{equation}

\noindent In particular,~\eqref{vbal} implies the following upper and
lower bounds for this determinant in the case that $\bu$ ranges in a
compact set:

\begin{lemma}
Let $I \subset (0,+\infty)$ be a compact interval, and let $K$ be a
compact subset of the cone
\[
\{ (\alpha_1,\dots,\alpha_N) \in \R^N: \alpha_1 < \alpha_2 < \cdots <
\alpha_N \}.
\]
Then there exist constants $C, c>0$ such that
\[
c |V(\bu)| \leqslant |\det( \bu^{\circ \alpha_1} | \dots |\bu^{\circ
\alpha_N} )| \leqslant C |V(\bu)|
\]
for all $\bu \in I^N$ and all $\bal = (\alpha_1,\dots,\alpha_N) \in K$.
\end{lemma}

\begin{proof}
For $\bal \in K$ and $\bu \in I^N$, $V(\bal)$ is bounded above and below
by constants depending only on $I, K, N$, as are $\bu^{\bal}$ and
$\bu^{\overline{\bal}}$.  Furthermore, for each $1 \leqslant i < j
\leqslant N$, $|\log(u_i) - \log(u_j)|$ is comparable to $|u_i-u_j|$
thanks to the mean value theorem. The claim now follows
from~\eqref{vbal}.
\end{proof}

Now we extend the above lemma to obtain estimates when the
arguments~$\bu$ are not restricted to a compact set.

\begin{lemma}\label{plead-2}  Let $K$ be a compact subset of the cone
\[
\{ (n_0,\dots,n_{N-1}) \in \R^N: n_0 < n_1 < \cdots < n_{N-1} \}.
\]
Then there exist constants $C, c>0$ such that
\[
c V(\bu) \bu^{\bn - \bn_{\min}} \leqslant \det( \bu^{\circ n_0} |
\dots | \bu^{\circ n_{N-1}} ) \leqslant C V(\bu) \bu^{\bn -
\bn_{\min}}
\]
for all $\bu = (u_1,\dots,u_N)^T \in ((0,+\infty)^N)^T$
with $u_1 \leqslant \dots \leqslant u_N$ and all $\bn =
(n_0,\dots,n_{N-1}) \in K$.
\end{lemma}

If $n_0,\dots,n_{N-1}$ were restricted to be integers, then this claim
would follow directly from Proposition~\ref{Pschur} and
Proposition~\ref{Pleading}.  One can view this lemma as a substitute for
these propositions in the non-integer setting.

\begin{proof}
The claim is easy when $N=1$, so we suppose inductively that $N>1$ and
that the claim has already been proven for all smaller values of $N$.  By
a limiting argument we may assume that $0 < u_1 < \dots < u_N$.

Let $A>2$ be a large constant to be chosen later.  We first consider the
non-separated case in which $u_{i+1}/u_i < A$ for all $i=1,\dots,N-1$.
By dividing all the $u_j$ by (say) $u_1$, we may normalize $u_1=1$
without loss of generality, and now the $u_1,\dots,u_N$ are all confined
to a compact subset of $(0,+\infty)$, in which case the claim follows
from the previous lemma.  

Now suppose that one has $u_{i+1}/u_i \geqslant A$ for some $1 \leqslant
i < N$.  We can split $\bu$ into the two smaller vectors $\bu'
\coloneqq (u_1,\dots,u_i)^T$ and $\bu'' \coloneqq (u_{i+1},\dots,u_N)^T$.
By cofactor expansion, we may then express $\det( \bu^{\circ n_0} | \dots
| \bu^{\circ n_{N-1}} )$ as the alternating sum of $\binom{N}{i}$
products of the form
\[
\det( (\bu')^{\circ n'_1} | \dots | (\bu')^{\circ n'_i} )
\cdot \det( (\bu'')^{\circ n''_1} | \dots | (\bu'')^{\circ n''_{N-i}} )
\]
where the $n'_1,\dots,n'_i, n''_1,\dots,n''_{N-i}$ are a permutation of
$n_0,\dots,n_{N-1}$ with $n'_1 < \dots < n'_i$ and $n''_1 < \dots <
n''_{N-i}$.  By the induction hypothesis, each such product is comparable
to \[
V(\bu') (\bu')^{\bn' - (0,\dots,i-1)} V(\bu'') (\bu'')^{\bn'' -
(0,\dots,N-i-1)}
\]
where $\bn' \coloneqq (n'_1,\dots,n'_i)$ and $\bn'' \coloneqq
(n''_1,\dots,n''_{N-i})$.

If $1 \leqslant j \leqslant i < k \leqslant N$, then $u_k \geqslant
A u_i \geqslant 2 u_j$, and hence $u_k - u_j$ is
comparable to $u_k$.  From this we conclude that $V(\bu)$ is comparable
to $V(\bu') \cdot V(\bu'') \cdot (\bu'')^{(i,\dots,i)}$, and hence
the preceding expression is comparable to
\[
V(\bu) \bu^{(\bn',\bn'') - (0,\dots,N-1)}.
\]
As $(\bn', \bn'')$ is a rearrangement of $\bn$, one has
\[
\bu^{(\bn',\bn'') - (0,\dots,N-1)} \leqslant \bu^{\bn - \bn_{\min}}
\]
and furthermore (because all the entries of $\bu''$ are at least $A$
times larger than those of $\bu'$) one has the refinement
\[
\bu^{(\bn',\bn'') - (0,\dots,N-1)} \leqslant \frac{1}{A} \bu^{\bn -
\bn_{\min}}
\]
unless $\bn' = (0,\dots,i-1)$ and $\bn'' = (i,\dots,N-1)$.  For $A$ large
enough, this (together with~\eqref{Egantmacher}) proves the desired lower
bound; and the upper bound also follows, using the triangle inequality.
\end{proof}

Repeating the proof of Proposition~\ref{Crank1}, using
Lemma~\ref{plead-2} as a replacement for Lemma~\ref{Pleading}, we
conclude

\begin{proposition}\label{Crank2}
Let $n_0 < \cdots < n_{N-1} < M$ and scalars $c_{n_0}, \dots,
c_{n_{N-1}}$ $> 0$ be real numbers. Let $I \subset (0,+\infty)$ be a
bounded domain.  Then for sufficiently large $t$, the function
\[
x \mapsto t (c_{n_0} x^{n_0} + \cdots + c_{n_{N-1}} x^{n_{N-1}}) - x^{M}
\]
is entrywise positivity preserving on $\bp^1_N(I)$.
\end{proposition}

Using Theorem~\ref{Tfitzhorn}, we may remove the rank one restriction
assuming that the $n_0,\dots,n_{N-1}$ are either natural numbers or not
too small, giving a version of Theorem~\ref{Tmain} for real exponents:

\begin{theorem}\label{Treal-1}
Let $0 \leqslant n_0 < \cdots < n_{N-1} < M$ and scalars $c_{n_0}, \dots,
c_{n_{N-1}}$ $> 0$ be real numbers.   Assume that each $n_i$ is either a
non-negative integer, or is greater than $N-2$ (or both).  Let $I \subset
[0,+\infty)$ be a bounded domain.  Then for sufficiently large $t$, the
function
\begin{equation}\label{tpol}
 x \mapsto t (c_{n_0} x^{n_0} + \cdots + c_{n_{N-1}} x^{n_{N-1}}) -
x^{M}
\end{equation}
is entrywise positivity preserving on $\bp_N(I)$.
\end{theorem}

The condition that each $n_i$ is either a non-negative integer
or greater than $N-2$ is natural in view of the results
in~\cite{FitzHorn}, in which it is shown that these conditions are
necessary and sufficient to ensure that $x \mapsto x^{n_i}$ is entrywise
positivity preserving on $\bp_N((0,+\infty))$.

\begin{proof}
The claim is trivial for $N=1$.  Now we consider the $N=2$ case.  In this
case it follows from the results in~\cite{FitzHorn} that the map $x
\mapsto x^{n_0}$ is entrywise positivity preserving on
$\bp_N((0,+\infty))$.  By the Schur product theorem, we may thus factor
out $x^{n_0}$ and assume without loss of generality that $n_0=0$.
Similarly, the map $x \mapsto x^{n_1}$ is entrywise positivity preserving
on $\bp_N((0,+\infty))$, and so by composing with this map we may assume
that $n_1=1$.  For $t$ large enough, the derivative of the function $t
(c_0 + c_1 x) - x^M$ is then entrywise positivity preserving on
$\bp_1(I)$, and the claim now follows from Theorem~\ref{Tfitzhorn} and
Proposition~\ref{Crank2}.

Now suppose inductively that $N > 2$, and that the claim has already been
proven for $N-1$.  Observe that the derivative of the
polynomial~\eqref{tpol} is of the form required for the inductive
hypothesis (all the surviving monomials have exponents that are either
non-negative integers, or greater than $N-3$, or both).  Thus the
derivative will be entrywise positivity preserving on~$\bp_{N-1}(I)$ for
$t$ large enough, and the claim again follows from
Theorem~\ref{Tfitzhorn} and Proposition~\ref{Crank2}.
\end{proof}

We can now give the complete solution to Question~\ref{Q1} for
real powers, which shows Lemma~\ref{Lhorn-real}(i) is sharp.

\begin{theorem}\label{Tpowerseries}
Let $N \geqslant 2$, and let $\{ n_i : i \geqslant 0 \} \subset
\Z^{\geqslant 0} \cup [N-2,\infty)$ be a set of pairwise distinct real
numbers. For each $i$, let $\epsilon_i \in \{-1,0,+1\}$ be
a sign such that whenever $\epsilon_{i_0} = -1$, one has $\epsilon_i =
+1$ for at least $N$ choices of $i$ satisfying: $n_i < n_{i_0}$.
Let $0 < \rho < +\infty$. Then there exists a convergent series with real
coefficients
\[
f(x) = \sum_{i=0}^\infty c_{n_i} x^{n_i}
\]
on $(0,\rho)$ that is an entrywise positivity preserver on
$\bp_N((0,\rho))$, such that~$c_{n_i}$ has the sign of
$\epsilon_i$ for all $i \geqslant 0$.
\end{theorem}

We remark that a difference between power series (from previous sections)
and countable sums of real powers is that the latter can include an
infinite decreasing set of powers.

\begin{proof}
The proof uses the following computation that is also useful below: given
any set $\{ n_i : i \geqslant 0 \}$ of (pairwise distinct)
non-negative powers, we claim that
\begin{equation}\label{Erealsum}
\sum_{i \geqslant 0} \frac{x^{n_i}}{i! \lceil n_i \rceil!} < \infty,
\qquad \forall x > 0.
\end{equation}
Indeed, partition the non-negative integers as
\[
\Z^{\geqslant 0} = \sqcup_{j \geqslant 0} I_j, \quad \text{where} \quad
I_j \coloneqq \{ i \geqslant 0 : n_i \in (j-1,j] \}.
\]
Using Tonelli's theorem, we crudely estimate:
\begin{align*}
\sum_{i \geqslant 0} \frac{x^{n_i}}{i! \lceil n_i \rceil!} =
\sum_{j \geqslant 0} \frac{1}{j!} \sum_{i \in I_j} \frac{x^{n_i}}{i!}
\leqslant &\ e + \sum_{j \geqslant 1} \frac{1}{j!} \sum_{i \in I_j}
\frac{x^j + x^{j-1}}{i!}\\
< &\ e + e(e^x + x^{-1} e^x) < \infty.
\end{align*}

Now to prove the result, let $J \subset \Z^{\geqslant 0}$ denote the
subset $\{ i : \epsilon_i = -1 \}$. For each $j \in J$ we have $i_1(j),
\dots, i_N(j)$ such that $\epsilon_{i_k(j)} = 1$ and $n_{i_k(j)} < n_j$,
for $k=1,\dots,N$. We define
\[
f_j(x) \coloneqq \sum_{k=1}^N \frac{x^{n_{i_k(j)}} }{\lceil n_{i_k(j)}
\rceil!} - \delta_j \frac{x^{n_j}}{\lceil n_j \rceil!},
\]
where $\delta_j \in (0,1)$ is chosen such that $f_j[-]$ preserves
positivity on $\bp_N((0,\rho))$ by Theorem~\ref{Treal-1}.
Let $J'$ denote the set of all $i \geqslant 0$ for which $\epsilon_i =
+1$ but $i \neq i_k(j)$ for any $j \in J,\ k \in [1,N]$. Finally, define
\[
f(x) \coloneqq \sum_{j \in J} \frac{f_j(x)}{j!} + \sum_{i \in J'}
\frac{x^{n_i}}{i! \lceil n_i \rceil!}, \qquad x > 0.
\]
By repeating the above computation~\eqref{Erealsum}, one verifies $f$ is
absolutely convergent on $(0,+\infty)$ and hence on $(0,\rho)$. By the
Schur product theorem and the above hypotheses, it follows that $f[-]$
preserves positivity on $\bp_N((0,\rho))$. 
\end{proof}

In a similar vein to the bounded case, for the unbounded domain
$(0,+\infty)$ we may adapt the proof of Theorem~\ref{Tunbdd} to real
exponents, using Lemma~\ref{plead-2} as a replacement for
Lemma~\ref{Pleading}, to obtain

\begin{theorem}\label{Tunbdd-2}
Let $N > 0$ and $0 \leqslant n_0 < \dots < n_{N-1} < M < n_N < \dots <
n_{2N-1}$ be real numbers, such that each of the $n_0,\dots,n_{N-2}$ are
either non-negative integers, greater than $N-2$, or both.
Let $c_{n_0},\dots,c_{n_{2N-1}}$ be positive reals.  Then there exists a
negative number $c_M$ such that
\[
x \mapsto c_{n_0} x^{n_0} + c_{n_1} x^{n_1} + \dots + c_{n_{N-1}}
x^{n_{N-1}} + c_M x^M + c_{n_N} x^{n_N} + \dots + c_{n_{2N-1}}
x^{n_{2N-1}}
\]
entrywise preserves positivity on $\bp_N((0,+\infty))$.
\end{theorem}

Using this theorem, we can show that Lemma~\ref{Lhorn-real}(ii) is sharp,
vis-a-vis Question~\ref{Q1}:

\begin{theorem}\label{Tpowerseries2}
Let $N \geqslant 2$, and let $\{ n_i : i \geqslant 0 \} \subset
\Z^{\geqslant 0} \cup [N-2,\infty)$ be a set of pairwise distinct real
numbers. For each $i$, let $\epsilon_i \in \{-1,0,+1\}$ be a sign such
that whenever $\epsilon_{i_0} = -1$, one has $\epsilon_i = +1$ for at
least $N$ choices of $i$ satisfying: $n_i < n_{i_0}$, and at least $N$
choices of $i$ satisfying: $n_i > n_{i_0}$.
Then there exists a convergent series with real coefficients
\[
f(x) = \sum_{i=0}^\infty c_{n_i} x^{n_i}
\]
on $(0,+\infty)$ that is an entrywise positivity preserver on
$\bp_N((0,+\infty))$, such that $c_{n_i}$ has the sign of
$\epsilon_i$ for all $i \geqslant 0$.
\end{theorem}

The proof is similar to that of Theorem~\ref{Tpowerseries} and is left to
the interested reader.

\subsection{Bounds for Laplace transforms}

Our final result in this section obtains a similar assertion to
Corollary~\ref{Tanalytic} for real powers. In this setting we begin with
real powers $0 \leqslant n_0 < \cdots < n_{N-1}$, and replace the
analytic function $g(x) = \sum_{M > n_{N-1}} g_M x^M$ from
Corollary~\ref{Tanalytic} by Laplace transforms against more general
measures,
\begin{equation}\label{Elaplace}
g_\mu(x) \coloneqq \int_{n_{N-1}}^\infty x^t\ d\mu(t),
\end{equation}

\noindent which we assume to be absolutely convergent at $\rho$. We now
prove:

\begin{theorem}\label{Tlaplace}
Fix an integer $N>0$ and suppose $n_0,\dots, n_{N-1}$ are strictly
increasing real numbers in the set $\Z^{\geqslant 0} \cup [N-2,\infty)$.
Also fix positive real scalars $\rho, c_{n_0}, \dots, c_{n_{N-1}} > 0$.
Given $\varepsilon > 0$ and a real measure $\mu$ supported on
$[n_{N-1} + \varepsilon,\infty)$ whose ``Laplace transform'' $g_\mu$
(defined in~\eqref{Elaplace}) is absolutely convergent at $\rho$,
there exists a finite threshold
\[
\mathcal{K} = \mathcal{K}(n_0,\dots,n_{N-1}, \rho, c_{n_0}, \dots,
c_{n_{N-1}}, g_\mu) > 0
\]
such that the function
\[
x \mapsto \mathcal{K} \sum_{j=0}^{N-1} c_{n_j} x^{n_j} -
g_\mu(x)
\]
is entrywise positivity preserving on $\bp_N((0,\rho))$.  Equivalently,
one has
\begin{equation}
 g_\mu[A] \preceq {\mathcal K} \sum_{j=0}^{N-1} c_{n_j} A^{\circ n_j}
\end{equation}
for all $A \in \bp_N((0,\rho))$.
\end{theorem}

The remainder of this section is devoted to proving
Theorem~\ref{Tlaplace}. The key improvement over the previous subsection
that is required in the proof is a sharper version of
Lemma~\ref{plead-2}:

\begin{proposition}[Leading term of generalized Vandermonde
determinants]\label{plead-3}
Let $n_0 < \dots < n_{N-1}$ be $1$-separated, in the sense that $n_{i+1}
- n_i \geqslant 1$ for all $0 \leqslant i < N-1$.  Then for all $\bu =
(u_1,\dots,u_N)^T \in ((0,+\infty)^N)^T$ with $u_1 \leqslant \dots
\leqslant u_N$, we have
\begin{equation}\label{pup-real}
 1 \times V(\bu) \bu^{\bn - \bn_{\min}} \leqslant \det(
 \bu^{\circ n_0} | \dots | \bu^{\circ n_{N-1}} ) \leqslant 
 \frac{V(\bn)}{V(\bn_{\min})} \times V(\bu) \bu^{\bn - \bn_{\min}},
\end{equation}
where $\bn_{\min} \coloneqq (n_0,\dots,n_{N-1})$. Moreover, the lower and
upper bounds of $1, \frac{V(\bn)}{V(\bn_{\min})}$ cannot be improved.
\end{proposition}

Note that~\eqref{pup-real} matches the bounds in
Proposition~\ref{Pleading}, and hence extends that result to all real
powers.

To prove Proposition~\ref{plead-3}, we require the following
generalization of a symmetric function identity to real powers (which is
also required later; and which we show for completeness).

\begin{proposition}[Principal specialization of generalized Vandermonde
determinants]\label{Pprincipal}
Fix an integer $N>0$ and real powers $n_0 < \cdots < n_{N-1}$. Also
define
\begin{equation}\label{Evector}
\bu(\epsilon) \coloneqq (1, \epsilon, \dots, \epsilon^{N-1})^T, \qquad
\epsilon > 0.
\end{equation}
Defining the matrix $\bu(\epsilon)^{\circ \bn} \coloneqq
(\bu(\epsilon)^{\circ n_0} | \dots | \bu(\epsilon)^{\circ n_{N-1}})$, we
have:
\begin{equation}\label{Eprincipal}
\frac{\det \bu(\epsilon)^{\circ \bn}}{V(\bu(\epsilon))} = \prod_{0
\leqslant i < j \leqslant N-1} \frac{\epsilon^{n_j} -
\epsilon^{n_i}}{\epsilon^j - \epsilon^i}, \qquad \forall \epsilon > 0, \
\epsilon \neq 1.
\end{equation}
\end{proposition}

Notice that unlike Proposition~\ref{plead-3}, we do not require the
entries of $\bu(\epsilon)$ to be non-decreasing, whence $\epsilon > 0$
can be arbitrary.

\begin{proof}
The transpose of $\bu(\epsilon)^{\circ \bn}$ is the Vandermonde
matrix $(\bv^{\circ 0}|\dots|\bv^{\circ (N-1)})$ with $\bv \coloneqq
(\epsilon^{n_0}, \dots, \epsilon^{n_{N-1}})^T$.  In particular, the
determinant of this matrix is $V(\bv) = \prod_{1 \leqslant i < j
\leqslant N-1} (\epsilon^{n_j} - \epsilon^{n_i})$, and the claim follows.
\end{proof}

\begin{remark}
One can view this identity as the real exponent version of the principal
specialization of the Weyl Character Formula in type $A$ (see
e.g.\ \cite[Chapter I.3]{Macdonald}), which says that for integers $0
\leqslant n_0 < \cdots < n_{N-1}$, and a variable $q$,
\[
s_\bn(1,q,\dots,q^{N-1}) = \prod_{0 \leqslant i < j \leqslant N-1}
\frac{q^{n_j} - q^{n_i}}{q^j - q^i}
\]
over any ground field. 
\end{remark}

We now prove the aforementioned tight bounds on generalized Vandermonde
determinants.

\begin{proof}[Proof of Proposition~\ref{plead-3}]
Note that if $n_0 < 0$, then by multiplying all terms in the
inequality~\eqref{pup-real} by $(u_1 \cdots u_N)^{-n_0}$, one can reduce
to the case of non-negative powers $n_j - n_0$. Thus, we suppose
henceforth that $n_0 \geqslant 0$. By a limiting argument, we may assume
without loss of generality that $0 < u_1 < \dots < u_N$.
From~\eqref{hciz} we have
\begin{align}\label{eo}
&\ \det( \bu^{\circ n_0} | \dots | \bu^{\circ n_{N-1}} )\\
= &\ \frac{V(\bn) V(\log[\bu])}{V(\bn_{\min})} \int_{U(N)} \exp
\mathrm{tr} \left( \mathrm{diag}(\bn) U \mathrm{diag}(\log[\bu]) U^*
\right) \ dU;\nonumber
\end{align}

\noindent replacing $\bn$ by $\bn_{\min}$, we also see that
\[
V(\bu) = V(\log [\bu]) \int_{U(N)} \exp \mathrm{tr} \left(
\mathrm{diag}(\bn_{\min}) U \mathrm{diag}(\log[\bu]) U^* \right) \ dU
\]
and hence
\begin{align*}
&\ \det( \bu^{\circ n_0} | \dots | \bu^{\circ n_{N-1}} )\\
\leqslant &\ \frac{V(\bn) V(\bu)}{V(\bn_{\min})} \sup_{U \in U(N)} \exp
\mathrm{tr} \left( \mathrm{diag}(\bn - \bn_{\min}) U
\mathrm{diag}(\log[\bu]) U^* \right).
\end{align*}

By \thinspace the \thinspace Schur--Horn \thinspace theorem
\cite{sh, ahorn}, \thinspace the \thinspace diagonal \thinspace entries
\thinspace of \thinspace the \thinspace
matrix $U\mathrm{diag}(\log[\bu]) U^*$ are majorized by $\log[\bu]$.  By
hypothesis, the vectors $\bn - \bn_{\min}$ and $\log[\bu]$ have
non-decreasing entries, and hence
\[
\exp \mathrm{tr} \left( \mathrm{diag}(\bn - \bn_{\min}) U
\mathrm{diag}(\log[\bu]) U^* \right) \leqslant \bu^{\bn - \bn_{\min}}.
\]
Putting all this together, we obtain the upper bound in~\eqref{pup-real}.


Now we turn to the lower bound.  Using the integration formula in
\cite[(3.2)]{shat} (we thank Abdelmalek Abdesselam for this reference),
we may write the right-hand side of~\eqref{eo} as
\begin{equation}\label{eo1}
 V(\log [\bu]) \int_{GT(\bn)} \exp\left(\sum_{k=1}^N \left(
 \sum_{i=1}^{N-k+1} m_i^{k-1} - \sum_{i=1}^{N-k} m_i^k\right) \log u_k
 \right)
\end{equation}
where the \emph{Gelfand--Tsetlin polytope} $GT(\bn)$ is the collection of
all tuples
\[
(m_i^k)_{1 \leqslant k \leqslant N-1; 1 \leqslant i \leqslant
N-k} \in \R^{N(N-1)/2}
\]
of real numbers $m_i^k$ obeying the interlacing relations
\[
m_i^k > m_i^{k+1} > m_{i+1}^k, \qquad 0 \leqslant k < N-1, \ 1 \leqslant
i < N-k,
\]
with the convention that $m_i^0 = n_{i-1}$ for $i=1,\dots,N$, and the
integration is with respect to Lebesgue measure on this polytope.

\begin{remark}
The reader may wish to first run the argument here in the simple case
$N=2$, $n_0=0$, $u_1=1$, in which case the formula~\eqref{eo1} simplifies
to $u_2^{n_1} - 1 = \log(u_2) \int_0^{u_2} \exp( m_1^1 \log u_2 )\
dm_1^1$, while the derivation~\eqref{eo-2} below becomes $u_2^{n_1} - 1 =
\log(u_2) u_2^{n_2} \int_0^{u_2} \exp( - \beta_1^1 \log u_2)\
d\beta_1^1$.  The formula~\eqref{eo1} can also be thought of as a
continuous or ``classical'' version of~\eqref{detfactor}, or
equivalently~\eqref{detfactor} may be thought of as a discrete or
``quantized'' version of~\eqref{eo1}.
\end{remark}

Returning to the general case, if we now
write $\log u_k = \sum_{j=1}^k \alpha_j$ for some reals $\alpha_j = \log
u_j - \log u_{j-1}$ (with the convention $\log u_0 = 0$), we can
telescope the expression
\[
\sum_{k=1}^N \left(\sum_{i=1}^{N-k+1} m_i^{k-1} - \sum_{i=1}^{N-k}
m_i^k\right) \log u_k
\]
appearing in the above formula as
\[
\sum_{j=1}^N \alpha_j \sum_{i=0}^{N-j+1} m_i^{j-1};
\]
making the linear change of variables
\[
m_i^j = n_{i+j-1} - \beta^1_{i+j-1} - \beta^2_{i+j-2} - \dots - \beta^i_j
\]
for a tuple $(\beta_i^k)_{1 \leqslant k \leqslant N-1; 1 \leqslant i
\leqslant N-k} \in \R^{N(N-1)/2}$ in a sheared version
$\widetilde{GT}(\bn)$ of the Gelfand--Tsetlin polytope, this expression
can be telescoped again as
\[
\sum_{j=1}^N \alpha_j \sum_{i=0}^{N-j+1} n_{i+j} - \sum_{k=1}^{N-1}
\sum_{i=1}^{N-k} \beta_i^k (\log u_{i+k-1}-\log u_{i-1}).
\]
Since
\[
\exp \left( \sum_{j=1}^N \alpha_j \sum_{i=0}^{N-j+1} n_{i+j} \right) =
\bu^\bn
\]
we thus have the identity
\begin{align}\label{eo-2}
&\ \det( \bu^{\circ n_0} | \dots | \bu^{\circ n_{N-1}} )\\
= &\ V(\log [\bu]) \bu^\bn \int_{\widetilde{GT}(\bn)} \exp\left( -
\sum_{k=1}^{N-1} \sum_{i=1}^{N-k} \beta_i^k (\log u_{i+k-1}-\log u_{i-1})
\right).\nonumber
\end{align}

\noindent Replacing $\bn$ by $\bn_{\min}$, we also have
\[
V(\bu) = V(\log [\bu]) \bu^{\bn_{\min}} \int_{\widetilde{GT}(\bn_{\min})}
\exp \left( - \sum_{k=1}^{N-1} \sum_{i=1}^{N-k} \beta_i^k (\log
u_{i+k-1}-\log u_{i-1}) \right).
\]
Observe that the polytope $\widetilde{GT}(\bn)$ is cut out by the
inequalities $\beta_i^k > 0$, as well as
\[
\beta^1_{i+k} + \dots + \beta^{k+1}_i - \beta^1_{i+k-1} - \dots -
\beta^k_i \leqslant n_{i+k} - n_{i+k-1}
\]
for $0 \leqslant k < N$ and $1 \leqslant i < N-k$.  In particular, as
$\bn$ is $1$-separated, we have the inclusion
\[
\widetilde{GT}(\bn_{\min}) \subset \widetilde{GT}(\bn)
\]
and the lower bound in~\eqref{pup-real} follows.

Finally, we prove sharpness of the lower and upper bounds
in~\eqref{pup-real}, \thinspace using the principal specialization
formula~\eqref{Eprincipal} with $\epsilon > 1$. Indeed, if $\bu_\epsilon
\coloneqq \rho \bu(\epsilon) =$ $\rho (1, \epsilon, \dots,
\epsilon^{N-1})^T$, then
\begin{align}\label{Egvm}
\begin{aligned}
\frac{\det \bu_\epsilon^{\circ \bn}}{V(\bu_\epsilon)
\bu_\epsilon^{\bn - \bn^{\min}} } = &\ \prod_{j=0}^{N-1} \epsilon^{j(j -
n_j)} \prod_{0 \leqslant i < j \leqslant N-1} \frac{\epsilon^{n_j} -
\epsilon^{n_i}}{\epsilon^j - \epsilon^i}\\
= &\ \prod_{0 \leqslant i < j \leqslant N-1} \frac{1 - \epsilon^{n_i -
n_j}}{1 - \epsilon^{i - j}}.
\end{aligned}
\end{align}
Now the sharpness of the upper bound in~\eqref{pup-real} follows by
taking $\epsilon \to 1^+$, while that of the lower bound follows by
taking $\epsilon \to \infty$.
\end{proof}

Using the tight bounds in Proposition~\ref{plead-3}, we now sharpen
Proposition~\ref{Crank2} to obtain an explicit bound for rank-one
matrices, with arbitrary tuples of real powers.

\begin{proposition}\label{Crank3}
Let $n_0 < \cdots < n_{N-1} < M$ and scalars
$\rho,c_{n_0}, \dots, c_{n_{N-1}}$ $> 0$ be real numbers. Define
$ \delta_{\bn,M} \coloneqq \min (n_1 - n_0, \dots, n_{N-1} -
n_{N-2}, M - n_{N-1})$.
Then the function
\[
x \mapsto \mathcal{K} (c_{n_0} x^{n_0} + \cdots + c_{n_{N-1}}
x^{n_{N-1}}) - x^{M}
\]
is entrywise positivity preserving on $\bp^1_N((0,\rho))$, where
\begin{equation}\label{k2def}
 \mathcal{K} \coloneqq \delta_{\bn,M}^{-N(N-1)}
 \sum_{j=0}^{N-1} \frac{V(\bn_j)^2}{V(\bn_{\min})^2} \frac{\rho^{M -
 n_j}}{c_{n_j}}.
\end{equation}
\end{proposition}

\begin{proof}
Notice that the proof of Proposition~\ref{ab} applies on the nose to real
powers $n_0, \dots, n_{N-1}, M$, replacing $V(\bu) s_\bn(\bu)$ by $\det
(\bu^{\circ n_0} | \dots | \bu^{\circ n_{N-1}})$ and similarly with
$\bn_j$ instead of $\bn$. Now define
\[
{\bf m} \coloneqq \frac{1}{\delta} {\bf n}, \ \
{\bf m}_j \coloneqq \frac{1}{\delta} \bn_j, \ \
\bv \coloneqq (u_1^\delta, \dots, u_N^\delta)^T \in
(0,\rho^{\delta/2})^N, \ \ \text{where } \delta =  \delta_{\bn,M};
\]
note that ${\bf m}$ and ${\bf m}_j$ are all $1$-separated. Now we repeat
the proof of Proposition~\ref{Crank1} using Proposition~\ref{plead-3} and
assuming by a continuity argument that the coordinates of $\bu \in
((0,\sqrt{\rho})^N)^T$ are strictly increasing (and slightly abusing
notation for generalized Vandermonde determinants):
\begin{align*}
\sum_{j=0}^{N-1} \frac{(\det \bu^{\circ \bn_j})^2}{c_{n_j}(\det
\bu^{\circ \bn})^2} = \sum_{j=0}^{N-1} \frac{(\det \bv^{\circ
\bm_j})^2}{c_{n_j}(\det \bv^{\circ \bm})^2} \leqslant &\ \sum_{j=0}^{N-1}
\frac{\left(\frac{V(\bm_j)}{V(\bn_{\min})} \bv^{\bm_j -
\bn_{\min}}\right)^2}{c_{n_j} (\bv^{\bm - \bn_{\min}})^2}\\
\leqslant &\ \delta_{\bn,M}^{-N(N-1)} \sum_{j=0}^{N-1}
\frac{V(\bn_j)^2}{V(\bn_{\min})^2} \frac{\rho^{M - n_j}}{c_{n_j}},
\end{align*}
and this is precisely $\mathcal{K}$ by~\eqref{k2def}.
\end{proof}

As in the case of integer powers, we now (mildly modify the above
threshold to) extend Proposition~\ref{Crank3} to all matrices in
$\bp_N((0,\rho))$.

\begin{theorem}
Let \thinspace the \thinspace notation \thinspace be \thinspace as
\thinspace in \thinspace Proposition~\ref{Crank3}. Define \ $\bc
\coloneqq (c_{n_0}, \dots, c_{n_{N-1}})$ and
\begin{equation}\label{k3def}
 \mathcal{K}_{\bn,\bc,M} \coloneqq
  \max(1,\delta_{\bn,M}^{-N(N-1)})
 \sum_{j=0}^{N-1} \prod_{\alpha=0}^{j-1}
 \max(1,g(\bn,\alpha)) \cdot \frac{V(\bn_j)^2}{V(\bn_{\min})^2}
 \frac{\rho^{M - n_j}}{c_{n_j}}, 
\end{equation}
where the empty product in the $j=0$ summand is taken to be
$1$, and
\begin{equation}
g(\bn,\alpha) \coloneqq
\frac{(N-1-\alpha)!^2}{\prod_{k = \alpha+1}^{N-1} (n_k - n_\alpha)^2}.
\end{equation}
Suppose further that $n_j \in \Z^{\geqslant 0} \cup [N-2,\infty)$ for all
$j$. Then the function
\[
x \mapsto \mathcal{K}_{\bn,\bc,M} (c_{n_0} x^{n_0} + \cdots + c_{n_{N-1}}
x^{n_{N-1}}) - x^{M}
\]
is entrywise positivity preserving on $\bp_N((0,\rho))$.
\end{theorem}

Notice that the constant $\mathcal{K}_{\bn, \bc, M}$
specializes to the one in Theorem~\ref{Tmain} (i.e.,~\eqref{k1def}) when
the $n_j$ and $M$ are integers.

\begin{proof}
The proof is by induction on $N$, with the $N=1$ case a consequence of
Proposition~\ref{Crank3}. For the induction step, we apply
Theorem~\ref{Tfitzhorn} with $h(x) = \mathcal{K}_{\bn,\bc,M}
\sum_{j=0}^{N-1} c_{n_j} x^{n_j} - x^M$. Akin to the proof of
Theorem~\ref{Tmain}, define
\[
\bn' \coloneqq (n_1 - 1, \dots, n_{N-1} - 1), \qquad
\bc' \coloneqq (n_1 c_{n_1}, \dots, n_{N-1} c_{n_{N-1}}).
\]
Now note that since $\mathcal{K}_{\bn',\bc',M-1} \sum_{j=1}^{N-1} n_j
c_{n_j} x^{n_j-1} - x^{M-1}$ preserves positivity on
$\bp_{N-1}((0,\rho))$, hence so does $h'(x)$ in view of~\cite{FitzHorn},
if we can show (akin to Theorem~\ref{Tmain}) that
\[
\mathcal{K}_{\bn,\bc,M} \geqslant M \, \mathcal{K}_{\bn',\bc',M-1}.
\]

\noindent To verify this, noting that  $0 \leqslant
\delta_{\bn,M} \leqslant \delta_{\bn',M-1}$, we compute:
\begin{align*}
&\ \mathcal{K}_{\bn,\bc,M}\\
\geqslant &\ \max(1,\delta_{\bn',M-1}^{-N(N-1)})
\sum_{j=0}^{N-1} \prod_{\alpha=0}^{j-1} \max(1,g(\bn,\alpha))
\frac{V(\bn_j)^2}{V(\bn_{\min})^2} \frac{\rho^{M - n_j}}{c_{n_j}}\\
\geqslant &\ \max(1,\delta_{\bn',M-1}^{-(N-1)(N-2)})
\sum_{j=1}^{N-1} g(\bn,0) \prod_{\alpha=1}^{j-1} \max(1,g(\bn,\alpha))
\frac{V(\bn_j)^2}{V(\bn_{\min})^2} \frac{\rho^{M - n_j}}{c_{n_j}}\\
> &\ \max(1,\delta_{\bn',M-1}^{-(N-1)(N-2)})
\sum_{j=1}^{N-1} \prod_{\alpha=1}^{j-1} \max(1,g(\bn,\alpha))
\frac{V(\bn'_j)^2}{V(\bn'_{\min})^2} \frac{M \rho^{M - n_j}}{n_j
c_{n_j}}\\
= &\ M \, \mathcal{K}_{\bn',\bc',M-1},
\end{align*}

\noindent as desired, where the final inequality follows from the fact
that
\begin{equation}\label{Einequality}
\frac{(M - n_0)^2}{(n_j - n_0)^2} > \frac{M}{n_j}, \qquad \forall 0
\leqslant n_0 < n_j < M.
\end{equation}

Finally, that $h[-]$ preserves positivity on $\bp_N^1((0,\rho))$ follows
from Proposition~\ref{Crank3}, since $\mathcal{K}_{\bn, \bc, M}$
dominates the bound in~\eqref{k2def}. Therefore we are done by
Theorem~\ref{Tfitzhorn}.
\end{proof}

As mentioned above, a pleasing consequence of the preceding result is to
obtain explicit threshold bounds on Laplace transforms of real measures.
We thus conclude the section by showing

\begin{proof}[Proof of Theorem~\ref{Tlaplace}]
Akin to Corollary~\ref{Tanalytic}, by the preceding result it suffices by
Fubini's theorem (and discarding the negative components of the measure)
to show the finiteness of the expression
\[
\int_{n_{N-1} + \varepsilon}^\infty \mathcal{K}_{\bn,\bc,M}\ d \mu_+(M),
\]
where $\mu_+$ is the positive part of $\mu$.  Also, by a limiting
argument and adjusting $\rho$ and $\varepsilon$ as necessary, we may
assume that
\[
\int_{n_{N-1} + \varepsilon}^\infty \rho^M (1+\varepsilon)^M\ d \mu_+(M)
< \infty.
\]
By~\eqref{k3def}, it thus suffices to show the finiteness of 
\[
\sup_{M \geqslant n_{N-1} + \varepsilon}
\frac{\max(1,\delta_{\bn,M}^{-N(N-1)})}{(1 + \varepsilon)^M}
\sum_{j=0}^{N-1} \prod_{\alpha=0}^{j-1} \max(1,g(\bn,\alpha)) \cdot
\frac{V(\bn_j)^2}{V(\bn_{\min})^2} \frac{\rho^{- n_j}}{c_{n_j}}.
\]
But as $M$ varies, $\delta_{\bn,M}$ is bounded away from zero, $V(\bn_j)$
grows polynomially in $M$, and $V(\bn_{\min})$, 
$g(\bn,\alpha)$, $\rho^{-n_j}$, and $c_{n_j}$ do not depend on $M$.  The
claim follows.
\end{proof}

\section{Two-sided domains: complete homogeneous symmetric
polynomials}\label{twoside}

We now address the extent to which the above results continue to hold if
we work with a two-sided domain, i.e.,
$\bp_N((-\rho,\rho))$ instead of $\bp_N((0,\rho))$.  For this we must
return to the case of natural number exponents, since exponentiation to
fractional powers is problematic when the base is negative.

On the one hand, we have the trivial observation (using the Schur product
theorem) that if $f: [0,\rho^2) \to \R$ is entrywise positivity
preserving on $\bp_N([0,\rho^2))$, then the map $x \mapsto f(x^2)$ is
entrywise positivity preserving on $\bp_N((-\rho,\rho))$.  By combining
this with the results of the preceding sections, we can obtain a number
of polynomials or power series with some negative coefficients that are
entrywise positivity preserving on $\bp_N((-\rho,\rho))$ or even on all
of $\bp_N$.

On the other hand, we have a new necessary condition:

\begin{lemma}
Let $0 \leqslant n_0 < n_1 < \dots < n_{N-1} < M$ be integers, and let
$0 < \rho \leqslant +\infty$.  Suppose that there exists a polynomial
\[
x \mapsto c_{n_0} x^{n_0} + \dots + c_{n_{N-1}} x^{n_{N-1}} + c_M x^M
\]
with $c_{n_0},\dots,c_{n_{N-1}}$ positive and $c_M$ negative, which is
entrywise positivity preserving on $\bp_N^1((-\rho,\rho))$.  Then
whenever $\bu \in (\R^N)^T$ is a root of $s_\bn$, it is also a root of
$s_{\bn_j}$ for every $j=0,\dots,N-1$, where $\bn \coloneqq
(n_0,\dots,n_{N-1})$ and $\bn_j \coloneqq
(n_0,\dots,n_{j-1},n_{j+1},\dots,n_{N-1},M)$.
\end{lemma}

\begin{proof}
Suppose for contradiction that there existed $\bu$ such that
$s_\bn(\bu)=0$ but $s_{\bn_j}(\bu) \neq 0$ for some $0 \leqslant j
\leqslant N-1$.  By multiplying $\bu$ by a small constant, we may assume
that $\bu$ has coordinates in $(-\sqrt{\rho}, \sqrt{\rho})$, so that $\bu
\bu^T \in \bp_N^1((-\rho,\rho))$.  But from Lemma~\ref{detform} or
equation~\eqref{dhb}, we see that $\det h[\bu \bu^T]$ is negative, giving
the required contradiction.
\end{proof}

Thus, for instance, when $N=2$, no polynomial of the form
\[
x \mapsto c_0 + c_2 x^2 + c_3 x^3
\]
with $c_0,c_2$ positive and $c_3$ negative \thinspace can be entrywise
positivity preserving on $\bp^1_2((-\rho,\rho))$ for any $\rho>0$, since
$s_{(0,2)}(\bu) = u_1 + u_2$ vanishes when $u_2 = -u_1 \neq 0$, but
$s_{(0,3)}(\bu) = u_1^2 + u_1 u_2 + u_2^2$ does not.  This is closely
related to the failure of~\eqref{Egantmacher} when the bases $u_i$ are
allowed to be negative; the point here is that $(u_1^3, u_2^3)$ will not
lie in the span of $(u_1^0, u_2^0)$ and $(u_1^2, u_2^2)$ if $u_2 = -u_1
\neq 0$.

This lemma already rules out analogues of Theorem~\ref{Tmain} on
$\bp_N((-\rho,\rho))$ for most choices of exponents $n_0,\dots,n_{N-1}$,
since typically $s_\bn$ will have a non-trivial zero set which will not
be covered by the zero sets of other $s_{\bn_j}$.  However, there are a
small number of exponents $n_0,\dots,n_{N-1}$ for which a version of this
theorem may be salvaged:

\begin{proposition}\label{Ponlyzero}
Let $\bn = (n_0,\dots,n_{N-1})$ be a tuple of integers $0 \leqslant n_0 <
n_1 < \dots < n_{N-1}$, with the property that the Schur polynomial
$s_\bn$ does not vanish on $\R^N$ except at the origin.  Then for any
integers $h \geqslant 0$ and $M > n_{N-1}$, and any positive constants
$\rho, \ c_{n_0},\dots,c_{n_{N-1}}$, the polynomial
\begin{equation}\label{xt}
 x \mapsto t(  c_{n_0} x^{h+n_0} + \dots + c_{n_{N-1}} 
 x^{h+n_{N-1}} ) - x^{h+M} 
\end{equation}
is entrywise positivity preserving on $\bp_N((-\rho,\rho))$ for $t$
sufficiently large.
\end{proposition}

\begin{proof}
We may assume that $N \geqslant 2$, as the case $N=1$ is trivial;
and by the Schur product theorem we may assume without loss of generality
that $h=0$.
We first verify entrywise positivity preservation on the rank one
matrices $\bp_N^1((-\rho,\rho))$; such matrices take the form $\bu \bu^T$
where $\bu$ has coordinates in~$(-\sqrt{\rho}, \sqrt{\rho})$.  From
Lemma~\ref{detform} and a continuity argument, it suffices to show that
for $t$ sufficiently large, one has
\[
t^N |s_\bn(\bu)|^2 - \sum_{j=0}^{N-1} \frac{t^{N-1}}{c_{n_j}}
|s_{\bn_j}(\bu)|^2 \geqslant 0
\]
for all such $\bu$.  This will follow if we can establish a bound of the
form
\[
|s_{\bn_j}(\bu)| \leqslant C |s_\bn(\bu)|
\]
uniformly for all $\bu \in ([-\sqrt{\rho}, \sqrt{\rho}]^N)^T$, and for
some finite $C$.  The left-hand side has a higher order of homogeneity
than the right-hand side, so it suffices to verify this on the boundary
$\partial [-\sqrt{\rho}, \sqrt{\rho}]^N$.  This is a compact set on
which~$s_\bn(\bu)$ is non-zero by hypothesis, so the claim now follows
from continuity.

To remove the restriction to rank $1$ matrices, we would like to use
Theorem~\ref{Tfitzhorn}.  We first observe that $n_0$ must vanish, since
otherwise $s_\bn(\bu)$ will contain a factor of $\bu^{(1,\dots,1)}$ and
thus vanishes outside of the origin.  From~\eqref{detfactor} (or from the
Young tableaux definition of $s_\bn$) we then observe the identity
\[
s_\bn(u_1,\dots,u_{N-1},0) = (-1)^{N-1}
s_{(n_1-1,\dots,n_{N-1}-1)}(u_1,\dots,u_{N-1}),
\]
and hence $s_{(n_1-1,\dots,n_{N-1}-1)}$ is also non-vanishing on
$\R^{N-1}$ except at the origin.  By the induction hypothesis, we now
conclude that the derivative of~\eqref{xt} is entrywise positivity
preserving on $\bp_{N-1}((-\rho,\rho))$ for $t$ sufficiently large, and
the claim now follows from Theorem~\ref{Tfitzhorn}.
\end{proof}

In fact it is possible to identify the integer tuples $\bn$
satisfying the hypothesis in Proposition~\ref{Ponlyzero}, and these yield
a well-known family of symmetric functions:

\begin{proposition}\label{Phunter}
Given integers $N \geqslant 1$ and $0 \leqslant n_0 < \cdots < n_{N-1}$,
the following are equivalent:
\begin{enumerate}
\item The Schur polynomial $s_\bn$ does not vanish on $\R^N$ except at
the origin.

\item We have: $n_0 = 0, \dots, n_{N-2} = N-2$, and $n_{N-1} - (N-1)$ is
even, say $2r$ for $r \in \Z^{\geqslant 0}$. In other
words, $s_\bn(\bu)$ is precisely the complete homogeneous symmetric
polynomial (of even degree)
\[
h_{2r}(\bu) \coloneqq \sum_{1 \leqslant i_1 \leqslant \cdots \leqslant
i_{2r} \leqslant N} u_{i_1} \cdots u_{i_{2r}}.
\]
\end{enumerate}
\end{proposition}

\begin{proof}
If (1) holds, then the argument in the above proof of
Proposition~\ref{Ponlyzero} shows $n_0 = 0$ and $(n_1 - 1, \dots, n_{N-1}
- 1)$ satisfies the same property for real $(N-1)$-tuples. This reduces
the problem to $N=2$, in which case the assertion is easily verified.
(Alternatively, one can evaluate $s_\bn$ at the basis vector
$(1,0,\dots,0)$ and observe using~\eqref{sdef} that this vanishes unless
$n_j=j$ for $j=0,\dots,N-2$.) Conversely, that $s_\bn(\bu) = h_{2r}(\bu)$
follows from definition; now the proof is completed using the inequality
\[
h_{2r}(\bu) \geqslant \frac{\| \bu \|^{2r}}{2^r r!}
\]
proven by Hunter~\cite{hunter} for all integers $r \geqslant 0$.
\end{proof}

\begin{remark} 
We \ were \ made \ aware \ of \thinspace the \thinspace following
\thinspace alternate \thinspace proof \thinspace of $(2) \implies (1)$
using \thinspace the \thinspace method \thinspace of \thinspace moments;
\thinspace see \thinspace an \thinspace anonymous \thinspace comment
\thinspace on \
\href{https://terrytao.wordpress.com/2017/08/06/schur-convexity-and-positive-definiteness-of-the-even-degree-complete-homogeneous-symmetric-polynomials/}{\tt
https://terrytao.wordpress.com/2017/08/06}, or \cite[Lemma
3.1]{Barvinok}. Namely, given i.i.d.~exponential$(1)$ random variables
$Z_1, \dots, Z_N$, we have
\begin{equation}\label{Eprob}
k! \, h_k(u_1, \dots, u_N) = \mathbb{E} \left[ ( u_1 Z_1 + \cdots + u_N Z_N
)^k \right] \qquad \forall u_1, \dots, u_N \in \R
\end{equation}
for any $k \geqslant 0$; whence $h_{2r}(u_1,\dots,u_N) \geqslant 0$ for
any integer $r \geqslant 0$ and any $u_1,\dots,u_N$, and equality holds
if and only if $u_1 = \cdots = u_N = 0$.

An \textit{a priori} different proof is to obtain a sum-of-squares
decomposition of $h_{2r}$. For low values of $r$, we have:
\begin{align*}
h_0(\bu) = &\ 1,\\
h_2(\bu) = &\ \frac{1}{2} (h_1(\bu)^2 + p_2(\bu)),\\
h_4(\bu) = &\ \frac{1}{72} \left( 3 h_1(\bu)^4 + 2 \sum_i u_i^2(2 u_i +
3 h_1(\bu))^2 + 9 p_2(\bu)^2 + 10 p_4(\bu) \right),
\end{align*}

\noindent where $p_r(\bu) \coloneqq \sum_i u_i^r$ are the power-sum
symmetric polynomials in the tuple $\bu^T = (u_1, \dots, u_N)$.
As pointed out to us by David Speyer, one way to similarly obtain a
sum-of-squares decomposition for every even $r \geqslant 0$ is to
use~\eqref{Eprob}, replacing the exponential random variable $Z$ by a
discrete one $Y$, which matches moments with $Z$ up to order $2r$. Notice
that the existence of such a discrete variable $Y$ follows from
Caratheodory's theorem.
\end{remark}

\begin{remark}
The proofs of Propositions~\ref{Ponlyzero} and~\ref{Phunter} lead to an
explicit bound on the threshold $t$ in~\eqref{xt}:
\[
t \geqslant \mathcal{K} \coloneqq 2^r r! \sum_{j=0}^{N-1} V(\bn_j)^2
\frac{(N \rho)^{M - n_j}}{c_{n_j}},
\]
where the $\bn_j$ are as in the proof of Proposition~\ref{Ponlyzero}, and
$r$ is as in Proposition~\ref{Phunter}.
\end{remark}

\section{Complex domains}\label{complex}

We next briefly study matrices in $\bp_N$ with complex entries.
In~\cite{BGKP-fixeddim} it was shown that for every $M \geqslant N$,
positive coefficients $c_0,\dots,c_{N-1}$, and $0 < \rho < +\infty$, the
polynomial
\[
z \mapsto t ( c_0 + c_1 z + \dots + c_{N-1} z^{N-1} ) - z^M
\]
is entrywise positivity preserving on $\bp_N(D(0,\rho))$ for $t$
sufficiently large, where $D(0,\rho)$ denotes the complex disk $\{ z \in
\C: |z| < \rho \}$.  From the Schur product theorem, the same assertion
holds for
\[
z \mapsto t ( c_0 z^h + c_1 z^{h+1} + \dots + c_{N-1} z^{h+N-1} ) -
z^{h+M}
\]
for any integer $h \geqslant 0$. However, such a result cannot hold for
any other set of exponents, at least if one allows $M$ to vary:

\begin{proposition}\label{compl}
Let $N \geqslant 2$ and $0 \leqslant n_0 < \dots < n_{N-1}$ be \thinspace
integers \thinspace with $(n_0,\dots,n_{N-1}) \neq (h,h+1,\dots,h+N-1)$
for any integer $h \geqslant 0$.  Then there exists $M > n_{N-1}$, such
that any polynomial of the form
\begin{equation}\label{Ecomplex}
z \mapsto c_{n_0} z^{n_0} + \dots + c_{n_{N-1}} z^{n_{N-1}} + c_M z^M,
\end{equation}
with $c_{n_0},\dots,c_{n_{N-1}}$ positive and $c_M$ negative, cannot be
entrywise positivity preserving on $\bp^1_N(D(0,\rho))$ for any $\rho >
0$.
\end{proposition}

\begin{proof}  
By Lemma~\ref{detform}, it suffices to show that there exists $M >
n_{N-1}$ and vectors $\bu \in (\C^N)^T$ of arbitrarily small norm such
that
\[
\frac{1}{c_M} |s_\bn(\bu)|^2 + \sum_{j=0}^{N-1} \frac{1}{
c_{n_j}} |s_{\bn_j}(\bu)|^2 < 0.
\]
We see from the definition~\eqref{sdef} that the specialization
\[
s_\bn( 1, 2, \dots, n-1, z) \in \C[z]
\]
is a polynomial that is positive on the positive real axis; as the Young
tableaux appearing in~\eqref{sdef} can have as few as $n_0$ and as many
as $n_{N-1}-N+1$ entries equal to $N$, and $\bn$ is not of the form
$(h,h+1,\dots,h+N-1)$, this polynomial is not a monomial.  From the
fundamental theorem of algebra, we conclude that there exists
$z_0 \in \C \setminus [0,+\infty)$ such that
\[
s_\bn( 1, 2, \dots, n-1, z_0) = 0.
\]
Rescaling, we see that there exist arbitrarily small $\bu \in (\C^N)^T$,
with all coordinates non-zero and distinct, such that $s_\bn(\bu) = 0$;
thus the vectors $\bu^{\circ n_0}, \dots, \bu^{\circ n_{N-1}}$ are
linearly dependent in $(\C^N)^T$.  On the other hand, from Vandermonde
determinants we see that $\bu^{\circ (n_{N-1}+1)},\dots, \bu^{\circ
(n_{N-1}+N)}$ are linearly independent in $(\C^N)^T$.
Thus there must exist some $M$ between $n_{N-1}+1$ and $n_{N-1}+N$ for
which $\bu^{\circ M}$ lies outside the span of $\bu^{\circ n_0}, \dots,
\bu^{\circ n_{N-1}}$, which implies from~\eqref{detfactor} that
$s_{\bn_j}(\bu)$ is non-zero for some~$j$.  The claim follows.
\end{proof}

\begin{remark}
In fact the above proof shows the infinitude of such ``rigid'' powers
$M$; more precisely, for any $\bn$ that is not a shift by
$h \in \Z^{\geqslant 0}$ of $\bn_{\min}$, among any $N$ consecutive
integers in $(n_{N-1}, \infty)$ there is some $M$ such that every
entrywise positivity preserver on $\bp_N(D(0,\rho))$ of the
form~\eqref{Ecomplex} must be absolutely monotonic.
\end{remark}


\section{Quantitative bounds, via Schur positivity}\label{S8}

Having answered Question~\ref{Q1} for integer and real powers, we now
present stronger versions of the above results, as well as several
applications. We begin by proving the quantitative results in
Section~\ref{Squant}. The improvement over estimates in previous sections
comes from understanding the behavior of the functions $s_{\bn_j}(\bu) /
s_\bn(\bu)$ for $\bu \in ((0,\sqrt{\rho})^N)^T$ and \textit{integer}
powers $n_0, \dots, n_{N-1}, M$ (and $j=0,\dots,N-1$). We now show the
following result in a slightly more general setting.

\begin{proposition}\label{Pschur-ratio}
Fix tuples of non-negative integers $0 \leqslant n_0 < \cdots < n_{N-1}$
and $0 \leqslant m_0 < \cdots < m_{N-1}$, such that $n_j \leqslant m_j\
\forall j$. Then the function
\[
f : ((0,+\infty)^N)^T \to \R, \qquad
f(\bu) \coloneqq \frac{s_\bm(\bu)}{s_\bn(\bu)}
\]
is non-decreasing in each coordinate.
\end{proposition}

In fact this result is the analytical shadow of a deeper algebraic
Schur-positivity phenomenon; see Theorem~\ref{Tschur-ratio}.
Proposition~\ref{Pschur-ratio} was also independently observed by Rachid
Ait-Haddou; see the remarks following Corollary~\ref{Creal} for more
details.

It turns out that Proposition~\ref{Pschur-ratio} for integer powers
suffices to derive exact thresholds for negative coefficients in the case
of integer exponents. In fact, we also show below that it easily implies
exact thresholds for all non-negative real exponents as well.

\subsection{Exact thresholds over bounded domains}

Before proceeding to the proof of Proposition~\ref{Pschur-ratio}, in this
subsection we use it to prove another main result above --
Theorem~\ref{Treal-rank1}, which obtains a sharp value for the threshold
for rank-one matrices, and all non-negative real powers.

\begin{proof}[Proof of Theorem~\ref{Treal-rank1}]
First notice that by taking the entrywise product with $(\bu
\bu^T)^{\circ \pm n_0}$ for $\bu \in ((0,\sqrt{\rho})^N)^T$, it suffices
to show the theorem when the powers $n_j, M \geqslant 0$, by the Schur
product theorem.

Now the proof for non-negative powers is in steps, starting with proving
the result for integer powers $n_0, \dots, n_{N-1}, M \geqslant 0$. In
this case, define $p_t(x) \coloneqq t \sum_{j=0}^{N-1} c_{n_j} x^{n_j} -
x^M$, and
\begin{equation}
(0,\sqrt{\rho})^N_{\neq} \coloneqq \{ (u_1, \dots, u_N) \in
(0,\sqrt{\rho})^N : V(\bu) \neq 0 \},
\end{equation}
where the non-vanishing of $V(\bu)$ simply means that the tuple
$(u_1, \dots, u_N)$ has pairwise distinct coordinates.

That $(1) \implies (2)$ is shown under the weaker assumption
that $\det p_t[ \bu \bu^T ]$ $\geqslant 0$ for $\bu \in
((0,\sqrt{\rho})^N_{\neq})^T$. Under this assumption, we see
from~\eqref{tal} that
\[
 t \geqslant \sup_{\bu \in ((0,\sqrt{\rho})^N_{\neq})^T }
 \sum_{j=0}^{N-1} \frac{s_{\bn_j}(\bu)^2}{c_{n_j} s_{\bn}(\bu)^2}.
\]
But now by Proposition~\ref{Pschur-ratio}, this supremum is attained as
all $u_j \to \sqrt{\rho}{}^-$, and equals precisely the expression
in~\eqref{Esharp}. In fact, note that since the highest power in $\bn_j$
is strictly larger than the highest power in $\bn$ (namely, $M >
n_{N-1}$), it follows by Theorem~\ref{Tschur-ratio} below that the
supremum is uniquely attained as $u_j \to \sqrt{\rho}{}^-$. This is
because the leading $u_1$-term of
\[
s_\bn \cdot \partial_{u_1}(s_{\bn_j}) - 
s_{\bn_j} \cdot \partial_{u_1}(s_\bn)
\]
is nonzero, so that the ratio $\frac{s_{\bn_j}(\bu)}{s_\bn(\bu)}$ is
strictly increasing along each coordinate.

That $(2) \implies (1)$ is immediate if $c_{n_j}, c' \geqslant 0$, while
if $c' < 0 < c_{n_j}$ and $M > n_j$ then we repeat the proof of
Proposition~\ref{ab}. This concludes the proof for non-negative integer
powers.

We next show the result when $M, n_j$ are all rational and non-negative.
Choose a large integer $L \in \N$ such that $LM, L n_j \in \Z$. Repeating
the proof of Theorem~\ref{Treal-rank1} for $p_t[{\bf y} {\bf y}^T]$,
where $p_t(y) \coloneqq t \sum_{j=0}^{N-1} c_{n_j} y^{L n_j} - y^{LM}$
and ${\bf y} \coloneqq \bu^{\circ 1/L} \in ((0, \sqrt[2L]{\rho})^N)^T$,
we obtain the threshold
\[
\mathcal{C} = \sum_{j=0}^{N-1} \frac{V(L \bn_j)^2}{V(L \bn)^2}
\frac{\rho^{M - n_j}}{c_{n_j}} = \sum_{j=0}^{N-1}
\frac{V(\bn_j)^2}{V(\bn)^2} \frac{\rho^{M - n_j}}{c_{n_j}},
\]
as desired.

We now claim the equivalence of the two statements in
Theorem~\ref{Treal-rank1} holds for all non-negative real powers $n_j,
M$. In other words, if we define $p_t(x) \coloneqq t \sum_{j=0}^{N-1}
c_{n_j} x^{n_j} - x^M$, then
\begin{equation}\label{Erank1}
p_t[ \bu \bu^T ] \in \bp_N, \qquad \forall \bu \in ((0,\sqrt{\rho})^N)^T,
\end{equation}

\noindent if and only if $M > n_j\ \forall j$ and $\displaystyle t
\geqslant \sum_{j=0}^{N-1} \frac{V(\bn_j)^2}{V(\bn)^2} \frac{\rho^{M -
n_j}}{c_{n_j}}$.

To prove one implication, suppose
\[
M > n_j\ \forall j \qquad \text{and} \qquad t > \sum_{j=0}^{N-1}
\frac{V(\bn_j)^2}{V(\bn)^2} \frac{\rho^{M - n_j}}{c_{n_j}}.
\]
By continuity, choose strictly decreasing rational sequences $n_{j,k} \to
n_j$ and $M_k \to M$, such that $0 \leqslant n_{0,k} < \cdots < n_{N-1,k}
< M_k$ for all $k$, and
\[
\sum_{j=0}^{N-1} \frac{V(\bn_{j,k})^2}{V(\bn'_k)^2} \frac{\rho^{M_k -
n_{j,k}} }{c_{n_j}} < t,
\]
where
\[
\bn'_k \coloneqq (n_{0,k}, \dots, n_{N-1,k}), \quad \bn_{j,k} \coloneqq
(n_{0,k}, \dots, n_{j-1,k}, n_{j+1,k}, \cdots, n_{N-1,k}, M_k).
\]

Now define $p_{t,k}(x) \coloneqq t \sum_{j=0}^{N-1} c_{n_j} x^{n_{j,k}} -
x^{M_k}$. By the result for rational powers, it follows that $p_{t,k}[\bu
\bu^T] \in \bp_N$ for all $\bu \in ((0,\sqrt{\rho})^N)^T$ and all $k$.
Taking the limit as $k \to \infty$ proves one implication for the real
powers $n_j, M$ (by continuity in $t$).

Finally, we prove the converse implication. Suppose~\eqref{Erank1} holds;
given $\rho>0$ and $\epsilon \in (0,1)$, define the vector
\begin{equation}
\bu_\epsilon \coloneqq \sqrt{\rho \epsilon} \cdot \bu(\epsilon) \coloneqq
\sqrt{\rho \epsilon} (1, \epsilon, \dots, \epsilon^{N-1})^T.
\end{equation}

\noindent Consider the threshold function (for single rank-one matrices,
as in Propositions~\ref{ab} and~\ref{Crank3}):
\begin{equation}\label{Efunction}
t = t(\epsilon, \bn, M) \coloneqq \sum_{j=0}^{N-1} \frac{(\det
\bu_\epsilon^{\circ \bn_j})^2}{c_{n_j} (\det \bu_\epsilon^{\circ
\bn})^2}.
\end{equation}

Suppose $p_t(x) \coloneqq t \sum_{j=0}^{N-1} c_{n_j} x^{n_j} - x^M$ for
real powers $0 \leqslant n_0 < \cdots < n_{N-1} < M$, and $p_t[\bu \bu^T]
\in \bp_N$ for $\bu \in ((0,\sqrt{\rho})^N)^T$.
Following the proof of Proposition~\ref{Crank3} for the matrices
$\bu_\epsilon \bu_\epsilon^T$, we obtain:
\[
t \geqslant \lim_{\epsilon \to 1^-} t(\epsilon, \bn, M),
\]
and hence it suffices to prove that
\begin{equation}\label{Etoprove}
\lim_{\epsilon \to 1^-} t(\epsilon, \bn, M) =
\sum_{j=0}^{N-1} \frac{V(\bn_j)^2}{V(\bn)^2} \frac{\rho^{M -
n_j}}{c_{n_j}}.
\end{equation}
For this, we compute using Proposition~\ref{Pprincipal} that
\[
\frac{(\det \bu_\epsilon^{\circ \bn_j})^2}{c_{n_j} (\det
\bu_\epsilon^{\circ \bn})^2} = \frac{(\rho \epsilon)^{M - n_j}}{c_{n_j}}
\prod_{1 \leqslant a < b \leqslant N} \frac{(\epsilon^{\bn_j(b)} -
\epsilon^{\bn_j(a)})^2}{(\epsilon^{\bn(b)} - \epsilon^{\bn(a)})^2},
\]
where $\bn(a)$ denotes the $a$th coordinate of $\bn$. Now summing over
$j$ and taking $\epsilon \to 1^-$ yields~\eqref{Etoprove}, as desired.
\end{proof}

Akin to previous sections, Theorem~\ref{Treal-rank1} for rank-one
matrices extends to Theorem~\ref{Treal-rank2} for all matrices in
$\bp_N((0,\rho))$, with the same threshold:

\begin{proof}[Proof of Theorem~\ref{Treal-rank2}]
Clearly $(3) \implies (1)$. Conversely, assuming (2), to show (3) we use
the extension principle from Theorem~\ref{Tfitzhorn}. Repeating the proof
of Theorem~\ref{Tmain}, it suffices to show that
\[
\mathcal{C} \geqslant M \, \tilde{\mathcal{C}},
\]
where $\mathcal{C}$ is as in~\eqref{Esharp}, and $\tilde{\mathcal{C}}$ is
defined like $\mathcal{C}$ but with $N$ replaced by $N-1$,
$n_0,\dots,n_{N-1}$ replaced by $n_1-1,\dots,n_{N-1}-1$, $M$ replaced by
$M-1$, and $c_{n_0},\dots,c_{n_{N-1}}$ replaced by $n_1 c_{n_1}, \dots,
n_{N-1} c_{n_{N-1}}$ respectively. Setting
$\bn' \coloneqq (n_1 - 1, \dots, n_{N-1}-1)$ and 
\[
\bn'_j \coloneqq (n_1 - 1, \dots, n_{j-1}-1, n_{j+1}-1, \dots, n_{N-1}-1,
M-1),\ \quad \forall 0 < j < N,
\]
a straightforward computation as in~\eqref{Einequality}
shows that
\begin{equation}
\frac{V(\bn_j)}{V(\bn)} \geqslant
\frac{V(\bn'_j)}{V(\bn')} \cdot \frac{M - n_0}{n_j - n_0} >
\frac{V(\bn'_j)}{V(\bn')} \cdot \frac{\sqrt{M}}{\sqrt{n_j}}, \qquad
\forall 0 < j < N.
\end{equation}
Using this, it follows that 
\[
\mathcal{C} = \sum_{j=0}^{N-1} \frac{V(\bn_j)^2}{V(\bn)^2} \frac{\rho^{M
- n_j}}{c_{n_j}} > \sum_{j=1}^{N-1} \frac{V(\bn'_j)^2}{V(\bn')^2}
\frac{M}{n_j} \cdot \frac{\rho^{M - n_j}}{c_{n_j}} = M \,
\tilde{\mathcal{C}},
\]
and the proof is complete.
\end{proof}

\begin{remark}
Notice that the sharp quantitative bound in Theorem~\ref{Treal-rank2}
also implies Theorems~\ref{Tmain} and~\ref{Treal-1} in previous sections
(which showed that the Horn--Loewner-type necessary conditions were sharp
for integer or real powers, over bounded domains $I = (0,\rho)$).
However, the tight lower and upper bounds on Schur polynomials and
generalized Vandermonde determinants in Propositions~\ref{Pleading}
and~\ref{plead-3} are not implied.
\end{remark}

\begin{remark}
That the constant $\mathcal{K}_{\bn, \bc, M}$ in~\eqref{k3def} dominates
over the sharp bound \eqref{Esharp} can be directly verified, since
\[
\max(1, \delta_{\bn,M}^{-N(N-1)}) \prod_{\alpha=0}^{j-1}
\max(1,g(\bn,\alpha)) \geqslant \prod_{\alpha=0}^{j-1} g(\bn,\alpha),
\quad \forall j.
\]
\end{remark}

\begin{remark}
Notice that Theorems~\ref{Treal-rank1} and~\ref{Treal-rank2} extend the
main result in previous work~\cite{BGKP-fixeddim} from $(n_0, \dots,
n_{N-1}) = (0, \dots, N-1)$ and integers $M > n_{N-1}$ to all real powers
$n_0, \dots, n_{N-1}, M$. Once again, as discussed above, the methods
in~\cite{BGKP-fixeddim} necessarily fail to work for any tuple $\bn$
other than an integer translate of $\bn_{\min}$ (or even for $\bn =
\bn_{\min}$ and $M \not\in \N$), since the analysis
in~\cite{BGKP-fixeddim} also works for matrices with complex entries.
\end{remark}

\begin{remark}\label{Rhankel}
Suppose all the $c_{n_j}$ are strictly positive, and let $\bu_n$ be a
sequence of vectors in $((0,\sqrt{\rho})^N)^T$ with distinct coordinates
that converges to $(\sqrt{\rho},\dots,\sqrt{\rho})^T$.  Then \thinspace
by \thinspace Proposition~\ref{ab}, the \thinspace condition \thinspace
in \thinspace Theorem~\ref{Treal-rank2}(2) is \thinspace equivalent
\thinspace to $\det f[ \bu_n \bu_n^T] \geqslant 0$ for all $n$.  In a
similar spirit, if $\epsilon_n$ is a sequence of positive numbers going
to zero, and $\det f[ \bu(\epsilon_n) \bu(\epsilon_n)^T] \geqslant 0$ for
all $n$, where $\bu(\epsilon)$ is defined in~\eqref{Evector}, then by the
discussion in \cite[Remark~4.9]{BGKP-hankel}, one can adapt the proof of
Lemma~\ref{Lhorn-real}(i) to conclude that the $c_{n_j}$ are all
non-negative, and either $c$ is non-negative or the $c_{n_j}$ are all
strictly positive.  Combining these observations, we conclude that to
obtain Theorem~\ref{Treal-rank2}(2), one does not need to test positive
semidefiniteness of $f[A]$ for all $A \in \bp_N((0,\rho))$; it suffices
to do so for the two sequences $A = \bu_n \bu_n^T$ and $\bu(\epsilon_n)
\bu(\epsilon_n)^T$.  We do not know if one can simplify this test to
involve only a finite number of matrices, rather than two infinite
sequences of matrices.
\end{remark}

\subsection{Proof of Proposition~\ref{Pschur-ratio} and extension to real
powers}

It remains to prove Proposition~\ref{Pschur-ratio}. Our proof once again
combines analysis with symmetric function theory, via type $A$
representation theory. Namely, to prove the proposition, it suffices via
the quotient rule to show by symmetry that for any fixed $j \in [1,N]$,
the polynomial
\begin{equation}\label{Equotient}
s_\bn \cdot \partial_{u_j}(s_\bm) - 
s_\bm \cdot \partial_{u_j}(s_\bn)
\end{equation}
is non-negatively valued on $(0,+\infty)^N$. This is guaranteed if the
expression~\eqref{Equotient} is a linear combination of monomials with
positive (integer) coefficients, i.e.~monomial-positive. We will show a
stronger statement. Notice that expanding $s_\bn, s_\bm$ as polynomials
in $u_1$ (by symmetry), the coefficients of each power $u_1^k$ are
skew-Schur functions $s_{\bn / (k)}(u_2,\dots,u_N)$ for $0 \leqslant k
\leqslant n_{N-1} - N + 1$. (See \cite[Chapter I.5]{Macdonald} for more
details on these symmetric polynomials.) Now we claim:

\begin{theorem}\label{Tschur-ratio}
When written as a polynomial in $u_1$, every coefficient in the
expression~\eqref{Equotient} is {\em Schur-positive}, i.e., a
non-negative integer linear combination of Schur polynomials in $u_2,
\dots, u_N$.
\end{theorem}

Note that Schur-positivity immediately implies monomial-positivity, and
Proposition~\ref{Pschur-ratio} is a trivial consequence of the latter.

\begin{proof}
In order to make the notation compatible with that in~\cite{LPP}, we now
index Schur polynomials $s_\bn$ by the partition
\[
\overline{\bn - \bn_{\min}} {= (n_{N-1} - (N-1), \dots, n_1 - 1, n_0)};
\]
more precisely we write
\[
s_\bn = \tilde s_{\overline{\bn - \bn_{\min}} }.
\]
We also say that a shape $\nu$ \textit{contains} a shape $\lambda$ if
$\nu_j \geqslant \lambda_j$ for $0 \leqslant j \leqslant N-1$. Now
setting $\lambda \coloneqq \overline{\bn - \bn_{\min}}$ and $\nu
\coloneqq \overline{\bm - \bn_{\min}}$, we see that $\nu$ contains
$\lambda$, and our task is to show that the coefficient of each power of
$u_j$ in
\begin{equation}\label{Eq2}
 \tilde s_\lambda \cdot \partial_{u_j}(\tilde s_\nu) -
 \tilde s_\nu \cdot \partial_{u_j}(\tilde s_\lambda)
\end{equation}
is Schur-positive.

As Schur polynomials are symmetric, we may assume without loss of
generality that $j=1$. We have the well-known expansion of Schur
polynomials \cite[Chapter I, Equation (5.12)]{Macdonald} into skew-Schur
polynomials
\[
\tilde s_\lambda(\bu) = \sum_{j \geqslant 0} \tilde
s_{\lambda/(j,0,\dots,0)}(\bu') u_1^j
\]
where $\bu' \coloneqq (u_2,\dots,u_N)$, where $\tilde s_{\lambda/\mu}$ is
the skew-Schur polynomial corresponding to the shape $\lambda/\mu$ if
$\lambda$ contains $\mu$, and we adopt the convention that $\tilde
s_{\lambda/\mu}=0$ if $\lambda$ does not contain $\mu$. Note that only
finitely many of the terms will be non-zero. Similarly we have
\[
\tilde s_\nu(\bu) = \sum_{k \geqslant 0} \tilde
s_{\nu/(k,0,\dots,0)}(\bu') u_1^k
\]
and hence we may write~\eqref{Eq2} as
\[
\sum_{j,k \geqslant 0} \tilde s_{\lambda/(j,0,\dots,0)}(\bu') \tilde
s_{\nu/(k,0,\dots,0)}(\bu') (k-j) u_1^{j+k-1}.
\]
Symmetrizing, this can be rewritten as
\begin{align*}
\sum_{k > j \geqslant 0} (\tilde s_{\lambda/(j,0,\dots,0)}(\bu') \tilde
s_{\nu/(k,0,\dots,0)}(\bu') - \tilde s_{\lambda/(k,0,\dots,0)}(\bu')
\tilde s_{\nu/(j,0,\dots,0)}(\bu')) &\ \times\\
\times (k-j) & u_1^{j+k-1}.
\end{align*}

\noindent Thus it will suffice to show that the polynomial
\[
\tilde s_{\lambda/(j,0,\dots,0)} \tilde s_{\nu/(k,0,\dots,0)} -
\tilde s_{\lambda/(k,0,\dots,0)} \tilde s_{\nu/(j,0,\dots,0)}
\]
is Schur-positive whenever $k>j$ and $\nu$ contains $\lambda$.
But this is a special case of a result \cite[Theorem 4]{LPP} of Lam,
Postnikov, and Pylyavskyy (we thank Pavlo Pylyavskyy for this reference).
These authors establish in \textit{loc.~cit.}\ the more general claim
that
\begin{equation}\label{Elpp}
\tilde s_{\lambda \wedge \nu / \mu \wedge \rho} \tilde s_{\lambda \vee
\nu / \mu \vee \rho} - \tilde s_{\lambda / \mu} \tilde s_{\nu / \rho}
\end{equation}

\noindent is Schur-positive for any partitions $\lambda,\nu, \mu, \rho$,
where we write
\begin{align*}
\lambda \wedge \nu \coloneqq &\ (\min(\lambda_{N-1}, \nu_{N-1}), \dots,
\min(\lambda_0, \nu_0)), \\
\lambda \vee \nu \coloneqq &\ (\max(\lambda_{N-1}, \nu_{N-1}), \dots,
\max(\lambda_0, \nu_0))
\end{align*}
if $\lambda = (\lambda_{N-1}, \dots,\lambda_0)$, and similarly for $\nu,
\mu, \rho$.  Note that the Littlewood--Richardson rule \cite[Chapter I,
Equations (5.2), (5.3)]{Macdonald} implies that all skew-Schur
polynomials and their products are Schur-positive, so the
Schur-positivity of~\eqref{Elpp} is only non-trivial when $\tilde
s_{\lambda/\mu}$ and $\tilde s_{\nu/\rho}$ are non-zero, thus $\lambda$
contains $\mu$ and $\nu$ contains $\rho$. This implies that $\lambda
\wedge \nu$ contains $\mu \wedge \rho$ and $\lambda \vee \nu$ contains
$\mu \vee \rho$.  This is the case that is treated in~\cite{LPP}.
\end{proof}

In addition to the proof of the above Theorems~\ref{Treal-rank1},
\ref{Treal-rank2}, we now list some other applications of
Proposition~\ref{Pschur-ratio}. The first is that it can be extended to
hold for all real powers:

\begin{corollary}\label{Creal}
Fix tuples of real powers $\bn = (n_0 < \cdots < n_{N-1})$ and
$\bm = (m_0 < \cdots < m_{N-1})$, such that $n_j \leqslant m_j\ \forall
j$. Defining $\bu^{\circ \bn} \coloneqq (u_j^{n_{k-1}})_{j,k=1}^N$ as
above, the function
\[
f : ((0,+\infty)^N_{\neq})^T \to \R, \qquad
f(\bu) \coloneqq \frac{\det \bu^{\circ \bm}}{\det \bu^{\circ \bn}}
\]
is non-decreasing in each coordinate, where $S^N_{\neq}$ for a set $S$
denotes ordered $N$-tuples of pairwise distinct elements.
\end{corollary}

Note that we do not need to assume the $u_j$ to be strictly increasing in
$\bu \in ((0,+\infty)^N_{\neq})^T$, since the ratio of determinants used
to define $f(\bu)$ is invariant with respect to permutations of the
components of $\bu$.

We also remark that this corollary allows us to bypass
Theorem~\ref{Treal-rank1} (for integer powers) and prove
Theorem~\ref{Treal-rank2} for tuples of real powers via a slightly
different approach.

Rachid Ait-Haddou has mentioned to us (see the comments at\hfill\break
\href{https://terrytao.wordpress.com/2017/08/17}{\tt
https://terrytao.wordpress.com/2017/08/17}) that Corollary~\ref{Creal}
can be proved directly using the theory of Chebyshev blossoming in
M\"untz spaces \cite{Blossom2,Blossom1}.  We give a further proof of
Corollary~\ref{Creal}, relying primarily on determinant identities such
as Dodgson condensation, in Section~\ref{logmod}.

\begin{proof}[Proof of Corollary~\ref{Creal}]
As above, we first observe that it suffices to assume $n_j, m_j$
$\geqslant 0$, since multiplying and dividing $f(\bu)$ by $(u_1 \cdots
u_N)^{-n_0}$ shows that
\[
f(\bu) = \frac{\det \bu^{\circ \bm'}}{\det \bu^{\circ \bn'}},
\]
where
\[
\bm' \coloneqq (m_0 - n_0, \dots, m_{N-1} - n_0), \quad
\bn' \coloneqq (n_0 - n_0, \dots, n_{N-1} - n_0).
\]

Thus we may assume $n_j, m_j$ are non-negative. We now show the result
for non-negative rational powers $\bm, \bn$.
Choose $L \in \N$ such that $L n_j, L m_j \in \Z$ for all $j$, and set
${\bf y} \coloneqq \bu^{\circ 1/L}$ as above.
Notice that $\bu \in ((0,+\infty)^N_{\neq})^T$ if and only if
${\bf y} \in ((0,+\infty)^N_{\neq})^T$. Now we compute:
\[
f(\bu) = \frac{s_{L \bm}({\bf y})}{s_{L \bn}({\bf y})},
\]
and by Proposition~\ref{Pschur-ratio}, this is non-decreasing in each
coordinate $y_j$, whence in $u_j$.

Next, given non-negative real powers $n_j, m_j$, as above we choose
rational sequences $n_{j,k} \to n_j$ and $m_{j,k} \to m_k$, such that
\[
0 \leqslant n_{0,k} < \cdots < n_{N-1,k}, \quad
m_{0,k} < \cdots < m_{N-1,k}, \quad
n_{j,k} \leqslant m_{j,k} \ \forall j, \quad
\forall k \in \N.
\]
Define $\bn_k \coloneqq (n_{0,k}, \dots, n_{N-1,k})$ and
similarly $\bm_k$ for $k \in \N$; and now define
\[
f_k(\bu) \coloneqq \frac{\det \bu^{\circ \bm_k}}{\det \bu^{\circ \bn_k}}.
\]
Then the functions $f_k(\bu)$ are all non-decreasing in every coordinate,
whence so is their pointwise limit on $((0,+\infty)^N_{\neq})^T$, namely,
$f$.
\end{proof}

\begin{remark}
Another consequence of Proposition~\ref{Pschur-ratio} and the principal
specialization in Proposition~\ref{Pprincipal} is that for
any (integer or) real tuple $\bm \geqslant \bn$ coordinatewise, we have:
\begin{equation}
\sup_{\bu \in ((0,1)^N_{\neq})^T} \frac{\det \bu^{\circ \bm}}{\det
\bu^{\circ \bn}} = \frac{V(\bm)}{V(\bn)} = \lim_{\epsilon \to 1^-}
\frac{\det \bu(\epsilon)^{\circ \bm}}{\det \bu(\epsilon)^{\circ \bn}}.
\end{equation}

Notice \thinspace by \thinspace homogeneity \thinspace that \thinspace
this \thinspace identity has an obvious variant for $\bu \in$ \
$((0,\sqrt{\rho})^N_{\neq})^T$ for any $0 < \rho < +\infty$. In turn,
this variant leads to the following ``depolarization-type'' phenomenon:
for all $\bu \in ((0,+\infty)^N_{\neq})^T$, the ratio $s_\bm(\bu) /
s_\bn(\bu)$ equals the value at some diagonal point $(u,\dots,u)^T$,
where $0 < \min_j u_j \leqslant u \leqslant \max_j u_j$.
\end{remark}

\section{Application 1: Entrywise preservers of total non-negativity}

The next few sections are devoted to applications of the above results.
In this section we study entrywise preservers of total non-negativity,
with the ultimate aim of extending the above classification of sign
patterns -- and computation of exact threshold bounds -- to totally
non-negative matrices.

Recall that a real matrix is said to be \textit{totally non-negative}
(resp.~\textit{strictly totally positive}) if every minor is a
non-negative (resp.~positive) real number.
A well-known example of total non-negativity -- in fact, strict total
positivity -- is that of generalized Vandermonde
determinants~\eqref{Egantmacher}. 
Totally non-negative and strictly totally positive matrices feature in
analysis and differential equations, probability and stochastic
processes, representation theory, discrete mathematics, and particle
systems; and are closely connected to positive semidefinite matrices.

\begin{remark}
As the notion of \textit{total positivity} is taken by various authors in
the literature to mean either total non-negativity or strict total
positivity (see \cite{FJS,karlin}), we avoid using the former notion in
the paper, and use the latter two instead.
\end{remark}

In the dimension-free setting, it was recently shown
in~\cite{BGKP-hankel} that, in the spirit of Schoenberg and Rudin's
original theorems, an entrywise function $F : [0,+\infty) \to \R$
preserves total non-negativity on Hankel matrices of all sizes, if and
only if $F$ preserves positivity on the same set, if and only if $F$
agrees with an absolutely monotonic entire function
\[
\sum_{k \geqslant 0} c_k x^k, \ c_k \geqslant 0\ \forall k
\]
on $(0,+\infty)$, and $0 \leqslant F(0) \leqslant \lim_{\epsilon \to 0^+}
F(\epsilon)$. 

\begin{definition}
Given an integer $N>0$ and a domain $I \subset [0,+\infty)$, define
$\HTN_N(I)$ to be the set of Hankel totally non-negative matrices with
entries in~$I$.
\end{definition}

As the above remarks show, Hankel totally non-negative matrices
(which are automatically positive semidefinite) are a ``well-behaved''
test set for preserving total non-negativity in arbitrary dimension -- in
fact, they are closed under taking Schur products. On the other hand, if
one works with the slightly larger class of all positive semidefinite
(equivalently, symmetric) totally non-negative matrices, then any
preserver is necessarily constant or linear. This is true even for $5
\times 5$ matrices; for these and related results, we refer the reader
to~\cite{BGKP-TN}.

In light of these remarks, we now study the preservation of total
non-negativity in the fixed dimension setting, again for Hankel totally
non-negative matrices. It turns out that the above results for positivity
preservation apply to preserving total non-negativity as well; note that
in all cases, the entries of totally non-negative matrices, whence the
domains of their entrywise preservers, are contained in $[0,+\infty)$. We
begin by presenting some of the main results, which will be followed by
proofs.

\begin{theorem}\label{Thtn1}
Fix an integer $N>0$ and real powers $n_0 < \cdots < n_{N-1} < M$.
Also fix real scalars $\rho > 0$ and $c_{n_0}, \dots, c_{n_{N-1}}, c'$,
and define:
\begin{equation}
f(x) \coloneqq \sum_{j=0}^{N-1} c_{n_j} x^{n_j} + c' x^M,
\qquad c_{n_j}, c' \in \R.
\end{equation}

\noindent Then the following are equivalent:
\begin{enumerate}
\item The entrywise map $f[-]$ preserves total non-negativity on rank-one
matrices in $\HTN_N((0,\rho))$.

\item The \thinspace entrywise \thinspace map \thinspace $f[-]$
\thinspace preserves \thinspace positivity \thinspace on \thinspace
rank-one \thinspace matrices \thinspace in $\HTN_N((0,\rho))$.

\item Either all $c_{n_j}, c' \geqslant 0$; or $c_{n_j} > 0\ \forall j$
and $c' \geqslant -\mathcal{C}^{-1}$, where
\begin{equation}
\mathcal{C} = \sum_{j=0}^{N-1} \frac{V(\bn_j)^2}{V(\bn)^2}
\frac{\rho^{M - n_j}}{c_{n_j}}.
\end{equation}
\end{enumerate}
\end{theorem}

In other words, preserving total non-negativity on the test set in
assertion (1) is equivalent to preserving positivity on the same set, and
provides additional equivalent conditions to the ones in
Theorem~\ref{Treal-rank1}.

As shown in~\cite{FJS}, a power $x \mapsto x^\alpha$ entrywise preserves
total non-negativity on $\HTN_N((0,\rho))$ if and only if it preserves
positivity on $\bp_N((0,\rho))$, namely, $\alpha \in \Z^{\geqslant 0}
\cup [N-2,\infty)$. With this constraint in mind, the preceding result
extends to matrices of all ranks:

\begin{theorem}\label{Thtn2}
With notation as in Theorem~\ref{Thtn1}, if we further assume that
$n_j \in \Z^{\geqslant 0} \cup [N-2,\infty)$ for all $j$, then the
conditions (1)--(3) are further equivalent to:
\begin{enumerate}
\setcounter{enumi}{3}
\item The \thinspace entrywise \thinspace map \ $f[-]$
\ preserves \ total \thinspace non-negativity
\thinspace on \thinspace the \thinspace set \ $\HTN_N([0,\rho])$.
\end{enumerate}
\end{theorem}

Moreover, the bound obtained above (and earlier in the paper) is tight
enough to enable bounding more general functions:

\begin{corollary}\label{Chtn-laplace}
Notation as in Theorem~\ref{Thtn1}; assume further that $n_j \in
\Z^{\geqslant 0} \cup [N-2,\infty)$ for all $j$.
Given $\varepsilon > 0$ and a real measure $\mu$ supported on $[n_{N-1} +
\varepsilon,\infty)$ whose Laplace transform $g_\mu$ (defined
in~\eqref{Elaplace}) is absolutely convergent at $\rho$, there exists a
finite threshold
\[
\mathcal{K} = \mathcal{K}(n_0,\dots,n_{N-1}, \rho, c_{n_0}, \dots,
c_{n_{N-1}}, g_\mu) > 0
\]
such that the entrywise function
$\displaystyle x \mapsto \mathcal{K} \sum_{j=0}^{N-1} c_{n_j} x^{n_j} -
g_\mu(x)$
preserves total non-negativity on $\HTN_N((0,\rho))$, i.e.,
\begin{equation}
{\mathcal K} \sum_{j=0}^{N-1} c_{n_j} A^{\circ n_j} - g_\mu[A] \in
\HTN_N([0,+\infty)), \qquad \forall A \in \HTN_N((0,\rho)).
\end{equation}
\end{corollary}

These theorems follow from two results. The first relates total
non-negativity and positive semidefiniteness for Hankel matrices:

\begin{lemma}[{see \cite[Corollary 3.5]{FJS}}]\label{Lhtn}
Let $A_{N \times N}$ be a Hankel matrix. Then $A$ is totally non-negative
if and only if $A$ and its truncation $A^{(1)}$ have non-negative
principal minors. Here, $A^{(1)}$ denotes the submatrix of $A$ with the
first column and last row removed.
\end{lemma}

The second result uses Lemma~\ref{Lhtn} to connect entrywise preservers
of these two notions. Recall below that $\bp_N^k$ comprises all positive
semidefinite $N \times N$ matrices of rank at most $k$.

\begin{proposition}\label{Phtn}
Fix integers $1 \leqslant k \leqslant N$ and a scalar $0 < \rho
\leqslant +\infty$. Suppose $f : [0,\rho) \to \R$ is such that the
entrywise map $f[-]$ preserves positivity on $\bp_N^k([0,\rho))$. Then
$f[-]$ preserves total non-negativity on the set
$\HTN_N([0,\rho)) \cap \bp_N^k([0,\rho))$.
\end{proposition}

\begin{proof}
Given a matrix $A \in \HTN_N([0,\rho)) \cap \bp_N^k([0,\rho))$, by
Lemma~\ref{Lhtn} the padded matrix $A^{(1)} \oplus (0)_{1 \times 1}$
belongs to $\bp_N([0,\rho))$. Moreover, it has vanishing $(k+1) \times
(k+1)$ minors, hence is of rank at most $k$. By assumption, $f[A]$ and
$f[A]^{(1)} = f[A^{(1)}]$ are thus positive semidefinite, whence $f[A]$
is totally non-negative by Lemma~\ref{Lhtn}.
\end{proof}

We now prove the above theorems.

\begin{proof}[Proof of Theorem~\ref{Thtn1}]
Clearly $(1) \implies (2)$. Now observe that the proof of the
``Horn--Loewner-type'' Lemma~\ref{Lhorn-real} goes through if we restrict
to rank-one matrices of the form $\bu(\epsilon) \bu(\epsilon)^T$ for
$\epsilon \in (0,1)$, where $\bu(\epsilon)$ is defined
in~\eqref{Evector}. Moreover, every such matrix $\bu(\epsilon)
\bu(\epsilon)^T$ is in $HTN_N((0,+\infty))$, since all $k \times k$
minors vanish for $k \geqslant 2$.
Thus if $f[-]$ preserves total non-negativity on rank-one matrices in
$HTN_N((0,\rho))$, then either all $c_j, c'$ are non-negative, or $c_j >
0\ \forall j$. In the latter case, the discussion in Remark~\ref{Rhankel}
now implies assertion $(3)$. Finally, if $(3)$ holds then $f[-]$
preserves positivity on $\bp_N^1([0,\rho))$ by Theorem~\ref{Treal-rank1}.
We are now done by Proposition~\ref{Phtn}.
\end{proof}

\begin{proof}[Proof of Theorem~\ref{Thtn2}]
Clearly $(4) \implies (1)$. Conversely, if $(3)$ holds then $f[-]$
preserves positivity on $\bp_N([0,\rho))$ by Theorem~\ref{Treal-rank2}.
We are now done by Proposition~\ref{Phtn} for $k=N$.
\end{proof}

Finally, Corollary~\ref{Chtn-laplace} follows similarly from
Theorem~\ref{Tlaplace} via Proposition~\ref{Phtn}.\medskip

We conclude this part by answering a variant of Question~\ref{Q1}, for
total non-negativity. By \thinspace the Schur product \thinspace theorem
\thinspace and \thinspace Lemma~\ref{Lhtn}, it \thinspace follows
\thinspace that \thinspace absolutely \thinspace monotonic \thinspace
maps $f : (0,\rho) \to \R$ preserve \thinspace total \thinspace
non-negativity on \ $\bigcup_{N \geqslant 1} \HTN_N((0,\rho))$. Now given
a convergent power series $f(x)$ satisfying
\[
f[-] : \HTN_N((0,\rho)) \to \HTN_N([0,+\infty))
\]
for \textit{fixed} dimension $N$, which coefficients of $f$ can be
negative? The following results show that for both bounded and unbounded
domains, the above results on positivity preservers once again extend to
give the same characterizations:

\begin{theorem}
Let $N > 0$ and $0 \leqslant n_0 < n_1 < \cdots < n_{N-1}$ be integers,
and for each $M > n_{N-1}$, let $\epsilon_M \in \{-1,0,+1\}$ be a sign.
Let $0 < \rho < +\infty$, and let $c_{n_0},\dots,c_{n_{N-1}}$ be positive
reals.  Then there exists a convergent power series
\[
f(x) = c_{n_0} x^{n_0} + c_{n_1} x^{n_1} + \dots + c_{n_{N-1}}
x^{n_{N-1}} + \sum_{M > n_{N-1}} c_M x^M
\]
on $(0,\rho)$ that preserves total non-negativity on
$\HTN_N((0,\rho))$ when applied entrywise, such that for each $M >
n_{N-1}$, $c_M$ has the sign of $\epsilon_M$.
\end{theorem}

\begin{theorem}
Let $N > 0$ and $0 \leqslant n_0 < \dots < n_{N-1}$ be integers, and for
each $M > n_{N-1}$, let $\epsilon_M \in \{-1,0,+1\}$ be a sign. Suppose
that whenever $\epsilon_{M_0} = -1$ for some $M_0 > n_{N-1}$, one has
$\epsilon_M = +1$ for at least $N$ choices of $M > M_0$. Let
$c_{n_0},\dots,c_{n_{N-1}}$ be positive reals.  Then there exists a
convergent power series
\[
f(x) = c_{n_0} x^{n_0} + c_{n_1} x^{n_1} + \dots + c_{n_{N-1}}
x^{n_{N-1}} + \sum_{M > n_{N-1}} c_M x^M
\]
on $(0,+\infty)$ that preserves total non-negativity on
$\HTN_N((0,+\infty))$ when applied entrywise, such that for each $M >
n_{N-1}$, $c_M$ has the sign of $\epsilon_M$.
\end{theorem}

In fact, the more general results for sums of \textit{real} powers --
namely, Theorems~\ref{Tpowerseries} and~\ref{Tpowerseries2} -- naturally
extend to preservers of total non-negativity. We emphasize that these
four results were \textit{a priori} nontrivial to formulate and prove;
but with the above analysis in this paper, they follow directly from
their ``positivity'' counterparts, using Proposition~\ref{Phtn}.

\section{Application 2: Characterization of (weak) majorization via Schur
polynomials}

Recall the notion of (weak) majorization: given a real $N$-tuple $\bu =
(u_1, \dots, u_N)$, let $u_{[1]} \geqslant \cdots \geqslant u_{[N]}$
denote the decreasing rearrangement of its coordinates. Now given $\bu,
\bv \in \R^N$, one says $\bu$ \textit{weakly majorizes} $\bv$ -- and
writes $\bu \succ_w \bv$ -- if
\begin{equation}\label{Eweak}
\sum_{j=1}^k u_{[j]} \geqslant \sum_{j=1}^k v_{[j]}, \ \forall 0 < k < N,
\qquad \sum_{j=1}^N u_{[j]} \geqslant \sum_{j=1}^N v_{[j]},
\end{equation}
while $\bu$ \textit{majorizes} $\bv$ if the final inequality above (for
$k=N$) is an equality.

In this part, we apply Proposition~\ref{Pschur-ratio} and
Corollary~\ref{Creal} to obtain a new characterization of weak
majorization that involves Schur polynomials. In particular, the result
also extends a conjecture of Cuttler, Greene, and Skandera
\cite[Conjecture 7.4]{CGS}. The conjecture says that if $\bm$ majorizes
$\bn$, where $\bm, \bn$ are $N$-tuples of (distinct) non-negative
integers, then
\begin{equation}\label{Ecgs}
\frac{s_\bm(\bu)}{s_\bn(\bu)} \geqslant
\frac{s_\bm(1,\dots,1)}{s_\bn(1,\dots,1)},
\qquad \forall \bu \in ((0,+\infty)^N)^T.
\end{equation}
(We thank Rachid Ait-Haddou and Suvrit Sra for correcting the formulation
of this conjecture in a previous version of this manuscript.)
The conjecture has been proved very recently in~\cite{Sra} and
alternately in~\cite{Blossom2} (see Remark~5.1 therein) using the
Harish-Chandra--Itzykson--Zuber formula.

In our setting, first observe that if $\bm$ dominates $\bn$
coordinatewise and $\sum_j m_j > \sum_j n_j$, then by
homogeneity,~\eqref{Ecgs} necessarily cannot hold at e.g.~points $\bu =
\rho (1,\dots,1)^T$ with $\rho \in (0,1)$. However, it does hold at all
points in $([1,+\infty)^N)^T$:
\begin{equation}\label{Ecgs-wmaj}
m_j \geqslant n_j\ \forall j \ \implies \
\frac{s_\bm(\bu)}{s_\bn(\bu)} \geqslant
\frac{s_\bm(1,\dots,1)}{s_\bn(1,\dots,1)} = \frac{V(\bm)}{V(\bn)},
\qquad \forall \bu \in ([1,+\infty)^N)^T,
\end{equation}
since this is a direct consequence of Proposition~\ref{Pschur-ratio}.

Since~\eqref{Ecgs-wmaj} also holds if $\bm$ majorizes $\bn$, it is
natural to try to characterize the pairs of tuples $\bm,\bn$ for
which~\eqref{Ecgs-wmaj} holds. In other words, we seek a unification of
both settings: $\bm$ majorizing $\bn$, and $\bm$ dominating $\bn$
coordinatewise (and restricting to $\bu \in ([1,+\infty)^N)^T$ for
homogeneity reasons).
Observe that all of the cited works \cite{Blossom2,CGS,Sra} require
$\sum_j m_j = \sum_j n_j$. Replacing this equality by an inequality as
in~\eqref{Eweak} allows us to achieve such a unification, and thereby
provide the sought-for extension of (both implications in) the
Cuttler--Greene--Skandera conjecture. In fact, the following result holds
more generally for tuples of real powers -- including negative powers --
and characterizes weak majorization.

\begin{theorem}\label{Tcgs-wmaj}
Suppose $\bm, \bn$ are $N$-tuples of pairwise distinct real powers.
Then we have
\begin{equation}\label{Ecgs-revised}
\frac{|\det \bu^{\circ \bm}|}{|V(\bm)|} \geqslant
\frac{|\det \bu^{\circ \bn}|}{|V(\bn)|},
\qquad \forall \bu \in ([1,+\infty)^N)^T
\end{equation}
if and only if $\bm \succ_w \bn$.
In fact to deduce $\bm \succ_w \bn$, it suffices to work with
countably many vectors $\bu \in ((I_\infty)^N)^T$, where $I_\infty$ is an
arbitrary non-empty deleted neighborhood of $+\infty$ in $[1,+\infty)$.
\end{theorem}

In other words, $I_\infty$ is essentially of the form $(K,+\infty)$ or
$[K,+\infty)$, for some real $K \geqslant 1$. As the proof reveals, there
is sufficient freedom in choosing the aforementioned countable set of
vectors, e.g.~they can be chosen to be rational.

\begin{remark}
As a consequence of Theorem~\ref{Tcgs-wmaj}, we obtain the classification
of the set of pairs $(\bm, \bn)$ that satisfy~\eqref{Ecgs-revised} for
all $\bu \in ((K,+\infty)^N)^T$; and this classification is independent of
$K \geqslant 1$.
\end{remark}

\begin{proof}[Proof of Theorem~\ref{Tcgs-wmaj}]
By multiplying both sides by $(u_1 \cdots u_N)^{-k}$ where $k \coloneqq
\min \{ m_j, n_j : 0 \leqslant j < N \}$, the result reduces to the case
when all powers $m_j, n_j$ are non-negative. While the proof for
non-negative $\bm, \bn$ now follows that in \cite{CGS,Sra}, we include it
for completeness since real powers are involved, leading to some
differences in argument. We may assume $\bm, \bn$ are in increasing
order, and drop all absolute value signs in~\eqref{Ecgs-revised}. Now
suppose~\eqref{Ecgs-revised} holds, and for each $j \in [1,N]$, define
the following partial sums of $\bm, \bn$:
\[
\widetilde{n}_j \coloneqq n_{N-j} + \cdots + n_{N-1}, \qquad
\widetilde{m}_j \coloneqq m_{N-j} + \cdots + m_{N-1}.
\]

Fix scalars $0 < u_1 < \cdots < u_N < +\infty$ in the neighborhood
$I_\infty$, note that $t u_j \in I_\infty$ if $t \geqslant 1$, and now
appeal to our tight bounds in Proposition~\ref{plead-3}, setting
\begin{equation}\label{Ewmaj}
\bu = \bu_j(t) \coloneqq (u_1,\dots,u_{N-j},t u_{N-j+1},\dots, t u_N),
\qquad t \in [1,\infty).
\end{equation}

\noindent More precisely, we multiply all terms in~\eqref{pup-real} by
$\bu^{\bn_{\min}} / V(\bu)$, and compute using~\eqref{Ecgs-revised}:
\begin{align*}
0 \leqslant t^{\widetilde{n}_j} \prod_{k=1}^N u_k^{n_{k-1}}
= \bu^\bn \leqslant &\ \frac{\bu^{\bn_{\min}} }{V(\bu)} \det \bu^{\circ
\bn} \leqslant \frac{\bu^{\bn_{\min}} }{V(\bu)} \frac{V(\bn)}{V(\bm)}
\det \bu^{\circ \bm}\\
\leqslant &\ \frac{V(\bn)}{V(\bn_{\min})} \bu^\bm = t^{\widetilde{m}_j}
\frac{V(\bn)}{V(\bn_{\min})} \prod_{k=1}^N u_k^{m_{k-1}},
\qquad \forall t \geqslant 1.
\end{align*}
Taking a countable sequence $t = t_l \to \infty$ (for each $j \in [1,N]$)
shows that the growth rate $\widetilde{n}_j$ of the initial expression
must be exceeded by that of the final expression, which is
$\widetilde{m}_j$. It follows that $\bm$ weakly majorizes $\bn$.

Conversely, suppose $\bm$ weakly majorizes $\bn$ (with all non-negative
and possibly non-pairwise-distinct real coordinates). We then claim that
$F_\bu(\bm) \geqslant F_\bu(\bn)$, where $F_\bu : ([0,+\infty))^N \to \R$
is the Harish-Chandra--Itzykson--Zuber integral~\eqref{hciz}:
\begin{align}\label{Eweakmaj}
F_\bu(&\bm)  \\
\coloneqq &\ \int_{U(N)} \exp \mathrm{tr} \left(
\mathrm{diag}(m_0,\dots,m_{N-1}) U
\mathrm{diag}(\log(u_1),\dots,\log(u_N)) U^* \right) \ dU.\nonumber
\end{align}

By continuity and a limiting argument, note we only need to
prove~\eqref{Ecgs-revised} for all $\bu \in (1,+\infty)^N$ with pairwise
distinct, strictly increasing coordinates. But by the aforementioned
integral formula~\eqref{hciz}, such a statement would immediately imply
$G_\bu(\bm) \geqslant G_\bu(\bn)$ whenever $\bm$ weakly majorizes $\bn$
and both tuples have pairwise distinct coordinates, where:
\begin{equation}
G_\bu(\bm) \coloneqq \frac{\det \bu^{\circ \bm}}{V(\bm)} \cdot
\frac{V(\bn_{\min})}{V(\log[\bu])}.
\end{equation}
This would end the proof of the reverse implication, and with it, the
theorem.

It thus remains to show the assertion:
\[
\bm \succ_w \bn \quad \implies \quad F_\bu(\bm) \geqslant F_\bu(\bn)\
\forall \bu \in (1,+\infty)^N_{\neq}.
\]

\noindent For this we appeal to \cite[Chapter 3, C.2.d]{MOA} (which
follows from a result of Schur). This says that if $\phi$ is symmetric,
convex, and coordinatewise non-decreasing, then $\phi(\bm) \geqslant
\phi(\bn)$ whenever $\bm \succ_w \bn$. Defining $\phi(\bm) \coloneqq
F_\bu(\bm)$ for fixed $\bu$ as above, proving the three properties for
$\phi$ would conclude the proof.

From~\eqref{Eweakmaj}, $\phi$ is symmetric (since $U(N)$ contains the
permutation matrices~$S_N$). We next claim $\phi(\bm)$ is coordinatewise
non-decreasing in $\bm$. Indeed, by a limiting argument it suffices to
consider $\bm, \bn$ with pairwise distinct coordinates, whence $\phi(\bm)
= G_\bu(\bm)$ by the Harish-Chandra--Itzykson--Zuber
formula~\eqref{hciz}. But now if $\bm \geqslant \bn$ coordinatewise, then
by Corollary~\ref{Creal} (which we remark requires the positivity part of
the result of Lam et al~\cite{LPP}),
\[
\frac{\det \bu^{\circ \bm}}{\det \bu^{\circ \bn}} \geqslant \frac{\det
\bu(\epsilon)^{\circ \bm}}{\det \bu(\epsilon)^{\circ \bn}},
\]
with $\bu(\epsilon) \coloneqq (1,\epsilon, \dots, \epsilon^{N-1})^T$ as
in~\eqref{Evector}, and $\epsilon > 1$ chosen to be so small that
$\epsilon^{N-1} < u_j\ \forall j$. In other words,
\[
\frac{\det \bu^{\circ \bm}}{\det \bu(\epsilon)^{\circ \bm}} \geqslant
\frac{\det \bu^{\circ \bn}}{\det \bu(\epsilon)^{\circ \bn}}.
\]
Now taking $\epsilon \to 1^+$ and applying~\eqref{Eprincipal} proves the
claim for $\phi = G_\bu$, whence for $F_\bu$.

It remains to show that $F_\bu$ is convex, or by continuity, mid-convex.
But this is as in~\cite{Sra}: denoting
\[
D(\bm) \coloneqq {\rm diag}(m_0, \dots, m_{N-1}), \qquad
D_\bu \coloneqq {\rm diag}(\log(u_1), \dots, \log(u_N)),
\]
we compute:
\begin{align*}
&\ F_\bu \left( \frac{\bm + \bn}{2} \right)\\
= &\ \int_{U(N)} \exp \mathrm{tr} \left( D \left( \frac{\bm+\bn}{2}
\right) U D_\bu U^* \right)\ dU\\
= &\ \int_{U(N)} \left( \exp ( \mathrm{tr} (D(\bm) U D_\bu U^*) ) \cdot
\exp ( \mathrm{tr} (D(\bn) U D_\bu U^*) ) \right)^{1/2}\ dU\\
\leqslant &\ \frac{1}{2} \int_{U(N)} \exp \mathrm{tr} (D(\bm) U D_\bu
U^*)\ dU + \frac{1}{2} \int_{U(N)} \exp \mathrm{tr} (D(\bn) U D_\bu U^*)\
dU\\
= &\ \frac{F_\bu(\bm) + F_\bu(\bn)}{2},
\end{align*}
using the AM-GM inequality. From the above discussion, the proof is
complete.
\end{proof}

\begin{remark}
The key underlying fact used in deducing $\bm \succ_w \bn$
from \eqref{Ecgs-revised} is that if $t^\alpha \geqslant c$ for some
$c > 0$, $\alpha \in \R$, and all $t \geqslant 1$ (or a countable
sequence $t_l \to \infty$), then $\alpha$ must be non-negative. (In the
proof above, $\alpha = \widetilde{m}_j - \widetilde{n}_j$ and
$c = V(\bn_{\min}) \prod_{k=1}^N u_k^{n_{k-1} - m_{k-1}} / V(\bn)$.) If
it was possible to choose suitable $u_k \in I_\infty$ -- given any powers
$\bm, \bn$ -- such that $c \geqslant 1$, then our deduced condition
\[
t^\alpha \geqslant c,
\]
for a \textit{single} value of $t > 1$, would imply $\alpha \geqslant 0$.
In the absence of a ``uniform'' recipe to choose $u_k$ such that
$c \geqslant 1$, we require countably many $t_l$ to deduce the key
underlying fact mentioned above.
\end{remark}

As a consequence of Theorem~\ref{Tcgs-wmaj}, we immediately have:

\begin{corollary}\label{Ccgs}
Suppose $\bm, \bn$ are $N$-tuples of distinct real powers. The
inequality~\eqref{Ecgs-revised} holds for all $\bu \in ((0,1)^N)^T$ if
and only if $- \bm \succ_w -\bn$. In fact it suffices to work with
countably many vectors $\bu \in ((I_0)^N)^T$, where $I_0$ is an arbitrary
non-empty deleted neighborhood of $0$ in $(0,1]$.
\end{corollary}

This immediately follows from Theorem~\ref{Tcgs-wmaj} because
\[
\bu^{\circ \bm} = ((1/u_1, \dots, 1/u_N)^{\circ (-\bm)})^T, \qquad
\forall \bu = (u_1, \dots, u_N)^T \in ((0,1)^N)^T.
\]

We conclude this section by proving the twofold strengthening to the
original result of Cuttler--Greene--Skandera, stated in
Theorems~\ref{Tcgs-maj1} and~\ref{Tcgs-maj2}.

\begin{proof}[Proof of Theorems~\ref{Tcgs-maj1} and~\ref{Tcgs-maj2}]
That $(1) \implies (3)$ is immediate. To show $(3) \implies (2)$, first
consider the inequality~\eqref{Ecgs-revised} for $\bu = \bu_j(t_l)$ as
in~\eqref{Ewmaj}, and for $\bu_j(t_l)^{\circ (-1)}$, with $j \in [1,N]$
and $l \geqslant 1$. It follows from Theorem~\ref{Tcgs-wmaj} and
Corollary~\ref{Ccgs} that
\[
\bm \succ_w \bn, \qquad -\bm \succ_w -\bn,
\]
and these conditions are together equivalent to majorization: $\bm \succ
\bn$.

Finally, to show $(2) \implies (1)$, from the equivalence for
non-negative integer powers $\bm, \bn$ (proved by
Cuttler--Greene--Skandera, Sra, and Ait-Haddou--Mazure), one deduces the
result for non-negative rational powers $\bm, \bn$ by ``taking a common
denominator'' $L \in \N$ and changing variables to $y_j \coloneqq
u_j^{1/L}$ (see e.g.~the proof of Corollary~\ref{Creal}). The result for
non-negative real powers now follows by continuity. Finally, the result
for arbitrary real powers can be reduced to the case of non-negative
$\bm, \bn$ by multiplying both sides of the inequality by $(u_1 \cdots
u_N)^{-k}$, where $k \coloneqq \min \{ m_j, n_j : 0 \leqslant j < N \}$.
\end{proof}

\section{Further applications: Rayleigh quotients and the cube
problem}\label{Srayleigh}

We now discuss further remarks and applications of the above quantitative
results. To begin, our main quantitative result Theorem~\ref{Treal-rank2}
has an optimization-theoretic formulation, as follows. First notice that
the inequality~\eqref{tal} holds for all real powers $n_0 < \cdots <
n_{N-1} < M$, with both Schur polynomials $s_\bm(\bu)$ replaced by the
corresponding generalized Vandermonde determinants $\det \bu^{\circ
\bm}$, and with the same proof. Now as mentioned in
Remark~\ref{Rrayleigh}, using the theory of generalized Rayleigh
quotients, Theorem~\ref{Treal-rank2} implies the following result for a
\textit{single} positive semidefinite matrix:

\begin{proposition}\label{Prayleigh}
Fix an integer $N \geqslant 2$ and real powers $n_0 < \cdots < n_{N-1} <
M$, where $n_j \in \Z^{\geqslant 0} \cup [N-2,\infty)$. Given scalars
$c_{n_0}, \dots, c_{n_{N-1}} > 0$, define
\[
h(x) \coloneqq \sum_{j=0}^{N-1} c_{n_j} x^{n_j}, \qquad x \in
(0,+\infty).
\]
Then for $0 < \rho < +\infty$ and $A \in \bp_N([0,\rho])$,
\begin{equation}
 t \, h[A] \succeq A^{\circ M}, \qquad \text{if and only if} \quad t
 \geqslant \varrho( h[A]^{\dagger/2} A^{\circ M} h[A]^{\dagger/2} ),
\end{equation}

\noindent where $\varrho[B], B^\dagger$ denote the spectral radius and
the Moore--Penrose (pseudo-)inverse of a square matrix $B$, respectively.
Moreover, for every nonzero matrix $A \in \bp_N( [0, \rho] )$, we have
the variational formula
\[
\varrho( h[A]^{\dagger/2} A^{\circ M} h[A]^{\dagger/2} ) =
\sup_{\bu \in (\ker h[A])^\perp \setminus \{ {\bf 0} \}}
\frac{\bu^* A^{\circ M} \bu}{\bu^* h[\bu \bu^*] \bu} \leqslant
\sum_{j=0}^{N-1} \frac{V(\bn_j)^2}{V(\bn)^2} \frac{\rho^{M -
n_j}}{c_{n_j}}.
\]
\end{proposition}

Thus, for each  $A$ the threshold bound $\varrho( h[A]^{\dagger/2}
A^{\circ M} h[A]^{\dagger/2} )$ for the semidefinite inequality
\[
t \, h[A] \succeq A^{\circ M}
\]
is a generalized Rayleigh quotient. When the matrix $A$ is a generic
rank-one matrix, this Rayleigh quotient has a closed-form expression
using the Schur polynomials $s_{\bn_j}, s_\bn$ -- i.e., the generalized
Vandermonde determinants $\det \bu^{\circ \bn_j}, \det \bu^{\circ \bn}$:

\begin{proposition}
Notation as in Proposition~\ref{Prayleigh}; but now with $n_j$ not
necessarily in $\Z^{\geqslant 0} \cup [N-2,\infty)$. If $A = \bu \bu^T$
with $\bu \in ((0,+\infty)^N_{\neq})^T$ having distinct coordinates,
then $h[A]$ is invertible, and the threshold bound equals:
\begin{equation}\label{Erankone}
\varrho( h[A]^{\dagger/2} A^{\circ M} h[A]^{\dagger/2} )
= (\bu^{\circ M})^T h[ \bu \bu^T]^{-1} \bu^{\circ M}
= \sum_{j=0}^{N-1} \frac{(\det \bu^{\circ \bn_j})^2}{c_{n_j} (\det
\bu^{\circ \bn})^2}.
\end{equation}
\end{proposition}

In particular, by going from single rank-one matrices $A = \bu \bu^T$ for
generic vectors $\bu$, to all matrices in the linear matrix inequality
\[
t \, h[A] \succeq A^{\circ M} \quad \forall A \in \bp_N^1([0,\rho]),
\]
we recover by a second proof the inequality $\displaystyle t \geqslant
\sum_{j=0}^{N-1} \frac{V(\bn_j)^2}{V(\bn)^2} \frac{\rho^{M -
n_j}}{c_{n_j}}$, using Corollary~\ref{Creal}.

\begin{proof}
First observe that if $h[\bu \bu^T] v = 0$, then $v^T (\bu \bu^T)^{\circ
n_j} v = 0$ for all $j$, whence $v$ is killed by $\bu^{\circ n_j}$ for
all $j$. But then $v=0$ by~\eqref{Egantmacher}; thus $h[\bu \bu^T]$ is
nonsingular. Now we compute: 
\begin{align*}
\varrho( h[A]^{\dagger/2} A^{\circ M} h[A]^{\dagger/2} )
= &\ \varrho( h[\bu \bu^T]^{-1/2} \bu^{\circ M} \cdot (\bu^{\circ M})^T
h[\bu \bu^T]^{-1/2} )\\
= &\ (\bu^{\circ M})^T h[ \bu \bu^T]^{-1} \bu^{\circ M},
\end{align*}

\noindent since $\varrho( \bv \cdot \bv^T ) = \bv^T \bv$ for a real
vector $\bv$. As in the proof of Lemma~\ref{detform}, we have
\begin{align*}
&\ h[\bu \bu^T]^{-1}\\
= &\ ((\bu^{\circ n_0} | \dots | \bu^{\circ n_{N-1}})^{-1})^T {\rm diag}
(c_{n_0}^{-1}, \dots, c_{n_{N-1}}^{-1}) (\bu^{\circ n_0} | \dots |
\bu^{\circ n_{N-1}})^{-1}
\end{align*}

\noindent and using Cramer's rule, the vector $(\bu^{\circ n_0} | \dots |
\bu^{\circ n_{N-1}})^{-1} \bu^{\circ M}$ has $j$th coordinate
$(-1)^{N-j-1} \frac{\det \bu^{\circ \bn_j}}{\det \bu^{\circ \bn}}$. Hence
\eqref{Erankone} follows.
\end{proof}

\begin{remark}
Akin to Remark~\ref{Ralgebra}, the above argument shows the following
more general algebraic phenomenon. Suppose $\mathbb{F}$ is a field, $\bu,
\bv \in (\mathbb{F}^N_{\neq})^T$ each have pairwise distinct coordinates,
and $h(x) \coloneqq \sum_{j=0}^{N-1} c_{n_j} x^{n_j} \in \mathbb{F}[x]$
for integers $n_0 < \cdots < n_{N-1}$ and nonzero $c_{n_j} \in
\mathbb{F}^\times$. Then $h[\bu \bv^T] \coloneqq (h(u_j v_k))_{j,k=1}^N$
is invertible (where we also require $\bu, \bv$ to have nonzero
coordinates if $n_0 < 0$); and moreover for all integers $M > n_{N-1}$,
\[
(\bv^{\circ M})^T h[ \bu \bv^T ]^{-1} \bu^{\circ M} = \sum_{j=0}^{N-1}
\frac{\det \bu^{\circ \bn_j} \cdot \det \bv^{\circ \bn_j}}{c_{n_j} \cdot
\det \bu^{\circ \bn} \cdot \det \bv^{\circ \bn}}.
\]
\end{remark}

\begin{remark}
By Lemma~\ref{horn-type}, any linear matrix inequality of the form
\[
A^{\circ M} \preceq \mathcal{K} \sum_{j=0}^{N-1} c_{n_j} A^{\circ n_j},
\qquad \forall A \in \bp_N([0,\rho])
\]
does not hold for fewer than $N$ powers. Theorem~\ref{Treal-rank2} can be
reformulated to say such an inequality does not require more than $N$
powers either, and the constant $\mathcal{K} = \sum_{j=0}^{N-1}
\frac{V(\bn_j)^2}{V(\bn)^2} \frac{\rho^{M - n_j}}{c_{n_j}}$ is sharp.
\end{remark}

We end with an application to spectrahedra and the cube
problem~\cite{Nemirovski} in optimization theory, where we produce sharp
asymptotic bounds when the matrices involved are Hadamard powers. To
remind the reader, given $N \times N$ real symmetric matrices $A'_N, A_1,
\dots, A_{M+1}$ for $M \geqslant 0$, one defines the matrix cube to be:
\begin{equation}\label{Ecube}
\mathcal{U}[A'_N, A_1 \dots, A_{M+1}; \eta] \coloneqq \Bigl\{ A'_N +
\sum_{m = 1}^{M + 1} u_m A_m : u_m \in [ -\eta, \eta ] \Bigr\} \qquad (
\eta > 0 ).
\end{equation}
The matrix cube problem asks to find the largest $\eta \geqslant 0$ such
that $\mathcal{U}[ \eta ] \subset \bp_N(\R)$. As another consequence of
the main result, Theorem~\ref{Treal-rank2}, we obtain bounds for such
$\eta$:

\begin{corollary}\label{Ccube}
Let the notation as in Proposition~\ref{Prayleigh}.
Fix scalars $0 < \alpha_1 < \cdots < \alpha_{M+1}$ for some $M \geqslant
0$. Now given a matrix $A \in \bp_N( [ 0, \rho ] )$, define
\[
A'_N \coloneqq \sum_{j=0}^{N-1} c_{n_j} A^{\circ n_j}, \qquad A_m
\coloneqq A^{\circ (n_{N-1}+\alpha_m)} \ \textrm{for } 1 \leqslant m
\leqslant M+1.
\]
Also define for $\alpha > 0$:
\begin{equation}\label{Ecube1}
\mathcal{K}_\alpha \coloneqq \sum_{j=0}^{N-1}
\frac{V(\bn_j(\alpha))^2}{V(\bn)^2} \frac{\rho^{\alpha - n_j}}{c_{n_j}},
\end{equation}
where $\bn_j(\alpha) \coloneqq
(n_0,\dots,n_{j-1}, n_{j+1},\dots, n_{N-1}, n_{N-1} + \alpha)$. Then,
\begin{align}\label{Ecube2}
\begin{aligned}
 &\ \eta \leqslant (\mathcal{K}_{\alpha_1} + \cdots +
 \mathcal{K}_{\alpha_{M+1}})^{-1}\\
 \implies &\ \mathcal{U}[A'_N, A_1,
 \dots, A_{M+1}; \eta ] \subset \bp_N( \R ), \ \forall A \in
 \bp_N([0,\rho])\\
 \implies &\ \eta \leqslant \mathcal{K}_{\alpha_{M+1}}^{-1}.
\end{aligned}
\end{align}
\end{corollary}

We conclude this part by showing that these bounds are asymptotically
equal as $N \to \infty$, when the $n_j$ grow ``linearly'' at a bounded
rate. Notice that in such a setting, in light of~\cite{FitzHorn} we will
require all $n_j$ to be integers.

\begin{proposition}
Suppose $0 \leqslant n_0 < n_1 < \cdots$ are integers.
Fix scalars $0 < \alpha_1 < \cdots < \alpha_{M+1}$ for some $M \geqslant
0$, as well as $c_{n_j} > 0\ \forall j \geqslant 0$.
Given $N \in \N$ and a matrix $A \in \bp_N([0,\rho])$, define $A'_N,
A_m$, and $\mathcal{K}_\alpha = \mathcal{K}_\alpha(N)$ as in
Corollary~\ref{Ccube}.
If $\alpha_{M+1} - \alpha_M \geqslant n_{j+1} - n_j\ \forall j \geqslant
0$, then the lower and upper bounds for $\eta = \eta_N$ in~\eqref{Ecube2}
are asymptotically equal as $N \to \infty$, i.e.,
\begin{equation}\label{Easymptotic}
\lim_{N \to \infty} \mathcal{K}_{\alpha_{M+1}}(N)^{-1} \sum_{m = 1}^{M+1}
\mathcal{K}_{\alpha_m}(N) = 1.
\end{equation}
\end{proposition}

\begin{proof}
We compute for $1 \leqslant m < M+1$, first in the case $0 < j <
N-1$:
\begin{align*}
\frac{V(\bn_j(\alpha_m))}{V(\bn_j(\alpha_{M+1}))}
= &\ \prod_{k < N, k \neq j} \frac{n_{N-1} - n_k + \alpha_m}{n_{N-1} -
n_k + \alpha_{M+1}}\\
= &\ \frac{\alpha_m}{n_{N-1} - n_0 + \alpha_{M+1}} \cdot\\
& \quad \cdot
\frac{n_{N-1} - n_{j-1} + \alpha_m}{n_{N-1} - n_{j+1} + \alpha_{M+1}}
\prod_{0 \leqslant k < N-1, k \neq j-1,j}
\frac{n_{N-1} - n_k + \alpha_m}{n_{N-1} - n_{k+1} + \alpha_{M+1}}.
\end{align*}

By assumption on the $\alpha_m$, each factor in the ``final'' product
above is in~$(0,1]$. We now claim that the second fraction in the
preceding expression is at most~$2$. Indeed,
\begin{align*}
&\ 2(n_{N-1} - n_{j+1} + \alpha_{M+1}) - ( n_{N-1} - n_{j-1} + \alpha_m
)\\
= &\ (n_{N-1} - n_{j+1}) + \alpha_m + \left[ 2(\alpha_{M+1} - \alpha_m) -
( n_{j+1} - n_{j-1} ) \right] \geqslant \alpha_m > 0.
\end{align*}

\noindent It follows that
\begin{align*}
0 \leqslant \frac{V(\bn_j(\alpha_m))^2}{V(\bn_j(\alpha_{M+1}))^2}
\frac{\rho^{\alpha_m - n_j}}{\rho^{\alpha_{M+1} - n_j}} \leqslant &\
\frac{4 \alpha_m^2 \rho^{\alpha_m - \alpha_{M+1}} }{(N-1 +
\alpha_{M+1})^2}\\
\leqslant &\ \frac{4 \alpha_{M+1}^2 \max(1,\rho^{\alpha_1 -
\alpha_{M+1}})}{(N-2)^2}.
\end{align*}

Similarly in the case $j=0$, we have:
\begin{align*}
0 \leqslant \frac{V(\bn_0(\alpha_m))^2}{V(\bn_0(\alpha_{M+1}))^2}
\frac{\rho^{\alpha_m - n_0}}{\rho^{\alpha_{M+1} - n_0}} \leqslant &\
\frac{\alpha_m^2 \rho^{\alpha_m - \alpha_{M+1}} }{(n_{N-1} - n_1 +
\alpha_{M+1})^2}\\
\leqslant &\ \frac{4 \alpha_{M+1}^2 \max(1,\rho^{\alpha_1 -
\alpha_{M+1}})}{(N-2)^2}.
\end{align*}

\noindent Finally, if $j=N-1$, then by similar computations,
\[
0 \leqslant \frac{V(\bn_{N-1}(\alpha_m))}{V(\bn_{N-1}(\alpha_{M+1}))}
\leqslant \frac{\alpha_m + n_{N-1} - n_{N-2}}{n_{N-1} - n_0 +
\alpha_{M+1}},
\]
and the numerator on the right is at most $\alpha_m + (\alpha_{M+1} -
\alpha_m)$ by hypothesis. Therefore,
\begin{align*}
0 \leqslant \frac{V(\bn_{N-1}(\alpha_m))^2}{V(\bn_{N-1}(\alpha_{M+1}))^2}
\frac{\rho^{\alpha_m - n_0}}{\rho^{\alpha_{M+1} - n_0}} \leqslant &\
\frac{\alpha_{M+1}^2 \rho^{\alpha_m - \alpha_{M+1}} }{(n_{N-1} - n_0 +
\alpha_{M+1})^2}\\
\leqslant &\ \frac{4 \alpha_{M+1}^2 \max(1,\rho^{\alpha_1 -
\alpha_{M+1}})}{(N-2)^2}.
\end{align*}

From these computations it follows that if $1 \leqslant m < M+1$, then
for fixed~$N \geqslant 3$,
\begin{align*}
\mathcal{K}_{\alpha_m}(N) = &\ \sum_{j=0}^{N-1}
\frac{V(\bn_j(\alpha_m))^2}{V(\bn)^2} \frac{\rho^{\alpha_m -
n_j}}{c_{n_j}}\\
\leqslant &\ \frac{4 \alpha_{M+1}^2 \max(1,\rho^{\alpha_1 -
\alpha_{M+1}})}{(N-2)^2} \sum_{j=0}^{N-1}
\frac{V(\bn_j(\alpha_{M+1}))^2}{V(\bn)^2} \frac{\rho^{\alpha_{M+1} -
n_j}}{c_{n_j}}\\
\leqslant &\ \frac{4 \alpha_{M+1}^2 \max(1,\rho^{\alpha_1 -
\alpha_{M+1}})}{(N-2)^2} \mathcal{K}_{\alpha_{M+1}}(N)
\end{align*}
Summing this over $m = 1,\dots,M$ and taking $N \to \infty$
yields~\eqref{Easymptotic}, as desired.
\end{proof}

\section{Log-supermodularity of strictly totally positive matrix
minors}\label{logmod}

Given Proposition~\ref{Pschur-ratio} and its generalization to real
powers in Corollary~\ref{Creal}, it is natural to ask if there is a more
``accessible'' proof of the positivity property in~\eqref{Elpp}, avoiding
the heavy machinery required to prove Schur or monomial positivity. In
this final section, we provide a positive answer to this question. In
fact we prove a more general log-supermodularity phenomenon which had
previously been established by Skandera~\cite{skandera} by a different
argument.

We begin by setting notation and recalling preliminaries. Suppose $1
\leqslant k \leqslant n$ are natural numbers. Recall that $\nkneq$ is the
set of all tuples $(i_1,\dots,i_k)$ with $i_1,\dots,i_k$ distinct
elements of $[n] \coloneqq \{1,\dots,n\}$. We define $[n]^k_{<}$ to be
the subset of increasing tuples:
\begin{equation}
[n]^k_{<} \coloneqq \{ (i_1,\dots,i_k) \in \nkneq : 1 \leqslant i_1 <
\dots < i_k \leqslant n \}.
\end{equation}
Clearly, for each tuple $I = (i_1,\dots,i_k)$ in $\nkneq$, we may sort
the tuple in increasing order to obtain a sorted tuple $\mathrm{sort}(I)
= (i_{\sigma(1)},\dots,i_{\sigma(k)}) \in [n]^k_<$ for a uniquely
determined permutation $\sigma: [k] \to [k]$, which we will call the
\emph{ordering} of $I$.  The sign $\mathrm{sgn}(\sigma) \in \{-1,+1\}$ of
this permutation will be referred to as the \emph{sign of the tuple} and
denoted $\mathrm{sgn}(I)$.  Observe that $\mathrm{sgn}(I)$ is equal to
$-1$ raised to the number of inversions of $I$, that is to say those
pairs of natural numbers $1 \leqslant m < m' \leqslant k$ with $i_{m'} <
i_m$.

If $A = (a_{i,j})_{1 \leqslant i,j \leqslant n}$ is an $n \times n$
matrix with real entries $a_{i,j}$, and $I = (i_1,\dots,i_k)$ and $J =
(j_1,\dots,j_k)$ are tuples in $\nkneq$, we define the \emph{submatrix}
$A_{I,J}$ to be the $k \times k$ matrix
\[
A_{I,J} \coloneqq (a_{i_l,j_m})_{1 \leqslant l,m \leqslant k}.
\]
Note that we are generalizing slightly the usual notion of a submatrix by
allowing~$I,J$ to have non-trivial ordering, so the ordering of the rows
and columns of the submatrix $A_{I,J}$ may differ from that of the
original matrix $A$.  We say that $A$ is \emph{strictly totally positive}
if one has
\[
\det(A_{I,J}) > 0
\]
for any increasing $I,J \in [n]^k_<$, i.e.~by the
alternating nature of the determinant,
\begin{equation}\label{sgnsgn}
\mathrm{sgn}(I) \mathrm{sgn}(J) \det(A_{I,J}) > 0,
\quad \forall I,J \in \nkneq.
\end{equation}

Given two tuples $I = (i_1,\dots,i_k)$, $J = (j_1,\dots,j_k)$
of real numbers, we may define their meet
\[
I \wedge J \coloneqq (\min(i_1,j_1),\dots,\min(i_k,j_k))
\]
and join
\[
I \vee J \coloneqq (\max(i_1,j_1),\dots,\max(i_k,j_k)).
\]
It is easy to see that if $I,J \in \nkneq$ have the same ordering, then
$I \wedge J$ and~$I \vee J$ are also tuples in $\nkneq$ with the same
ordering as $I$ and $J$ (this can be verified by first sorting the
elements to reduce to the case when $I,J$ are increasing.)  The purpose
of this section is to establish the following log-supermodularity
property (also known as ``the FKG lattice condition''~\cite{fkg} or
``multivariate total positivity of order two'' MTP$_2$~\cite{kr}):

\begin{theorem}[Log-supermodularity]\label{logsup}
Let $A$ be a strictly totally positive $n \times n$ matrix, let $I_1,I_2
\in \nkneq$ be tuples of the same ordering, and let $J_1,J_2 \in \nkneq$
be tuples of the same ordering.  Then
\begin{equation}\label{ij2}
\det(A_{I_1 \wedge I_2, J_1 \wedge J_2}) \det(A_{I_1 \vee I_2, J_1 \vee
J_2}) \geqslant \det(A_{I_1,J_1}) \det(A_{I_2,J_2}).
\end{equation}
\end{theorem}

A special case of this proposition (in which $J_1=J_2$) appears
implicitly in~\cite{richards}, and our arguments here are loosely based
on the ones in that paper.  As pointed out to us by Steven Karp, this
result is also a special case of \cite[Theorem 4.2]{skandera}, which was
proven by a different method (using weighted directed graphs instead of
the determinant identities in the lemma below).

We will need two classical identities concerning determinants:

\begin{lemma}[Determinant identities]\label{det-ident}
Let $n \geqslant 2$ be a natural number.
\begin{itemize}
\item[(i)] (Dodgson condensation,
see~\cite{Dodgson}, \cite{BS}.)
For any $1 \leqslant i_1 < i_2 \leqslant n$ and $1 \leqslant j_1 < j_2
\leqslant n$, and any $n \times n$ matrix $A$, one has
\begin{align*}
\det(A) & \det(A_{[n] \backslash (i_1,i_2), [n] \backslash (j_1,j_2)})\\
= &\ \det(A_{[n] \backslash (i_1), [n] \backslash (j_1)}) \det(A_{[n]
\backslash (i_2), [n] \backslash (j_2)})\\
&\ - \det(A_{[n] \backslash (i_1), [n] \backslash (j_2)}) \det(A_{[n]
\backslash (i_2), [n] \backslash (j_1)}),
\end{align*}
where $[n] \backslash (i_1,\dots,i_k)$ denotes the increasing $n-k$-tuple
formed by deleting $i_1,\dots,i_k$ from $(1,\dots,n)$.

\item[(ii)] (Karlin identity, 
see \cite[p.7]{karlin}.)
If $X_1,X_2,Y_1,Y_2 \in (\R^n)^T$ are $n$-dimensional vectors, and $A$ is
an $n \times n-2$-matrix, then one has
\begin{align*}
&\ \det( X_1, Y_1, A) \det(X_2, Y_2, A) \; - \det(X_1, Y_2, A) \det(X_2,
Y_1, A)\\
= &\ \det(X_1,X_2,A) \det(Y_1,Y_2,A)
\end{align*}
where $\det(X,Y,A)$ denotes the determinant of the $n \times n$ matrix
formed by concatenating two $n \times 1$ vectors $X,Y$ with an $n
\times n-2$ matrix $A$.
\end{itemize}
\end{lemma}

\begin{proof}[Proof of Theorem~\ref{logsup}]
By sorting, we may assume without loss of generality that
$I_1,I_2,J_1,J_2$ are all increasing, so that all the determinants
involved are positive.  Writing
\[
F(I,J) \coloneqq \log \det(A_{I,J})
\]
for $I,J \in [n]^k_<$, our task is now to show that
\[
F(K_1 \wedge K_2) + F(K_1 \vee K_2) \geqslant F(K_1) + F(K_2)
\]
for all $K_1,K_2 \in [n]^k_< \times [n]^k_<$, where we write
\[
(I_1,J_1) \wedge (I_2,J_2) \coloneqq (I_1 \wedge I_2, J_1 \wedge J_2)
\]
and
\[
(I_1,J_1) \vee (I_2,J_2) \coloneqq (I_1 \vee I_2, J_1 \vee J_2).
\]
Write $K_0 \coloneqq K_1 \wedge K_2$, $\delta K_1 \coloneqq K_1 - K_0$,
and $\delta K_2 \coloneqq K_2 - K_0$ (where we view the $2k$-tuples $K_1,
K_2$ as lying in the vector space $\R^{2k}$).  Then $\delta K_1, \delta
K_2$ have non-negative coordinates and disjoint supports (that is to say,
the set of indices in which $\delta K_1$ has non-zero coordinates is
disjoint from that of $\delta K_2$), and our task is now to show that
\begin{equation}\label{tar}
 F( K_0 + \delta K_1 + \delta K_2 ) + F(K_0) \geqslant F(K_0 + \delta
 K_1) + F(K_0 + \delta K_2)
\end{equation}
whenever $K_0, K_0 + \delta K_1, K_0 + \delta K_2, K_0 + \delta K_1 +
\delta K_2$ lie in $[n]^k_< \times [n]^k_<$ with $\delta K_1, \delta K_2$
having non-negative coordinates and disjoint supports.
 
We can rewrite~\eqref{tar} as the assertion the quantity $F(K_0+\delta
K_1) - F(K_0)$ does not decrease if we increase $K_0$ by $\delta K_2$.
Writing $\delta K_2 = \sum_{m=1}^{2k} c_m e_m$ for some non-negative
$c_m$, where $e_1,\dots,e_{2k}$ is the standard basis of $\R^{2k}$, it
suffices to show that for each $1 < M \leqslant 2k$, the quantity
\[
F(K_0+\sum_{m=M+1}^{2k} c_m e_m+\delta K_1) - F(K_0+\sum_{m=M+1}^{2k} c_m
e_m)
\]
does not decrease if one increases $K_0+\sum_{m=M+1}^{2k} c_m e_m$ by
$c_M e_M$.  Note that as $K_0$ and $K_0 + \sum_{m=1}^{2k} c_m e_m$ both
lie in $[n]^k_< \times [n]^k_<$, the intermediate tuples
$K_0+\sum_{m=M+1}^{2k} c_m e_m$ and $K_0+\sum_{m=M}^{2k} c_m e_m$ do also
(here it is important we are summing the $c_m e_m$ from the right end of
the range $m=1,\dots,2k$, rather than from the left).  Similarly for
$K_0+\sum_{m=M+1}^{2k} c_m e_m + \delta K_1$ and $K_0+\sum_{m=M}^{2k} c_m
e_m + \delta K_1$.  Because of this, we see that to prove~\eqref{tar} it
suffices to do so in the special case when $\delta K_2$ is supported on a
single element, thus $\delta K_2 = c_2 e_{m_2}$ for some $m_2 =
1,\dots,2k$ and some positive $c_2$.  Similarly we may also assume
without loss of generality that $\delta K_1 = c_1 e_{m_1}$ for some
integer $m_1 \in [1,2k] \setminus \{ m_2 \}$ and some positive $c_1$.

Splitting into cases depending on whether $m_1$ and $m_2$
lie in $\{1,\dots,k\}$ or in $\{k+1,\dots,2k\}$, we can thus
reduce~\eqref{tar} to the verification of three special cases:
\begin{itemize}
\item[(a)]  One has
\[
F( I_0 + c_1 e_{m_1} + c_2 e_{m_2}, J_0) + F(I_0,J_0) \geqslant F( I_0 +
c_1 e_{m_1}, J_0) + F(I_0 + c_2 e_{m_2},J_0)
\]
whenever $1 \leqslant m_1 < m_2 \leqslant k$, $c_1,c_2$ are positive, and
$I_0, I_0 + c_1 e_{m_1}, I_0 + c_2 e_{m_2}, I_0 + c_1 e_{m_1} + c_2
e_{m_2}, J_0$ all lie in $[n]^k_<$.
\item[(b)]  One has
\[
F( I_0, J_0 + c_1 e_{m_1} + c_2 e_{m_2}) + F(I_0,J_0) \geqslant F( I_0,
J_0 + c_1 e_{m_1}) + F(I_0, J_0 + c_2 e_{m_2})
\]
whenever $1 \leqslant m_1 < m_2 \leqslant k$, $c_1,c_2$ are positive, and
$I_0, J_0, J_0 + c_1 e_{m_1}, J_0 + c_2 e_{m_2}, J_0 + c_1 e_{m_1} + c_2
e_{m_2}$ all lie in $[n]^k_<$.
\item[(c)]  One has
\[
F( I_0 + c_1 e_{m_1}, J_0 + c_2 e_{m_2}) + F(I_0,J_0) \geqslant F( I_0 +
c_1 e_{m_1}, J_0 ) + F(I_0, J_0 + c_2 e_{m_2})
\]
whenever $1 \leqslant m_1, m_2 \leqslant k$, $c_1,c_2$ are positive, and
$I_0, I_0 + c_1 e_{m_1}, J_0$, $J_0 + c_2 e_{m_2}$ all lie in $[n]^k_<$.
\end{itemize}

We first prove (a).  By exponentiating, this claim is equivalent to the
assertion that
\[
\det(A_{I_0 + c_1 e_{m_1} + c_2 e_{m_2}, J_0}) \det(A_{I_0,J_0}) -
\det(A_{I_0 + c_1 e_{m_1},J_0}) \det(A_{I_0 + c_2 e_{m_2},J_0}) \geqslant
0.
\]
By permuting $I_0$, it will suffice to show that
\begin{equation}\label{detor}
 \det(A_{I_0 + c_1 e_1 + c_2 e_2, J_0}) \det(A_{I_0,J_0}) - \det(A_{I_0 +
 c_1 e_1,J_0}) \det(A_{I_0 + c_2 e_2,J_0}) \geqslant 0
\end{equation}
whenever $c_1,c_2$ are positive, $J_0 \in [n]^k_<$, and $I_0, I_0 + c_1
e_1, I_0 + c_2 e_2 \in \nkneq$ all have the same ordering.

Write $I_0 = (i_1,i_2,i_3,\dots,i_k)$; thus,
\begin{align*}
I_0 + c_1 e_1 = &\ (i_1 + c_1, i_2,i_3,\dots,i_k),\\
I_0 + c_2 e_2 = &\ (i_1, i_2+c_2,i_3,\dots,i_k),\\
I_0 + c_1 e_1 + c_2 e_2 = &\ (i_1+c_1,i_2+c_2,i_3,\dots,i_k).
\end{align*}

\noindent Applying the Karlin identity (Lemma~\ref{det-ident}(ii)), the
left-hand side of~\eqref{detor} may be factorized as
\[
\det( A_{(i_1, i_1+c_1, i_3, \dots, i_k), J_0} ) \det( A_{(i_2, i_2+c_2,
i_3, \dots, i_k), J_0} ).
\]
Since the tuples $I_0, I_0 + c_1 e_1, I_0 + c_2 e_2, I_0 + c_1 e_1 + c_2
e_2$ all have the same ordering, we see that the tuples $(i_1, i_1+c_1,
i_3, \dots, i_k)$ and $(i_2, i_2+c_2, i_3, \dots, i_k)$ have the same
parity of inversion counts, and thus the same sign.  The claim now
follows from~\eqref{sgnsgn}.

The proof of (b) is identical to (a) (and indeed follows from (a) after
replacing $A$ with $A^T$), so we turn to (c).  By exponentiating and permuting as before, it will suffice to show that
\begin{equation}\label{detor-2}
 \det(A_{I_0 + c e_1, J_0 + d e_1}) \det(A_{I_0,J_0}) - \det(A_{I_0 + c
 e_1,J_0}) \det(A_{I_0, J_0 + d e_1}) \geqslant 0
\end{equation}
whenever $c,d$ are positive, $I_0, I_0 + c e_1 \in \nkneq$ have the same
ordering, and $J_0, J_0 + d e_1 \in \nkneq$ have the same ordering.

Write $I_0 = (i_1,i_2,\dots,i_k)$ and $J_0 = (j_1,j_2,\dots,j_k)$. 
Applying Dodgson condensation (Lemma~\ref{det-ident}(i)), the left-hand
side of~\eqref{detor-2} may be written as
\[
\det(A_{(i_1,i_1+c,i_2,\dots,i_k), (j_1,j_1+d,j_2,\dots,j_k)})
\det(A_{(i_2,\dots,i_k), (j_2,\dots,j_k)}).
\]
Since $I_0, I_0+ce_1$ have the same ordering, we see that
$(i_1,i_1+c,i_2,\dots,i_k)$ and $(i_2,\dots,i_k)$ have the same parity of
inversion counts, and thus the same sign.  Similarly for
$(j_1,j_1+d,j_2,\dots,j_k)$ and $(j_2,\dots,j_k)$.  The claim now follows
from~\eqref{sgnsgn}.
\end{proof}

\begin{remark}
The above argument in fact shows that the inequality in \eqref{ij2} is
strict whenever the tuples $(I_1,J_1)$ and $(I_2,J_2)$ are incomparable
(thus $(I_1 \vee I_2, J_1 \vee J_2)$ is not equal to either $(I_1,J_1)$
or $(I_2,J_2)$).  Of course,~\eqref{ij2} is satisfied with equality if
$(I_1,J_1)$ and $(I_2,J_2)$ are comparable.
\end{remark}

\subsection{Applications}

From the positivity of generalized Vandermonde determinants, we see that
any generalized Vandermonde matrix $(u_j^{n_{k-1}})_{1
\leqslant j,k \leqslant N}$ with $0 < u_1 < \dots < u_N$ and $0 < n_0 <
\dots < n_{N-1}$ is strictly totally positive. Writing
\[
\det (u_j^{n_{k-1}})_{1 \leqslant j,k \leqslant N} =
V(\mathbf{u}) s_{\mathbf n}(\mathbf{u})
\]
we arrive at the log-supermodularity inequality
\begin{align*}
&\ V(\mathbf{u}_1 \vee \mathbf{u}_2) s_{\mathbf{n}_1 \vee
\mathbf{n}_2}(\mathbf{u}_1 \vee \mathbf{u}_2) V(\mathbf{u}_1 \wedge
\mathbf{u}_2) s_{\mathbf{n}_1 \wedge \mathbf{n}_2}(\mathbf{u}_1 \wedge
\mathbf{u}_2)\\
\geqslant &\
V(\mathbf{u}_1) s_{\mathbf{n}_1}(\mathbf{u}_1 )
V(\mathbf{u}_2) s_{\mathbf{n}_2}(\mathbf{u}_2)
\end{align*}

\noindent whenever $\mathbf{u}_1^T, \mathbf{u}_2^T, \mathbf{n}_1,
\mathbf{n}_2$ are increasing tuples with positive coordinates.  In
particular, we have
\[
s_{\mathbf{n}_1}(\mathbf{u}_1) s_{\mathbf{n}_2}(\mathbf{u}_2)
\geqslant s_{\mathbf{n}_1}(\mathbf{u}_2) s_{\mathbf{n}_2}(\mathbf{u}_1)
\]
whenever $\mathbf{u}_1^T, \mathbf{u}_2^T, \mathbf{n}_1, \mathbf{n}_2$ are
increasing tuples with positive coordinates with $\mathbf{n}_1 \geqslant
\mathbf{n}_2$ and $\mathbf{u}_1 \geqslant \mathbf{u}_2 \geqslant {\bf 0}$
in the product order.  Equivalently, the ratio
\[
s_{\mathbf{n}_1}(\mathbf{u})/s_{\mathbf{n}_2}(\mathbf{u})
\coloneqq \det \bu^{\circ \bn_1} / \det \bu^{\circ \bn_2}
\]
is non-decreasing in the vector $\mathbf{u}$ coordinatewise (where all
tuples are restricted to be positive and increasing). Note that this
argument does not require $\mathbf{n}_1, \mathbf{n}_2$ to have integer
coordinates.

If one uses the Jacobi--Trudi identity, one similarly concludes the
pointwise inequality
\begin{equation}\label{Epositive}
\tilde s_{\lambda_1 \vee \lambda_2 / \mu_1 \vee
\mu_2} \tilde s_{\lambda_1 \wedge \lambda_2 / \mu_1 \wedge \mu_2}
\geqslant \tilde s_{\lambda_1 / \mu_1} \tilde s_{\lambda_2 / \mu_2}
\end{equation}

\noindent between products of skew-Schur polynomials, whenever
$\lambda_1, \lambda_2, \mu_1, \mu_2$
are nonincreasing tuples of natural or even real numbers, with
non-negative coordinates. (We remark here that these ``continuous''
versions of Schur and skew-Schur polynomials can be defined for all
non-negative real tuples of powers, in a manner generalizing (usual)
Schur and skew-Schur polynomials for integer powers. See
\href{https://terrytao.wordpress.com/2017/09/05/}{\tt
https://terrytao.wordpress.com/2017/09/05/} \thinspace for more details.)
This is a weaker form of the Schur-positivity of~\eqref{Elpp} by Lam,
Postnikov, and Pylyavskyy~\cite{LPP} that was previously used to prove
Theorem~\ref{Tschur-ratio}.  We also remark that in the special case
of~\eqref{Epositive} when $\mu_1 = \mu_2 = {\bf 0}$, this inequality was
also established by Richards~\cite{richards}, using essentially the same
techniques as those given here.

\begin{remark}
The above discussion shows that the positivity in~\eqref{Epositive}
implies part of the ``extended'' Cuttler--Greene--Skandera
conjecture~\eqref{Ecgs-revised} as a special case, via
Proposition~\ref{Pschur-ratio} (as also the positivity but not monomial-
or Schur-positivity in the inequality~\eqref{Elpp} by
Lam--Postnikov--Pylyavskyy). It would be interesting to explore if the
result~\eqref{Epositive} -- and the techniques in this paper -- can be
used to prove other positivity inequalities involving symmetric
functions.
\end{remark}


\section*{List of notation and symbols}

For the convenience of the reader, a tabulation of the symbols
used in this paper is provided. The following symbols involve matrices
and operations on them. In what follows, $N>0$ is a fixed integer. 
\begin{itemize}
\item \underline{$\bp_N(I)$}:
Given a subset $I \subset \C$, the set $\bp_N( I )$ is the collection of
positive semidefinite $N \times N$ matrices with entries in a subset $I
\subset \C$.\smallskip

\item \underline{$\bp_N^1(I)$}
is the subset of $\bp_N(I)$ consisting of all rank-one
matrices.\smallskip

\item \underline{$\HTN_N(I)$}
is the subset of $\bp_N(I)$ comprising all Hankel totally non-negative
matrices.\smallskip

\item \underline{${\bf 1}_{N \times N}$}
is the $N \times N$ matrix with each entry equal to $1$.\smallskip

\item \underline{$f[ A ]$}
is the result of applying $f$ to each entry of the matrix $A$.\smallskip

\item \underline{$A^{\circ \alpha}$}
is the entrywise $\alpha$th power of the matrix $A$ (usually with
positive entries if $\alpha$ is not an integer).\smallskip

\item \underline{$A^\dagger$}
is the Moore--Penrose (pseudo) inverse of the matrix $A$.\smallskip

\item \underline{$\varrho( A )$}
is the spectral radius of the matrix $A$.\smallskip

\item \underline{$\succeq$}:
Given matrices $A,B$, we say $A \preceq B$, or $B \succeq A$, if $B - A$
is positive semidefinite.
\end{itemize}

\noindent The second set of symbols concerns vectors, tuples, and
operations on them.
\begin{itemize}
\item \underline{$\bu^\bn, \bu^{\circ n}, \bu^{\circ \bn}$}:
Given a vector $\bu = (u_1, \dots, u_N)^T$ and a tuple $\bn = (n_0,
\dots, n_{N-1})$, with either complex $\bu$ and integer $\bn$, or
positive $\bu$ and real $\bn$, we define
\[
\bu^\bn \coloneqq u_1^{n_0} \cdots u_N^{n_{N-1}}, \quad
\bu^{\circ n_0} \coloneqq (u_1^{n_0}, \dots, u_N^{n_0})^T, \quad
\bu^{\circ \bn} \coloneqq ( \bu^{\circ n_0} | \dots | \bu^{\circ
n_{N-1}}).
\]

\item \underline{$V(\bu)$} for a vector $\bu = (u_1, \dots, u_N)^T$ or a
tuple $(u_1, \dots, u_N)$ is the Vandermonde determinant $V(\bu) =
V(u_1,\dots,u_N) = \prod_{1 \leqslant i < j \leqslant N} (u_j -
u_i)$.\smallskip

\item \underline{$\bu(\epsilon)$} for a scalar $\epsilon$ denotes the
vector $(1, \epsilon, \dots, \epsilon^{N-1})^T$.\smallskip

\item \underline{$\bn_{\min}$} is the integer tuple
$(0,1,\dots,N-1)$.\smallskip

\item \underline{$s_\bn(\bu)$}:
Given fixed integers $N,m>0$, an $N$-tuple of integers $0 \leqslant n_0 <
n_1 < \cdots < n_{N-1}$, and a vector $\bu = (u_1, \dots, u_m)^T$ or a
tuple $(u_1, \dots, u_m)$, denote by $s_\bn(\bu) = s_\bn(u_1,\dots,u_m)$
the corresponding Schur polynomial -- see equation~\eqref{sdef}. (In this
paper, unless otherwise specified we set $m=N$.)\smallskip

\item \underline{$\overline{\bu}$}:
Given a tuple $\bu = (u_1, \dots, u_N)$, define $\overline{\bu} \coloneqq
(u_N, \dots, u_1)$.\smallskip

\item \underline{$\tilde{s}_\lambda(\bu)$}:
Given a partition, i.e.~a non-increasing tuple of integers $\lambda_{N-1}
\geqslant \cdots \geqslant \lambda_0 \geqslant 0$, the corresponding
Schur polynomial is
\[
\tilde{s}_\lambda(u_1, \dots, u_m) = s_{\overline{\lambda} +
\bn_{\min}}(u_1, \dots, u_m), \qquad \text{where } \lambda \coloneqq
(\lambda_{N-1}, \dots, \lambda_0).
\]

\item \underline{$\tilde{s}_{\lambda/\mu}(\bu)$}:
Given partitions $\lambda, \mu$, the corresponding skew-Schur polynomial
is denoted by $\tilde{s}_{\lambda/\mu}(\bu)$ if $\lambda$ contains $\mu$
(i.e., $\lambda_j \geqslant \mu_j\ \forall j$), and zero
otherwise.\smallskip

\item \underline{$\lambda \wedge \mu, \lambda \vee \mu$}:
Given $N$-tuples of real numbers $\lambda = (\lambda_0 \dots,
\lambda_{N-1})$ and $\mu = (\mu_0, \dots, \mu_{N-1})$, define their meet
and join to be, respectively,
\begin{align*}
\lambda \wedge \mu \coloneqq &\ ( \min(\lambda_0, \mu_0), \dots,
\min(\lambda_{N-1}, \mu_{N-1})),\\
\lambda \vee \mu \coloneqq &\ ( \max(\lambda_0, \mu_0), \dots,
\max(\lambda_{N-1}, \mu_{N-1})).
\end{align*}

\item \underline{$S^N_{\neq}, S^N_<$}: Given a set $S$, denote by
$S^N_{\neq}$ the collection of all ordered $N$-tuples in $S$ with
pairwise distinct entries. If moreover $S \subset \R$, then denote by
$S^N_<$ the subset of $S^N_{\neq}$ with strictly increasing
entries.\smallskip

\item {\underline{$\succ_w$}:
Given real $N$-tuples $\bu = (u_1, \dots, u_N)$ and $\bv = (v_1, \dots,
v_N)$, let $u_{[1]} \geqslant \cdots \geqslant u_{[N]}$ and $v_{[1]}
\geqslant \cdots \geqslant v_{[N]}$ denote their decreasing
rearrangements. Now $\bu$ weakly majorizes $\bv$ -- denoted $\bu \succ_w
\bv$ -- if for all $k=1,\dots,N$, one has: $u_{[1]} + \cdots + u_{[k]}
\geqslant v_{[1]} + \cdots + v_{[k]}$. See equation~\eqref{Eweak}.}
\end{itemize}




\end{document}